\documentclass{article}
\usepackage[utf8]{inputenc}

\usepackage{amsmath,amsthm,amsfonts,color}
\usepackage{amssymb}
\usepackage{graphicx}
\usepackage{setspace}
\usepackage{mathrsfs}
\usepackage{graphicx}
\usepackage[shortlabels]{enumitem}
\usepackage{authblk}

\usepackage{calc}
\usepackage{tikz,caption}
\usetikzlibrary{calc} 
\usepackage{geometry}
\usepackage{pgffor}
\usepackage{ifthen}
\usetikzlibrary{arrows.meta}
\usetikzlibrary{decorations.markings}
\tikzstyle{vertex}=[circle, draw, inner sep=1pt, minimum size=3pt]
\tikzstyle{filledvertex}=[circle, draw, fill, inner sep=1pt, minimum size=3pt]

\newtheorem{theorem}{Theorem}

\newtheorem{lemma}{Lemma}

\newtheorem{case}{Case}

\newtheorem{subcase}{Case}[case]
\newtheorem{subsubcase}{Case}[subcase]
\newtheorem{claim}{Claim}
\newtheorem{subclaim}{Claim}[claim]

\tikzstyle{vertex}=[circle, draw, inner sep=2pt, minimum size=6pt]
\tikzstyle{filledvertex}=[circle, draw, fill, inner sep=2pt, minimum size=6pt]
\tikzstyle{directed}=[postaction={decorate,
    decoration={markings,mark=at position 0.5 with {\arrow{stealth}}}
}]

\usepackage{geometry}
\begin{document}
\begin{spacing}{1.30}

\title{On graphs without cycles of length 1 modulo 3
\thanks{Supported in part by National Natural Science Foundation of China (Grant Nos. 12311540140, 12242111, 12131013, 12171393, 11601430), Guangdong Basic \& Applied Basic Research Foundation (Grant Nos. 2023A1515030208, 2022A1515010899), Shaanxi Fundamental Science Research Project for Mathematics and Physics (Grant No. 22JSZ009).}}

\author[$a$,$b$,$c$]{Yandong Bai
\thanks{E-mail: bai@nwpu.edu.cn.}}
\author[$a$,$b$]{Binlong Li
\thanks{E-mail: binlongli@nwpu.edu.cn.}}
\author[$a$,$b$]{Yufeng Pan
\thanks{E-mail: yf.pan@mail.nwpu.edu.cn.}}
\author[$a$,$b$]{Shenggui Zhang
\thanks{E-mail: sgzhang@nwpu.edu.cn.}}

\affil[$a$]{\small School of Mathematics and Statistics, 
Northwestern Polytechnical University,
Xi'an, Shaanxi 710129, China}
\affil[$b$]{\small Xi’an-Budapest Joint Research Center for Combinatorics,
Northwestern Polytechnical University,
Xi'an, Shaanxi 710129, China}
\affil[$c$]{\small Research $\&$ Development Institute of 
Northwestern Polytechnical University in Shenzhen,
Shenzhen, Guangdong 518057, China}

\date{}
\maketitle

\begin{abstract}
Burr and Erd\H{o}s conjectured in 1976 that 
for every two integers $k>\ell\geqslant 0$ 
satisfying that $k\mathbb{Z}+\ell$ contains an even integer, 
an $n$-vertex graph containing no cycles of length $\ell$ modulo $k$ 
can contain at most a linear number of edges on $n$. 
Bollob\'{a}s confirmed this conjecture in 1977 and then Erd\H{o}s proposed the problem of determining the exact value of the maximum number of edges in such a graph.
For the above $k$ and $\ell$,
define $c_{\ell,k}$ to be the least constant such that
every $n$-vertex graph with at least $c_{\ell,k}\cdot n$ edges contains a cycle of length $\ell$ modulo $k$. 
The precise (or asymptotic) values of $c_{\ell,k}$ are known for very few pairs $\ell$ and $k$. 
In this paper,
we precisely determine the maximum number of edges in a graph containing no cycles of length 1 modulo 3.
In particular,
we show that every $n$-vertex graph with at least $\frac{5}{3}(n-1)$ edges contains a cycle of length 1 modulo 3, unless $9|(n-1)$ and each block of the graph is a Petersen graph.
As a corollary, 
we obtain that $c_{1,3}=\frac{5}{3}$.
This is the last remaining class modulo $k$ for $1\leqslant k\leqslant 4$.
\end{abstract}

\section{Introduction}

In this paper,
we only consider simple graphs, i.e., without loops or multiple edges.
For terminology and notations not defined here,
we refer the reader to \cite{BM08}.

For two integers $k>\ell\geqslant 0$ satisfying that
$k\mathbb{Z}+\ell$ contains an even integer,
what is the maximum number of edges in an $n$-vertex graph containing no cycles of length $\ell$ modulo $k$?
Burr and Erd\H{o}s \cite{Erd76} conjectured in 1976 that such a graph can contain at most a linear number of edges on $n$. 
Bollob\'{a}s \cite{B77} gave a positive answer to this conjecture.
Erd\H{o}s then asked what is the exact value of 
the maximum number of edges in such a graph.

Let $G$ be an $n$-vertex graph 
and denote its edge number by $e(G)$.
By an ($\ell$ mod $k$)-cycle we mean a cycle of length $\ell$ modulo $k$.
Denote by $\mathcal{C}_{\ell\bmod k}$ the set of ($\ell$ mod $k$)-cycles.
If $G$ contains no cycles, or alternatively, no (0 mod 1)-cycles,
then one can see that $e(G)\leqslant n-1$ 
and the equality holds if and only if $G$ is a tree.
A famous result states that
if $G$ contains no (0 mod 2)-cycles
then $e(G)\leqslant \frac{3}{2}(n-1)$
and, for $2|(n-1)$, 
the equality holds if and only if each block of $G$ is a triangle.
In view of the complete bipartite balanced graphs,
we get that if $G$ contains no (1 mod 2)-cycles
then $e(G)\leqslant \lfloor \frac{n^2}{4}\rfloor$.
Chen and Saito \cite{CS94} showed that 
if $G$ contains no (0 mod 3)-cycles
then $e(G)\leqslant 2(n-1)$ for large $n$
and the equality holds if and only if 
$G$ is isomorphic to $K_{2,n-2}$.
Dean et al. \cite{DKOT91}
and, independently, Saito \cite{Sai92}
showed that if $G$ contains no (2 mod 3)-cycles
then $e(G)\leqslant 3(n-3)$
and the equality holds if and only if 
$G$ is isomorphic to $K_4$ or $K_{3,n-3}$.
Recently,
Gy\H{o}ri et al. \cite{G+24} showed that 
if $G$ contains no (0 mod 4)-cycles
then $e(G)\leqslant \frac{19}{12}(n-1)$
and they also constructed a type of extremal graphs;
Gao et al. \cite{GLMX24} showed that 
if $G$ contains no (2 mod 4)-cycles
then $e(G)\leqslant \frac{5}{2}(n-1)$
and, for $4|(n-1)$, 
the equality holds if and only if 
each block of $G$ is $K_5$.


Dean et al. \cite{DKOT91},
Lu and Yu \cite{LY01} gave some conditions for the existence of (1 mod 3)-cycles.
Here we list the result of Dean et al.,
which will be used in our proof.

\begin{theorem}[Dean et al. \cite{DKOT91}]\label{thm: dean et al.}
Let $G$ be a $2$-connected graph with minimum degree at least $3$.
If $G$ contains no $(1\bmod 3)$-cycles,
then $G$ is isomorphic to the Petersen graph.
\end{theorem}


In this paper,
we precisely determine the maximum number of edges 
in an $n$-vertex graph containing no cycles of length $1$ modulo $3$.
This is the main result of this paper.

\begin{theorem}\label{thm: main}
Let $G$ be an $n$-vertex graph.
If 
\[
e(G)>15q+\left\lfloor\frac{3r}{2}\right\rfloor,~
\text{where~}n-1=9q+r,~0\leqslant r\leqslant 8,
\]
then $G$ contains a $(1\bmod 3)$-cycle.
\end{theorem}

For two integers $k>\ell\geqslant 0$ satisfying that
$k\mathbb{Z}+\ell$ contains an even integer,
define $c_{\ell,k}$ to be the least constant such that
every $n$-vertex graph with at least $c_{\ell,k}\cdot n$ edges contains a cycle of length $\ell$ modulo $k$. 
The last remaining class modulo $k$ for $2\leqslant k\leqslant 4$ is the case of $k=3$ and $\ell=1$. 
An upper bound of $c_{1,3}\leqslant 3$ can be obtained in the original paper of Erd\H{o}s \cite{Erd76}.
Our main result, Theorem \ref{thm: main}, implies that $c_{1,3}=\frac{5}{3}$.
Together with the known results mentioned above,
we have all the precise values of $c_{\ell,k}$ for $k\leqslant 4$, see Table \ref{tab:clk}.

\begin{table}[ht]
    \renewcommand{\arraystretch}{1.5}
	\centering
	\begin{tabular}{l|l|l}
		\hline \hline 
		$c_{0,2}=\frac{3}{2}$ (trivial)                   & -                        & -                                              \\ \hline
		$c_{0,3}=2$ (\cite{CS94})             & $c_{1,3}=\frac{5}{3}$ (this paper) & $c_{2,3}=3$ (\cite{DKOT91,Sai92}) \\ \hline
		$c_{0,4}=\frac{19}{12}$ (\cite{G+24}) & -                                      & $c_{2,4}=\frac{5}{2}$ (\cite{GLMX24})       \\ \hline \hline 
	\end{tabular}
	\caption{Precise values of $c_{\ell,k}$ for $k\leqslant 4$.}
	\label{tab:clk}
\end{table}

Furthermore, 
by using Theorem \ref{thm: dean et al.},
we show that if $9|(n-1)$ 
then graphs in which each block is a Petersen graph 
constitute the unique class of extremal graphs 
that have the maximum number of edges
but contain no cycles of length 1 modulo 3.
The proof can be found in Section \ref{section: extremal graphs}.

\begin{theorem}\label{thm: petersen-blocks}
Let $G$ be an $n$-vertex graph containing no $(1\bmod 3)$-cycles, 
then $e(G)\leqslant \frac{5}{3}(n-1)$. 
Moreover, the equality holds if and only if $9|(n-1)$ and each block of $G$ is a Petersen graph.   
\end{theorem}


It is worth noting that for $k>\ell\geqslant 3$, 
Sudakov and Verstra\"{e}te \cite{SV17} showed that 
\[
\frac{ex(k,C_{\ell})}{k}
\leqslant c_{\ell,k}
\leqslant 96\cdot \frac{ex(k,C_{\ell})}{k},
\]
where $ex(k,C_{\ell})$ is the Tur\'{a}n number of $C_{\ell}$, i.e., the maximum number of edges 
in a $k$-vertex graph containing no cycles of length $\ell$.
It follows that for even $\ell\geqslant 4$
determining $c_{\ell,k}$ is at least as hard as the problem of determining $ex(k,C_{\ell})$.

The rest of this paper is organized as follows.
In Section \ref{section: pre}, we give some necessary notations and definitions as well as some useful lemmas. 
In Section \ref{section: lemma L123},
we give a condition for the existence of two disjoint cycles.
Section \ref{section: proof of main thm} 
is devoted to the proof of Theorem \ref{thm: main}.
Constructions attaining the bound of Theorem \ref{thm: main} for every $n$ are given in Section \ref{section: extremal graphs}.

\section{Notations and preliminaries}\label{section: pre}

Let $G=(V,E)$ be a graph and $x\in V(G)$. The {\em degree} of $x$, denoted by $d_G(x)$, is the number of its neighbors in $G$, whose set is denoted by $N_G(x)$. The minimum degree of a graph $G$, denoted by $\delta(G)$, is the minimum number among all the degrees of the vertices in $G$. Denote by $N_2(G)$ the set of vertices with degree 2 in $G$. A subset of $V(G)$ is an {\em independent set} if no two vertices are adjacent to each other.

A path from $x$ to $y$ is called an $(x,y)$-path, 
and $x,y$ are the origin and terminus of the path. 
The distance between $x$ and $y$ in $G$ is the length of a shortest $(x,y)$-path of $G$. If $X,Y$ are two subgraphs of $G$ or two subsets of $V(G)$, then a path from $X$ to $Y$ is one with the origin in $X$, the terminus in $Y$ and all internal vertices outside $X\cup Y$. For a path $P$ (or a cycle $C$), we will use $|P|$ (or $|C|$) to denote its length.

A graph is {\em connected} if there exists a path between any two vertices in the graph. For a positive integer $k$, call a graph {\em $k$-connected} if it has order at least $k+1$ and the deletion of any set of at most $k-1$ vertices results in a connected subgraph. A {\em component} of a graph $G$ is a maximal connected subgraph of $G$. A subset $S$ of $V(G)$ is a {\em vertex cut} of $G$ if $G-S$ has more components than $G$.
If a vertex cut consists of exactly one vertex $x$, then $x$ is a {\em cut-vertex} of $G$. A {\em block} of $G$ is a maximum connected subgraph of $G$ with no cut-vertex. In other words, a block is an isolated vertex, a cut-edge or a maximal 2-connected subgraph. Call a graph {\em trivial} if it contains only one vertex. A vertex cut $S$ of $G$ is {\em essential} if $G-S$ has at least two nontrivial components. Call a connected graph {\em essentially $k$-connected} if it contains no essential vertex cut of size at most $k-1$.

For a path $P$ with two vertices $x,y$ in $V(P)$, denote by $P[x,y]$ the subpath from $x$ to $y$ in $P$. For a cycle $C$ with a given orientation $\overrightarrow{C}$, we use $\overleftarrow{C}$
to denote its reverse-orientation. We use $\overrightarrow{C}[x,y]$ and $\overleftarrow{C}[x,y]$ to denote the paths from $x$ to $y$ along with the orientations $\overrightarrow{C}$ and $\overleftarrow{C}$, respectively. Let $P(x,y)$ and $\overrightarrow{C}(x,y)$ be the paths obtained from $P[x,y]$ and $\overrightarrow{C}[x,y]$ by removing $x$ and $y$.
For a cycle $C$ with a given orientation, and a vertex $x\in V(C)$, denote by $x^+$ the successor, and $x^-$ the predecessor of $x$ along $C$.

For a cycle $C$ and two distinct vertices $x,y$ on it, 
if $|\overrightarrow{C}[x,y]|\equiv |\overrightarrow{C}[y,x]|\bmod 3$, 
then we say that $\{x,y\}$ is a {\em diagonal pair modulo $3$} in $C$; 
otherwise, we say that $\{x,y\}$ is a {\em non-diagonal pair modulo $3$} in $C$. When no confusion occurs, we say $x,y$ are {\em mod-diagonal} (resp. {\em mod-non-diagonal}) in $C$ for short in the following. For terminology and notations not defined here, we refer the reader to \cite{BM08}.


\begin{lemma}\label{LenDiagonal}
Let $C$ be a cycle and let $x,y,z$ be three distinct vertices that appear in this order along $C$. Then\\
\indent $(1)$ $x,y$ are mod-diagonal in $C$ if and only if $|\overrightarrow{C}[x,y]|\equiv-|C|\bmod 3$;\\
\indent $(2)$ if $x$ is mod-diagonal with both $y$ and $z$, then $|\overrightarrow{C}[y,z]|\equiv 0\bmod 3$.
\end{lemma}

\begin{proof}
The assertion (1) can be deduced from the fact that $|\overrightarrow{C}[x,y]|+|\overleftarrow{C}[x,y]|=|C|$, and this follows that $|\overrightarrow{C}[x,y]|\equiv|\overrightarrow{C}[x,z]|$, which implies the assertion (2).
\end{proof}

For a cycle $C$ of $G$,
a {\em bridge} of $C$ is either 
(1) a chord of $C$ (an edge not in $C$ but its end-vertices are in $C$), 
or (2) a component $H$ of $G-C$ together with the edges between $H$ and $C$ 
(and vertices in $N_C(H):=\bigcup_{x\in V(H)}N_C(v)$). The following lemma is a variation of Lemma 1 in \cite{BoVi}, 
and the proof is similar to that in \cite{BoVi}. 
We omit the details.

\begin{lemma}\label{lemma: two cycles}
Let $G$ be a $2$-connected graph, 
$C$ a cycle of $G$ and 
$D$ a connected subgraph of $G-C$. 
Let $C'$ be a cycle disjoint with $D$ 
such that the bridge $B$ of $C'$ containing $D$ is as large as possible. 
Then either \\
\indent $(1)$ $C'$ has only one bridge $B$; or \\
\indent $(2)$ $B$ has exactly two common vertices with $C'$, and all other bridges of $C'$ are paths between the two vertices.
\end{lemma}

\begin{lemma}\label{LePQ1Q2Q3}
Let $P$, $Q_1$, $Q_2$, $Q_3$ be four paths from $x$ to $y$ such that $P$ is internally-disjoint with $Q_1,Q_2,Q_3$ and $Q_1,Q_2,Q_3$ have pairwise distinct lengths modulo $3$. Then $P\cup Q_1\cup Q_2\cup Q_3$ contains a $(1\bmod 3)$-cycle. 
\end{lemma}

\begin{proof}
Notice that $P\cup Q_1$, $P\cup Q_2$, $P\cup Q_3$ are three cycles with distinct lengths modulo 3. So one of them is a $(1\bmod 3)$-cycle.
\end{proof}

For convenience, we use `disjoint' for `vertex-disjoint'.

\begin{lemma}\label{LeC1C2P1P2}
Let $C_1, C_2$ be two disjoint cycles, 
$P_1,P_2$ two disjoint paths from $C_1$ to $C_2$, 
and denote the end-vertices of $P_i$ by $x_i,y_i$, 
where $x_i\in V(C_1)$, $y_i\in V(C_2)$, $i=1,2$. 
If $x_1,x_2$ are mod-non-diagonal in $C_1$ and $y_1,y_2$ are mod-non-diagonal in $C_2$, then $C_1\cup C_2\cup P_1\cup P_2$ contains a $(1\bmod 3)$-cycle.
\end{lemma}

\begin{proof}
Since $x_1,x_2$ are mod-non-diagonal in $C_1$, 
the two paths from $x_1$ to $x_2$ on $C_1$ 
have lengths $a_1$ and $a_1+1\bmod 3$ for some $a_1$. 
Similarly, the two paths from $y_1$ to $y_2$ on $C_2$ 
have lengths $a_2$ and $a_2+1\bmod 3$ for some $a_2$,
see Figure \ref{fig:2 non-diagonal pairs}. 
Set $a_0=|P_1|+|P_2|$. 
Then there are three cycles of lengths 
$\sum_{i=0}^2a_i$, $\sum_{i=0}^2a_i+1$, $\sum_{i=0}^2a_i+2\bmod 3$, respectively, 
one of which is a $(1\bmod 3)$-cycle.
\end{proof}

\begin{figure}[ht]
\centering
\begin{tikzpicture}[scale=0.6]
\draw(0,0) circle (2); 
\draw(7,0){coordinate (o2)} circle (2);
\coordinate (x1) at (45:2);
\node at (45:1.5) {$x_1$};
\coordinate (x2) at (315:2);
\node at (315:1.5) {$x_2$};
\foreach \x in {1,2}
\draw[fill=black] (x\x) circle (0.1);

\node[left] at (180:2) {$a_1+1$};
\node[right] at (0:2) {$a_1$};

\path(o2)--+(135:2) coordinate (y1);
\path(o2)--+(135:1.5){node{$y_1$}};
\path(o2)--+(225:2) coordinate (y2);
\path(o2)--+(225:1.5){node{$y_2$}};
\foreach \x in {1,2}
\draw[fill=black] (y\x) circle (0.1);

\path(o2)--+(180:2){node[left]{$a_2$}};
\path(o2)--+(0:2){node[right]{$a_2+1$}};

\path(x1) edge [bend left=20] (y1);
\path(x2) edge [bend right=20] (y2);

\node at (135:2.4) {$C_1$};
\path(o2)--+(45:2.4){node{$C_2$}};
\node[above] at (3.5,1.8) {$P_1$};
\node[above] at (3.5,-1.8) {$P_2$};

\end{tikzpicture}
\caption{Two disjoint paths $P_1,P_2$ from $C_1$ to $C_2$. }
\label{fig:2 non-diagonal pairs}
\end{figure}
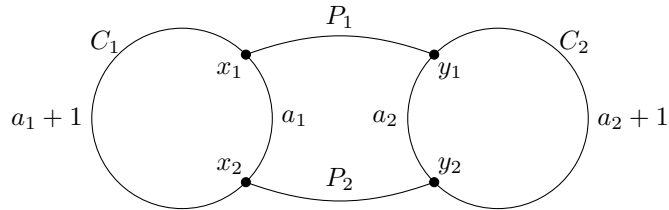

\begin{lemma}\label{lemma: 2 paths mod 3}
Let $G$ be a graph on at least $4$ vertices and $x,y$ two distinct vertices of $G$. If\\
\indent $(1)$ $G+xy$ is $2$-connected; \\
\indent $(2)$ $N_2(G)\backslash\{x,y\}$ is an independent set; and \\
\indent $(3)$ $G$ contains no $4$-cycles; \\
then $G$ contains two $(x,y)$-paths $P_1,P_2$ with $|P_1|\not\equiv|P_2|\bmod 3$.
\end{lemma}

\begin{proof}
We prove by induction on $n=|V(G)|$. 
If $n=4$, then denote by $V(G)=\{x,y,u,v\}$. 
Since $G+xy$ is $2$-connected, 
we get that $G$ is connected and $d_G(u),d_G(v)\geqslant 2$. 
If $uv\notin E(G)$, then $ux,uy,vx,vy\in E(G)$ and $xuyvx$ is a 4-cycle in $G$, contradicting (3). So $uv\in E(G)$. Since $N_2(G)\backslash\{x,y\}$ is an independent set, we have $\{u,v\}\nsubseteq N_2(G)$ and we can assume w.l.o.g. that $d_G(u)=3$. Then $u$ is adjacent to all vertices in $\{x,y,v\}$. By $d_G(v)\geqslant 2$, we have either $vx\in E(G)$ or $vy\in E(G)$, say $vx\in E(G)$. Now $xuy,xvuy$ are two desired $(x,y)$-paths. 

From now on we assume that $n\geqslant 5$ 
and the desired statement holds for graphs with less than $n$ vertices. We can assume that $xy\notin E(G)$, since otherwise we can remove the edge $xy$ and the resulting graph also satisfies the conditions of our assertion.

\begin{claim}
$G$ is $2$-connected.
\end{claim}

\begin{proof}
Since $G+xy$ is 2-connected, $G$ is connected. 
Assume that $G$ is not 2-connected. 
Then there exists a cut-vertex $z$ and two connected subgraphs $G_1,G_2$ of $G$ on at least two vertices such that $G=G_1\cup G_2$ and $V(G_1)\cap V(G_2)=\{z\}$. 
Since $z$ is not a cut-vertex of $G+xy$, 
we get that $z$ separates $x$ and $y$ in $G$. 
Assume that $x\in V(G_1)\backslash \{z\}$ and $y\in V(G_2)\backslash \{z\}$, and assume w.l.o.g. that $|V(G_1)|\geqslant |V(G_2)|$. If $|V(G_1)|\geqslant 4$, then by induction hypothesis $G_1$ contains two $(x,z)$-paths $Q_1,Q_2$ with different lengths modulo 3. Let $Q$ be any $(z,y)$-path in $G_2$. Now $P_1=Q_1\cup Q$, $P_2=Q_2\cup Q$ are two desired $(x,y)$-paths and we are done. If $|V(G_1)|\leqslant 3$, then $|V(G_2)|\leqslant 3$, and by $|V(G_1)|+|V(G_2)|=n+1\geqslant 6$, we have $|V(G_1)|=|V(G_2)|=3$. Set $V(G_1)=\{x,z,u\}$ and $V(G_2)=\{y,z,v\}$. Since $G+xy$ is 2-connected, we have $d_G(u),d_G(v)\geqslant 2$. Together with $|V(G_1)|=|V(G_2)|=3$, we have $N_G(u)=\{x,z\}$ and $N_G(v)=\{y,z\}$. Since $N_2(G)\backslash \{x,y\}$ is an independent set, $d_G(z)\geqslant 3$. It follows that either $xz\in E(G)$ or $yz\in E(G)$. Assume w.l.o.g. that $xz\in E(G)$. Now $xzvy,xuzvy$ are two $(x,y)$-paths of different lengths modulo 3.
\end{proof}

\begin{claim}
$G$ has a cycle which contains $x$ but does not contain $y$.
\end{claim}

\begin{proof}
Note that every vertex in a 2-connected graph is contained in some cycle. So there exists at least one cycle which contains $x$. Assume that any cycle containing $x$ also contains $y$. Then $\{x,y\}$ is a vertex cut of $G$. Let $H$ be a component of $G-\{x,y\}$ which has maximum number of vertices. Denote by $G_1=G[V(H)\cup \{x,y\}]$. If $|V(H)|=1$, then let $V(H)=\{u\}$. Since $G+xy$ is 2-connected, we have $ux,uy\in E(G)$. Recall that $|V(G)|\geqslant 4$. There exists another trivial component $H'$ with $V(H')=\{v\}$. Similarly, we can get that $vx,vy\in E(G)$. Then a 4-cycle $xuyvx$ appears, a contradiction. So $|V(H)|\geqslant 2$. Then by induction hypothesis there exist two $(x,y)$-paths with different lengths modulo 3 in $G_1$, as desired.
\end{proof}

\begin{claim}\label{CluvEquiv0mod3}
Let $C$ be a cycle of $G$ that contains $x$ but does not contain $y$, 
let $H$ be the component of $G-C$ with $y\in V(H)$, 
and let $u,v$ be two vertices in $N_C(H)\backslash\{x\}$. 
We give an orientation on $C$ such that $x,u,v$ appear in this order along $C$. 
Then $|\overrightarrow{C}[u,v]|\equiv 0\bmod 3$.
\end{claim}

\begin{proof}
    We claim that $x$ and $u$ are mod-diagonal in $C$. Otherwise, $\overleftarrow{C}[x,u]$ and $\overrightarrow{C}[x,u]$ have different lengths modulo 3. Together with an $(u,y)$-path with all internal vertices in $H$, we can find two $(x,y)$-paths of different lengths modulo 3, as desired. So $x$ and $u$ are mod-diagonal in $C$; and similarly, $x$ and $v$ are mod-diagonal in $C$. By Lemma \ref{LenDiagonal}, $|\overrightarrow{C}[x,u]|\equiv|\overrightarrow{C}[x,v]|\equiv-|C|\bmod 3$, and $|\overrightarrow{C}[u,v]|\equiv 0\bmod 3$.
\end{proof}

Now let $C$ be a cycle of $G$ that contains $x$ but does not contain $y$, and let $H$ be the component of $G-C$ that contains $y$. We give an orientation of $C$.

\begin{claim}\label{ClNCHgeq2}
    $|C|\geqslant 5$ and $|N_C(H)\backslash\{x\}|\geqslant 2$.
\end{claim}

\begin{proof}
    By (3), $|C|\neq 4$.
    If $C$ is a triangle, then any vertex in $V(C)\backslash\{x\}$ is mod-non-diagonal with $x$ in $C$. Since $G$ is 2-connected, we get that $H$ has at least one neighbor in $V(C)\backslash\{x\}$. It follows that there are two $(x,y)$-paths of different lengths modulo 3, as desired. So $|C|\geqslant 5$.

    If $|N_C(H)\backslash\{x\}|=1$, say $N_C(H)\backslash\{x\}=\{z\}$, then every vertex in $G-H$ other than $x,z$ has the same degree as that in $G$. Since $G-H$ is 2-connected, by induction hypothesis, $G-H$ contains two $(x,z)$-paths of different lengths modulo 3. Together with a $(z,y)$-path with all internal vertices in $H$, we get two $(x,y)$-paths of different lengths modulo 3 in $G$, as desired.
\end{proof}

By Claim \ref{ClNCHgeq2}, let $u,v\in N_C(H)\backslash\{x\}$, where $x,u,v$ appear in this order along $C$. We choose the cycle $C$ and $u,v$ such that $|\overrightarrow{C}[u,v]|$ is as small as possible. By Claim \ref{CluvEquiv0mod3}, $|\overrightarrow{C}[u,v]|\equiv 0\bmod 3$, and thus $|\overrightarrow{C}[u,v]|\geqslant 3$. 

\begin{claim}\label{CluvCut}
    $\{u,v\}$ is a vertex cut of $G$ separating $V(\overrightarrow{C}(u,v))$ and $V(\overrightarrow{C}(v,u))\cup V(H)$.
\end{claim}

\begin{proof}
Suppose that the assertion is not true. 
Then there exists a path $Q$ from $V(\overrightarrow{C}(u,v))$ to $V(\overrightarrow{C}(v,u))\cup V(H)$ avoiding $u,v$. 
We assume that $Q$ is as short as possible. 
By the choice of $u$ and $v$, 
the component $H$ has no neighbors in $\overrightarrow{C}(u,v)$. 
It follows that $Q$ is a path from $V(\overrightarrow{C}(u,v))$ to $V(\overrightarrow{C}(v,u))$ with all internal vertices in $V(G)\backslash(V(C)\cup V(H))$ (not excluding the case that $Q$ is a chord of $C$). 
Let $w_1,w_2$ be the end-vertices of $Q$, 
where $w_1\in V(\overrightarrow{C}(u,v))$ and $w_2\in V(\overrightarrow{C}(v,u))$. 
Assume w.l.o.g. that 
$w_2\in V(\overrightarrow{C}[x,u])\backslash\{u\}$, 
see Figure \ref{figure: uv-cut}. 
Let $C'=w_2Qw_1\overrightarrow{C}[w_1,w_2]$ 
and let $H'$ be the component of $G-C'$ containing $y$. 
Now $H'$ contains $H$ and $\overrightarrow{C}(w_2,w_1)$. 
It follows that $w_1,v\in N_{C'}(H')$.
Now $|\overrightarrow{C'}[w_1,v]|<|\overrightarrow{C}[u,v]|$,
a contradiction to the choice of $C$ and $u,v$.
\end{proof}

\begin{figure}[ht]
\centering
\begin{tikzpicture}[scale=0.6]
\draw(0,0) circle (2); \node at (135:2.3){$C$};
\draw[fill=black](180:2){node[right]{$x$}} circle (0.1);
\draw[fill=black](315:2) circle (0.1); \node at (315:1.6){$u$};
\draw[fill=black](45:2) circle (0.1); \node at (45:1.6){$v$};
\draw[fill=black](0:2) circle (0.1); \node at (8:1.6){$w_1$};
\draw[fill=black](250:2) circle (0.1); \node at (240:1.6){$w_2$};
\path(0:2) edge [bend right=30] (250:2); \node at (0.2,-0.2){$Q$};
\draw[dashed](7,0){node{$H$}} circle (2);
\draw[fill=black](8.5,0){node[below]{$y$}} circle (0.1);
\draw[fill=black](5.5,0.5) circle (0.1); \draw[fill=black](5.5,-0.5) circle (0.1);
\path[thick](45:2) edge [bend left=20] (5.5,0.5);
\path[thick](315:2) edge [bend right=20] (5.5,-0.5);
\end{tikzpicture}

\caption{A cycle $C$, a component $H$ and a path $Q$.}
\label{figure: uv-cut}
\end{figure}
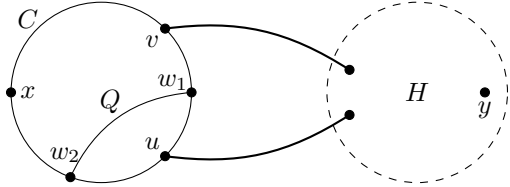

Now we let $H_1$ be a component of $G-\{u,v\}$ containing $\overrightarrow{C}(u,v)$ and let $G_1=G[V(H_1)\cup\{u,v\}]$. By Claim \ref{CluvEquiv0mod3}, $|\overrightarrow{C}[u,v]|\equiv 0\bmod 3$, implying that $|V(G_1)|\geqslant 4$. Notice that $G_1+uv$ is 2-connected. By induction hypothesis, $G_1$ has two $(u,v)$-paths $Q_1,Q_2$ of different lengths modulo 3. Together with $\overrightarrow{C}[x,u]$ and a $(v,y)$-path with all internal vertices in $H$, we find two $(x,y)$-paths of different lengths modulo 3, as desired.
\end{proof}

\section{Two disjoint cycles in graphs}\label{section: lemma L123}

In this paper, 
we need to consider the existence of two disjoint cycles in graphs. In 1965, Lov\'{a}sz \cite{L65} (also see \cite{B78}) characterized graphs that contain no two disjoint cycles. Here we characterize graphs that contain no ($1\bmod 3$)-cycles and no two disjoint cycles.

For a subset $U\subseteq V(G)$, 
we use $\rho(U)$ to denote the number of edges 
that are incident to a vertex in $U$.

\begin{lemma}\label{lemma: L123}
Let $G$ be a $2$-connected graph. Denote by $C$ a shortest cycle of $G$ and $T$ an arbitrary component of $G-C$. Suppose that \\
\indent $(1)$ $G$ contains no $(1\bmod 3)$-cycles; \\
\indent $(2)$ $G$ contains no two disjoint cycles; \\
\indent $(3)$ $\rho(U)>\lfloor\frac{3}{2}|U|\rfloor$ for any subset $U\subseteq V(T)$.\\
Then $C$ is a $5$-cycle and $G[V(C)\cup V(T)]$ is isomorphic to $L_1, L_2$ or $L_3$, as in Figure \ref{figure: L123}.
\end{lemma}

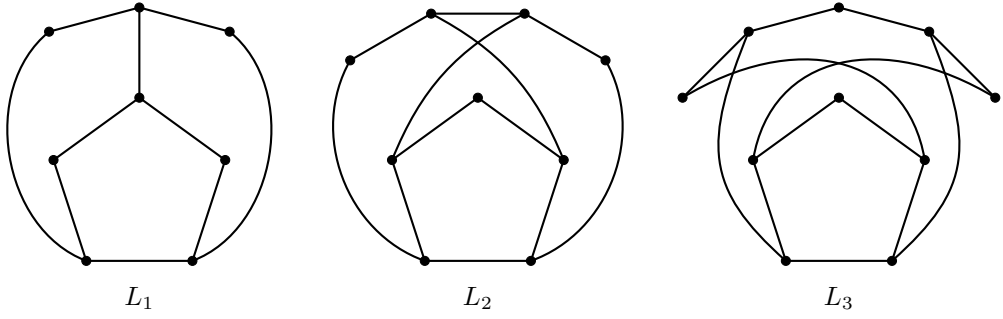
\begin{figure}[ht]
\centering
\begin{tikzpicture}[scale=0.6]

\begin{scope}
\foreach \x in {1,2,...,5}
{\coordinate (x\x) at (162-72*\x:2);
\draw[fill=black](x\x) circle (0.1);}
\draw[thick](x1)--(x2)--(x3)--(x4)--(x5)--(x1); 
\foreach \x in {1,2,3}
{\coordinate (y\x) at (150-30*\x:4);
\draw[fill=black](y\x) circle (0.1);}
\draw[thick](y1)--(y2)--(y3); 
\path[thick] (y1) edge [bend right=60] (x4);
\draw[thick] (y2)--(x1);
\path[thick] (y3) edge [bend left=60] (x3);
\node[below] at (0,-2) {$L_1$};
\end{scope}

\begin{scope}[xshift=7.5cm]
\foreach \x in {1,2,...,5}
{\coordinate (x\x) at (162-72*\x:2);
\draw[fill=black](x\x) circle (0.1);}
\draw[thick](x1)--(x2)--(x3)--(x4)--(x5)--(x1); 
\foreach \x in {1,2,3,4}
{\coordinate (y\x) at (165-30*\x:4);
\draw[fill=black](y\x) circle (0.1);}
\draw[thick](y1)--(y2)--(y3)--(y4); 
\path[thick](y1) edge [bend right=50] (x4);
\path[thick](y2) edge [bend left=20] (x2);
\path[thick](y3) edge [bend right=20] (x5);
\path[thick](y4) edge [bend left=50] (x3);
\node[below] at (0,-2) {$L_2$};
\end{scope}

\begin{scope}[xshift=15.5cm]
\foreach \x in {1,2,...,5}
{\coordinate (x\x) at (162-72*\x:2);
\draw[fill=black](x\x) circle (0.1);}
\draw[thick](x1)--(x2)--(x3)--(x4)--(x5)--(x1); 
\foreach \x in {1,2,...,5}
{\coordinate (y\x) at (180-30*\x:4);
\draw[fill=black](y\x) circle (0.1);}
\draw[thick](y1)--(y2)--(y3)--(y4)--(y5); 
\draw[thick](y1)..controls (-1,3.5) and (1.5,3)..(x2);
\draw[thick](y2)..controls (-3,1) and (-3,0)..(x4);
\draw[thick](y4)..controls (3,1) and (3,0)..(x3);
\draw[thick](y5)..controls (1,3.5) and (-1.5,3)..(x5);
\node[below] at (0,-2) {$L_3$};
\end{scope}
\end{tikzpicture}
\caption{Graphs $L_1$, $L_2$ and $L_3$.}\label{figure: L123}
\end{figure}

We remark that in each of the graphs $L_1,L_2,L_3$, 
every 5-cycle can be transferred to another 5-cycle by an isomorphic mapping.

\begin{proof}
Since $G$ contains no two disjoint cycles, $T$ is a tree. We first obtain some structural properties of $G$ from the the following claims.

\setcounter{claim}{0}
\begin{claim}\label{claim: at most one neighbor}
$d_C(z)\leqslant 1$ for any $z\in V(T)$.
\end{claim}

\begin{proof}
Assume the opposite that $z$ has two neighbors $x_1,x_2$ in $C$ for some $z\in V(T)$. We can assume w.l.o.g. that $|\overrightarrow{C}[x_1,x_2]|\leqslant |\overrightarrow{C}[x_2,x_1]|$. If $|C|=3$, then a 4-cycle appears, a contradiction. Recall that $|C|\neq 4$. If $|C|\geqslant 5$, then $\overrightarrow{C}[x_1,x_2]x_2zx_1$ is a cycle shorter than $C$ as $|\overrightarrow{C}[x_2,x_1]|\geqslant 3$, a contradiction to the minimality of $C$.
\end{proof}

\begin{claim}\label{ClLeaf}
$d_C(z)=1$ and $d_G(z)=2$ for any leaf $z$ of $T$.
\end{claim}

\begin{proof}
By $G$ being 2-connected, $d_G(z)\geqslant 2$. The desired result then follows from Claim \ref{claim: at most one neighbor}.
\end{proof}

\begin{claim}
$|C|\in \{3,5,6,8\}$.
\end{claim}

\begin{proof}
Since $|C|\geqslant 3$ and $|C|\not\equiv 1\bmod 3$,
it suffices to show that $|C|\leqslant 8$. 
Assume the opposite that $|C|\geqslant 9$. 
Let $P=z_1z_2z_3\ldots z_p$ be a longest path of $T$. Thus $z_1$ is a leaf of $T$ and $d_C(z_1)=1$ by Claim \ref{ClLeaf}. Denote by $y_1$ the neighbor of $z_1$ in $C$. If $z_2$ has a neighbor in $C$, say $y_2$, then $y_1\neq y_2$ (otherwise $z_1z_2y_1z_1$ is a cycle shorter than $C$) and $\max\{|\overrightarrow{C}[y_1,y_2]|,|\overrightarrow{C}[y_2,y_1]|\}\geqslant 5$. Replacing $\overrightarrow{C}[y_1,y_2]$ or $\overrightarrow{C}[y_2,y_1]$ with the 3-path $y_1z_1z_2y_2$ will yield a cycle shorter than $C$, a contradiction. So $d_C(z_2)=0$. By (3), $\rho(\{z_1,z_2\})\geqslant 4$, implying that $d_T(z_2)\geqslant 3$. So $z_2$ has a neighbor $z'_1$ distinct with $z_1,z_3$ in $T$. One can see that $z'_1$ is also a leaf of $T$ as $P$ is a longest path in $T$. So $z'_1$ has a neighbor in $C$, say $y'_1$. If $y_1=y'_1$, then a 4-cycle appears, a contradiction. So $y_1\neq y'_1$. Now replacing $\overrightarrow{C}[y_1,y'_1]$ or $\overrightarrow{C}[y'_1,y_1]$ with the 4-path $y_1z_1z_2z'_1y'_1$ will also yield a cycle shorter than $C$, a contradiction.
\end{proof}

\begin{claim}\label{Cldcx1x2}
Let $z_1,z_2\in V(T)$ with $d_T(z_1,z_2)=1$ or $2$, and let $y_1,y_2$ be neighbors of $z_1,z_2$, respectively, in $C$. Then the distances $d_C(y_1,y_2)$ is determined by $d_T(z_1,z_2)$ and $|C|$, as in Table \ref{tab:distance}.
Moreover, if $d_T(z_1,z_2)=1$, then $|C|\neq 8$. 

\begin{table}[ht]
\renewcommand{\arraystretch}{1.5}
\centering
\begin{tabular}{c|c|c|c|c}
  \hline\hline
   & $|C|=3$ & $|C|=5$ & $|C|=6$ & $|C|=8$\\\hline
  $d_T(z_1,z_2)=1$ & 0 & 2 & 3 & -\\\hline
  $d_T(z_1,z_2)=2$ & 1 & 1 & 2 & 4\\\hline\hline
\end{tabular}

\caption{The distance $d_C(y_1,y_2)$.}
\label{tab:distance}
\end{table}
\end{claim}

\begin{proof}
Suppose w.l.o.g. that $|\overrightarrow{C}[y_1,y_2]|\leqslant |C|/2$ (including the case $y_1=y_2$). Set $P=\overrightarrow{C}[y_1,y_2]$ and let $Q$ be the path from $z_1$ to $z_2$ in $T$. Then $|P|=d_C(y_1,y_2)$, $|Q|=d_T(z_1,z_2)$ and the cycle $y_1Py_2z_2Qz_1y_1$ has length $|P|+|Q|+2$. If $|Q|=1$, then $|P|\geqslant |C|-3$, and if $|Q|=2$, then $|P|\geqslant |C|-4$; since otherwise $y_1Py_2z_2Qz_1y_1$ is a cycle shorter than $C$. 

Suppose first that $|Q|=1$. If $|C|=3$, then $|P|=0$ or 1. But if $|P|=1$, then $Py_2z_2Qz_1y_1$ is a 4-cycle, contradicting (1). This implies that $|P|=0$, as desired. If $|C|=5$, then $|C|-3\leqslant |P|\leqslant |C|/2$, implying that $|P|=2$, as desired. If $|C|=6$, then $|C|-3\leqslant |P|\leqslant |C|/2$, implying that $|P|=3$, as desired. If $|C|=8$, then $|C|-3\leqslant |P|\leqslant |C|/2$ and $5\leqslant 4$, a contradiction.

Now suppose that $|Q|=2$. If $|C|=3$, then $|P|=0$ or 1. But if $|P|=0$, then $y_1=y_2$ and $y_1z_2Qz_1y_1$ is a 4-cycle, a contradiction. This implies that $|P|=1$, as desired. If $|C|=5$, then $|C|-4\leqslant |P|\leqslant |C|/2$, implying that $|P|=1$ or 2. But if $|P|=2$, then $y_1\overleftarrow{C}[y_1,y_2]y_2z_2Qz_1y_1$ is a 7-cycle, a contradiction. This implies that $|P|=1$, as desired. If $|C|=6$, then $|C|-4\leqslant |P|\leqslant |C|/2$, implying that $|P|=2$ or 3. But if $|P|=3$, then $Py_2z_2Qz_1y_1$ is a 7-cycle, a contradiction. This implies that $|P|=2$, as desired. If $|C|=8$, then $|C|-4\leqslant |P|\leqslant |C|/2$, implying that $|P|=4$, as desired.
\end{proof}

\begin{claim}\label{ClNoStar}
$|V(T)|\geqslant 3$ and $T$ is not a star.
\end{claim}

\begin{proof}
If $|V(T)|=2$, then by Claim \ref{ClLeaf} we have $\rho(V(T))=3$, contradicting (3). So $|V(T)|\geqslant 3$. Now assume that $T$ is a star $K_{1,t}$ with $t\geqslant 2$. Let $z_0$ be the center vertex of $T$ and $z_1,\ldots,z_t$ be its leaves. Then $d_C(z_i)=1$ for any $i$ with $1\leqslant i\leqslant t$. Moreover, by $\rho(V(T))>\frac{3}{2}|V(T)|$, we have that if $t\in\{2,3\}$ then $d_C(z_0)=1$. Denote by $y_i$ the (possible) unique neighbor of $z_i$ in $C$.

Suppose first that $t\geqslant 3$. By Claim \ref{Cldcx1x2}, all pairs of vertices $y_i,y_j$ with $1\leqslant i<j\leqslant t$ have the same distance in $C$. This implies that $t=3$ and $|C|\in \{3,6\}$. Now $d_C(z_0)=1$. By Claim \ref{Cldcx1x2}, 
all pairs of vertices $y_0,y_i$ with $1\leqslant i\leqslant 3$ have the same distance in $C$, which is impossible.

Now suppose that $t=2$. Then $d_C(z_0)=1$ and $|C|\neq 8$ by Claim \ref{Cldcx1x2}. If $|C|=3$, then by Claim \ref{Cldcx1x2}, $d_C(y_0,y_1)=d_C(y_0,y_2)=0$, implying that $d_C(y_1,y_2)=0$, contradicting Claim \ref{Cldcx1x2}. If $|C|=5$, then $d_C(y_0,y_1)=d_C(y_0,y_2)=2$ and $d_C(y_1,y_2)=1$. One can see that $G[V(C)\cup V(T)]$ is isomorphic to $L_1$, as desired. If $|C|=6$, then $d_C(y_0,y_1)=d_C(y_0,y_2)=3$, implying that $d_C(y_1,y_2)=0$, contradicting Claim \ref{Cldcx1x2}.
\end{proof}

Set $C=x_1x_2\ldots x_{|C|}x_1$, and let $P=z_1z_2\ldots z_p$ be a longest path in $T$. By Claim \ref{ClNoStar}, we have $p\geqslant 4$.

\begin{claim}\label{Cldy2}
$d_G(z_2)=3$ and $d_G(z_{p-1})=3$.
\end{claim}

\begin{proof}
By $\rho(\{z_1,z_2\})>3$, we have $d_G(z_2)\geqslant 3$. If $d_T(z_2)=2$, then $d_G(z_2)=3$ by Claim \ref{claim: at most one neighbor} and the first assertion holds. Now we assume that $d_T(z_2)\geqslant 3$. Since $P$ is a longest path in $T$, each vertex in $N_T(z_2)\backslash\{z_3\}$ is a leaf of $T$. 

Suppose first that $d_T(z_2)\geqslant 4$. There are three leaves of $T$, say $z_1,z'_1,z''_1$, adjacent to $z_2$. Let $y_1,y'_1,y''_1$ be the neighbors of $z_1,z'_1,z''_1$, respectively, in $C$. By Claim \ref{Cldcx1x2}, every two vertices of $y_1,y'_1,y''_1$ have the same distance in $C$. This holds only when $|C|\in \{3,6\}$. If $|C|=3$, then we can assume w.l.o.g. that $x_1=y_1$, $x_2=y'_1$, $x_3=y''_1$. Recall that $z_p$ is also a leaf of $T$. We assume w.l.o.g. that $x_1$ is the unique neighbor of $z_p$ in $C$. Now $z_2z_1x_1$, $z_2z'_1x_2x_1$, $z_2z''_1x_3x_2x_1$ are three $(z_2,x_1)$-paths internally-disjoint with $P[z_2,z_p]z_px_1$, and of lengths 2, 3, 4, respectively. By Lemma \ref{LePQ1Q2Q3}, there is a $(1\bmod 3)$-cycle, contradicting (1). 

If $|C|=6$, then we can assume w.l.o.g. that $x_1=y_1$, $x_3=y'_1$, $x_5=y''_1$. Let $y_p$ the unique neighbor of $z_p$ in $C$. We assume w.l.o.g. that $y_p\in\{x_1,x_2\}$. If $y_p=x_1$, then $z_2z_1x_1$, $z_2z'_1x_3x_2x_1$, $z_2z'_1x_3x_4x_5x_6x_1$ are three $(z_2,x_1)$-paths internally-disjoint with $P[z_2,z_p]z_px_1$, and of lengths 2, 4, 6, respectively; if $y_p=x_2$, then $z_2z_1x_1x_2$, $z_2z''_1x_5x_6x_1x_2$, $z_2z_1x_1x_6x_5x_4x_3x_2$ are three $(z_2,x_2)$-paths internally-disjoint with $P[z_2,z_p]z_px_2$, and of lengths 3, 5, 7, respectively. By Lemma \ref{LePQ1Q2Q3}, there is a $(1\bmod 3)$-cycle, contradicting (1). 

Suppose now that $d_T(z_2)=3$. 
If $d_C(z_2)=0$, then $d_G(z_2)=3$ and we are done. 
So we assume that $d_C(z_2)=1$. 
Set $N_T(z_2)=\{z_1,z'_1,z_3\}$ and let $y_1,y'_1,y_2$ be the neighbors of $z_1,z'_1,z_2$, respectively, in $C$. 
If $|C|=3$, then $d_C(y_1,y_2)=d_C(y'_1,y_2)=0$, 
implying that $d_C(y_1,y'_1)=0$, a contradiction. 
If $|C|=6$, then $d_C(y_1,y_2)=d_C(y'_1,y_2)=3$, 
implying that $d_C(y_1,y'_1)=0$, a contradiction. 
If $|C|=8$, then $z_2$ has no neighbor in $C$, a contradiction. 
So we conclude that $|C|=5$, and then $d_C(y_1,y_2)=d_C(y'_1,y_2)=2$ and $d_C(y_1,y'_1)=1$. 
We assume w.l.o.g. that $x_1=y_1$, $x_3=y_2$, $x_5=y'_1$, 
see Figure \ref{figure: 5-cycle}. 

\begin{figure}[ht]
\centering
\begin{tikzpicture}[scale=0.6]
\foreach \x in {1,2,...,5}
{\coordinate (x\x) at (162-72*\x:2);
\draw[fill=black](x\x) circle (0.1);
\node at (162-72*\x:1.5) {$x_\x$};}
\draw[thick](x1)--(x2)--(x3)--(x4)--(x5)--(x1); 
\coordinate (z2) at (4,1);
\draw[fill=black](z2){node[right]{$z_2$}} circle (0.1);
\coordinate (zp1) at (4,-1);
\draw[fill=black](zp1){node[right]{$z_{p-1}$}} circle (0.1);
\draw[dotted,thick](z2)--(zp1); 
\coordinate (z1) at (3,2.5); \coordinate (z1') at (5,2.5);
\draw[fill=black](z1){node[right]{$z_1$}} circle (0.1);
\draw[fill=black](z1'){node[right]{$z'_1$}} circle (0.1);
\draw[thick](z1)--(z2)--(z1'); 
\coordinate (zp) at (4,-2.5);
\draw[fill=black](zp){node[right]{$z_p$}} circle (0.1);
\draw[thick](zp1)--(zp); 
\draw(z1)--(x1); 
\path(z1') edge [bend right=60] (x5);
\path(z2) edge [bend left=10] (x3);
\draw(zp1)--(x3)--(zp);
\foreach \x in {1,2,4,5} \draw[fill=white](x\x) circle (0.1);
\end{tikzpicture}

\caption{$|C|=5$, $d_T(z_2)=3$ and $d_C(z_2)=1$.}
\label{figure: 5-cycle}
\end{figure}
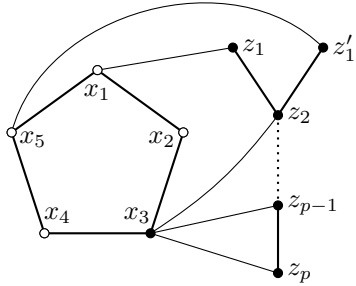

If $x_2$ has a neighbor $z_0$ in $T$, then let $Q$ be the path from $z_2$ to $z_0$ in $T$. Now $z_2x_3x_2$, $z_2z_1x_1x_2$, $z_2z'_1x_5x_1x_2$ are three $(z_2,x_2)$-paths internally-disjoint with $z_2Qz_0x_2$, and of lengths 2, 3, 4, respectively, a contradiction. So we conclude that $x_2$, and similarly, $x_4$, has no neighbors in $T$. If $x_1$ has a neighbor $z_0$ in $V(T)\backslash\{z_1\}$, then let $Q$ be the path from $z_2$ to $z_0$ in $T$. Now $z_2z_1x_1$, $z_2z'_1x_5x_1$, $z_2x_3x_4x_5x_1$ are three $(z_2,z_0)$-paths internally-disjoint with $z_2Qz_0x_1$, and of lengths 2, 3, 4, respectively, a contradiction. So we conclude that $x_1$, and similarly, $x_5$, has no neighbors in $V(T)\backslash\{z_1,z'_1\}$. Thus $N_C(V(T)\backslash\{z_1,z'_1\})=\{x_3\}$. We remark that in Figures \ref{figure: 5-cycle}-\ref{figure: 8-cycle}, hollow points represent vertices of $C$ that have no neighbors in $T$ other than $z_1,z'_1$ or $z_2$. 

If $d_T(z_{p-1})\geqslant 3$, 
then $T$ has at least two leaves adjacent to $z_{p-1}$, say $z_p,z'_p$. 
It follows that $x_3$ is adjacent to $z_p$ and $z'_p$,
contradicting Claim \ref{Cldcx1x2}. 
If $d_T(z_{p-1})=2$, 
then $z_{p-1}$ has a neighbor in $C$. 
It follows that $x_3$ is adjacent to $z_p$ and $z_{p-1}$, contradicting Claim \ref{Cldcx1x2}. 

The second assertion can be proved similarly.
\end{proof}

By Claim \ref{Cldy2}, 
for $z_2$,
either (i) $d_T(z_2)=2$, $d_C(z_2)=1$, 
or (ii) $d_T(z_2)=3$, $d_C(z_2)=0$; 
and for $z_{p-1}$,
either (i) $d_T(z_{p-1})=2$, $d_C(z_{p-1})=1$, 
or (ii) $d_T(z_{p-1})=3$, $d_C(z_{p-1})=0$. 
Let $y_1$ the neighbor of $z_1$ in $C$
and let $y_p$ be the neighbor of $z_p$ in $C$. 
If $d_T(z_2)=2$, 
then let $y_2$ be the neighbor of $z_2$ in $C$; 
and if $d_T(z_2)=3$, 
then let $z'_1$ be the neighbor of $z_2$ in $T$ other than $z_1,z_3$ and let $y'_1$ be the neighbor of $z'_1$ in $C$. 
If $d_T(z_{p-1})=2$, 
then let $y_{p-1}$ be the neighbor of $z_{p-1}$ in $C$; 
and if $d_T(z_{p-1})=3$, 
then let $z'_p$ be the neighbor of $z_{p-1}$ in $T$ other than $z_p,z_{p-2}$ and let $y'_p$ be the neighbor of $z'_p$ in $C$. 
Notice that $V(P)\cup N_T(z_2)\cup N_T(z_{p-1})$ induces a subgraph of four possible types, namely, Type-I, Y, \rotatebox[origin=c]{180}{Y} or X, 
see Figure \ref{figure: IXY}.  

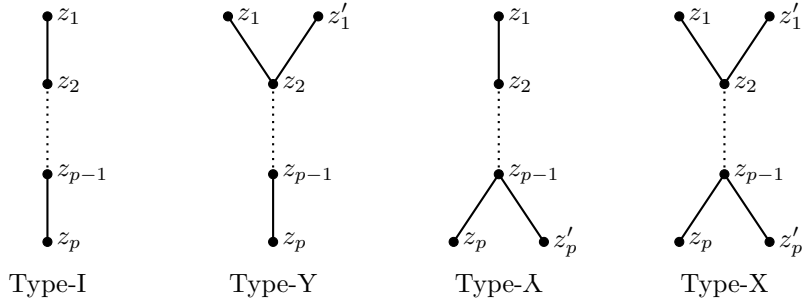
\begin{figure}[ht]
\centering
\begin{tikzpicture}[scale=0.6]

\begin{scope}
\coordinate (z2) at (0,1);
\draw[fill=black](z2){node[right]{$z_2$}} circle (0.1);
\coordinate (zp1) at (0,-1);
\draw[fill=black](zp1){node[right]{$z_{p-1}$}} circle (0.1);
\draw[dotted,thick](z2)--(zp1); 
\coordinate (z1) at (0,2.5);
\draw[fill=black](z1){node[right]{$z_1$}} circle (0.1);
\draw[thick](z1)--(z2); 
\coordinate (zp) at (0,-2.5);
\draw[fill=black](zp){node[right]{$z_p$}} circle (0.1);
\draw[thick](zp1)--(zp); 
\node[below] at (0,-3) {Type-I};
\end{scope}

\begin{scope}[xshift=5cm]
\coordinate (z2) at (0,1);
\draw[fill=black](z2){node[right]{$z_2$}} circle (0.1);
\coordinate (zp1) at (0,-1);
\draw[fill=black](zp1){node[right]{$z_{p-1}$}} circle (0.1);
\draw[dotted,thick](z2)--(zp1); 
\coordinate (z1) at (-1,2.5); \coordinate (z1') at (1,2.5);
\draw[fill=black](z1){node[right]{$z_1$}} circle (0.1);
\draw[fill=black](z1'){node[right]{$z'_1$}} circle (0.1);
\draw[thick](z1)--(z2)--(z1'); 
\coordinate (zp) at (0,-2.5);
\draw[fill=black](zp){node[right]{$z_p$}} circle (0.1);
\draw[thick](zp1)--(zp); 
\node[below] at (0,-3) {Type-Y};
\end{scope}

\begin{scope}[xshift=10cm]
\coordinate (z2) at (0,1);
\draw[fill=black](z2){node[right]{$z_2$}} circle (0.1);
\coordinate (zp1) at (0,-1);
\draw[fill=black](zp1){node[right]{$z_{p-1}$}} circle (0.1);
\draw[dotted,thick](z2)--(zp1); 
\coordinate (z1) at (0,2.5);
\draw[fill=black](z1){node[right]{$z_1$}} circle (0.1);
\draw[thick](z1)--(z2); 
\coordinate (zp) at (-1,-2.5); \coordinate (zp') at (1,-2.5);
\draw[fill=black](zp){node[right]{$z_p$}} circle (0.1);
\draw[fill=black](zp'){node[right]{$z'_p$}} circle (0.1);
\draw[thick](zp)--(zp1)--(zp'); 
\node[below] at (0,-3) {Type-\rotatebox[origin=c]{180}{Y}};
\end{scope}

\begin{scope}[xshift=15cm]
\coordinate (z2) at (0,1);
\draw[fill=black](z2){node[right]{$z_2$}} circle (0.1);
\coordinate (zp1) at (0,-1);
\draw[fill=black](zp1){node[right]{$z_{p-1}$}} circle (0.1);
\draw[dotted,thick](z2)--(zp1); 
\coordinate (z1) at (-1,2.5); \coordinate (z1') at (1,2.5);
\draw[fill=black](z1){node[right]{$z_1$}} circle (0.1);
\draw[fill=black](z1'){node[right]{$z'_1$}} circle (0.1);
\draw[thick](z1)--(z2)--(z1'); 
\coordinate (zp) at (-1,-2.5); \coordinate (zp') at (1,-2.5);
\draw[fill=black](zp){node[right]{$z_p$}} circle (0.1);
\draw[fill=black](zp'){node[right]{$z'_p$}} circle (0.1);
\draw[thick](zp)--(zp1)--(zp');
\node[below] at (0,-3) {Type-X};
\end{scope}

\end{tikzpicture}
\caption{Four types of $G[V(P)\cup N_T(z_2)\cup N_T(z_{p-1})]$.}\label{figure: IXY}
\end{figure}

Now we distinguish the following cases according to the length of $C$. 

\setcounter{case}{0}
\begin{case}
$|C|=3$.  
\end{case}

\begin{subcase}
$d_T(z_2)=3$.
\end{subcase}

Recall that $d_C(y_1,y'_1)=1$, 
we can assume w.l.o.g. that $y_1=x_1$, $y'_1=x_2$, 
see Figure \ref{figure: 3-cycle} (i).
If $x_1$ has a neighbor $z_0$ in $T$ other than $z_1$, 
then let $Q$ be the path from $z_2$ to $z_0$ in $T$. 
Now $z_2z_1x_1$, $z_2z'_1x_2x_1$, $z_2z'_1x_2x_3x_1$ are three $(z_2,x_1)$-paths internally-disjoint with $Qz_0x_1$, 
and of lengths 2, 3, 4, respectively, a contradiction. 
So we conclude that $x_1$, and similarly, $x_2$,
has no neighbors in $V(T)\backslash\{z_1,z'_1\}$. 
It follows that $N_C(V(T)\backslash\{z_1,z'_1\})=\{x_3\}$.

\begin{figure}[ht]
\centering
\begin{tikzpicture}[scale=0.6]
\begin{scope}
\foreach \x in {1,2,3}
{\coordinate (x\x) at (210-120*\x:2);
\draw[fill=black](x\x) circle (0.1);
\node at (210-120*\x:1.4) {$x_\x$};}
\draw[thick](x1)--(x2)--(x3)--(x1); 

\coordinate (z2) at (4,1);
\draw[fill=black](z2){node[right]{$z_2$}} circle (0.1);
\coordinate (zp1) at (4,-1);
\draw[fill=black](zp1){node[right]{$z_{p-1}$}} circle (0.1);
\draw[dotted,thick](z2)--(zp1); 
\coordinate (z1) at (3,2.5); \coordinate (z1') at (5,2.5);
\draw[fill=black](z1){node[right]{$z_1$}} circle (0.1);
\draw[fill=black](z1'){node[right]{$z'_1$}} circle (0.1);
\draw[thick](z1)--(z2)--(z1'); 
\coordinate (zp) at (4,-2.5);
\draw[fill=black](zp){node[right]{$z_p$}} circle (0.1);
\draw[thick](zp1)--(zp); 
\draw(z1)--(x1);
\draw(z1').. controls (2.5,4.5) and (1,1.5)..(x2);
\path(zp1) edge [bend left=30] (x3);
\path(zp) edge [bend left=30] (x3);
\foreach \x in {1,2} \draw[fill=white](x\x) circle (0.1);
\node[below] at (2,-3){(i) $d_T(z_2)=3$.};
\end{scope}

\begin{scope}[xshift=10cm]
\foreach \x in {1,2,3}
{\coordinate (x\x) at (210-120*\x:2);
\draw[fill=black](x\x) circle (0.1);
\node at (210-120*\x:1.4) {$x_\x$};}
\draw[thick](x1)--(x2)--(x3)--(x1); 
\coordinate (z2) at (4,1);
\draw[fill=black](z2){node[right]{$z_2$}} circle (0.1);
\coordinate (zp1) at (4,-1);
\draw[fill=black](zp1){node[right]{$z_{p-1}$}} circle (0.1);
\draw[dotted,thick](z2)--(zp1); 
\coordinate (z1) at (4,2.5);
\draw[fill=black](z1){node[right]{$z_1$}} circle (0.1);
\draw[thick](z1)--(z2); 
\coordinate (zp) at (4,-2.5);
\draw[fill=black](zp){node[right]{$z_p$}} circle (0.1);
\draw[thick](zp1)--(zp); 
\draw(z1)--(x1)--(z2);
\path(zp1) edge [bend right=10] (x1);
\path(zp) edge [bend right=10] (x1);
\foreach \x in {2,3} \draw[fill=white](x\x) circle (0.1);
\node[below] at (1.5,-3){(ii) $d_T(z_2)=2$.};
\end{scope}
\end{tikzpicture}

\caption{$|C|=3$.}
\label{figure: 3-cycle}
\end{figure}

If $d_T(z_{p-1})=3$, then $y_p=y'_p=x_3$, contradicting Claim \ref{Cldcx1x2}. So we conclude that $d_T(z_{p-1})=2$, and thus $y_{p-1}=y_p=x_3$. Now $z_2z_1x_1x_3z_{p-1}$, $z_2z_1x_1x_2x_3z_{p-1}$, $z_2z_1x_1x_2x_3z_pz_{p-1}$ are three $(z_2,z_{p-1})$-paths internally-disjoint with $P[z_2,z_{p-1}]$, and of lengths 4, 5, 6, respectively, a contradiction.

\begin{subcase}
$d_T(z_2)=2$.
\end{subcase}

Recall that $d_C(y_1,y_2)=0$, we can assume w.l.o.g. that $y_1=y_2=x_1$,
see Figure \ref{figure: 3-cycle} (ii).
If $x_2$ has a neighbor $z_0$ in $T$, then let $Q$ be the path from $z_2$ to $z_0$ in $T$. Note that by Claim \ref{claim: at most one neighbor} we have $z_0\notin \{z_1,z_2\}$. Now $z_2x_1x_2$, $z_2z_1x_1x_2$, $z_2z_1x_1x_3x_2$ are three $(z_2,x_2)$-paths internally-disjoint with $Qz_0x_2$, and of lengths 2, 3, 4, respectively, a contradiction. So we conclude that $x_2$, and similarly, $x_3$, has no neighbors in $T$. It follows that $N_C(T)=\{x_1\}$, and $x_1$ is a cut-vertex of $G$, a contradiction to the assumption that $G$ is 2-connected.

\begin{case}
$|C|=5$.  
\end{case}

\begin{subcase}\label{Ca2.1}
$d_T(z_2)=3$.   
\end{subcase}

Recall that $d_C(y_1,y'_1)=1$, we can assume w.l.o.g. that $y_1=x_1,~y'_1=x_2$,
see Figure \ref{figure: 5-cycle-(1)}.
If $x_3$ has a neighbor $z_0$ in $T$, then let $Q$ be the path from $z_2$ to $z_0$ in $T$. Now $z_2z'_1x_2x_3$, $z_2z_1x_1x_2x_3$, $z_2z_1x_1x_5x_4x_3$ are three $(z_2,x_3)$-paths internally-disjoint with $Qz_0x_3$, and of lengths 3, 4, 5, respectively, a contradiction. So we conclude that $x_3$, and similarly, $x_5$, has no neighbors in $V(T)$. It follows that $N_C(T)\subseteq\{x_1,x_2,x_4\}$.

\begin{figure}[ht]
\centering
\begin{tikzpicture}[scale=0.6]
\begin{scope}
\foreach \x in {1,2,...,5}
{\coordinate (x\x) at (162-72*\x:2);
\draw[fill=black](x\x) circle (0.1);
\node at (162-72*\x:1.5) {$x_\x$};}
\draw[thick](x1)--(x2)--(x3)--(x4)--(x5)--(x1); 
\coordinate (z2) at (4,1);
\draw[fill=black](z2){node[right]{$z_2$}} circle (0.1);
\coordinate (zp1) at (4,-1);
\draw[fill=black](zp1){node[right]{$z_{p-1}$}} circle (0.1);
\draw[dotted,thick](z2)--(zp1); 
\coordinate (z1) at (3,2.5); \coordinate (z1') at (5,2.5);
\draw[fill=black](z1){node[right]{$z_1$}} circle (0.1);
\draw[fill=black](z1'){node[right]{$z'_1$}} circle (0.1);
\draw[thick](z1)--(z2)--(z1'); 
\coordinate (zp) at (3,-2.5); \coordinate (zp') at (5,-2.5);
\draw[fill=black](zp){node[right]{$z_p$}} circle (0.1);
\draw[fill=black](zp'){node[right]{$z'_p$}} circle (0.1);
\draw[thick](zp)--(zp1)--(zp'); 
\draw(z1)--(x1);
\draw(z1').. controls (2.5,4) and (2,3.5)..(x2);
\draw(zp).. controls (3.5,0) and (2.5,2)..(x1);
\draw(zp').. controls (1.5,-5) and (2.5,-1)..(x2);
\foreach \x in {3,5} \draw[fill=white](x\x) circle (0.1);
\node[below] at (2,-3.5){(i) $d_T(z_{p-1})=2$};
\end{scope}

\begin{scope}[xshift=10cm]
\foreach \x in {1,2,...,5}
{\coordinate (x\x) at (162-72*\x:2);
\draw[fill=black](x\x) circle (0.1);
\node at (162-72*\x:1.5) {$x_\x$};}
\draw[thick](x1)--(x2)--(x3)--(x4)--(x5)--(x1); 
\coordinate (z2) at (4,1);
\draw[fill=black](z2){node[right]{$z_2$}} circle (0.1);
\coordinate (zp1) at (4,-1);
\draw[fill=black](zp1){node[right]{$z_{p-1}$}} circle (0.1);
\draw[dotted,thick](z2)--(zp1); 
\coordinate (z1) at (3,2.5); \coordinate (z1') at (5,2.5);
\draw[fill=black](z1){node[right]{$z_1$}} circle (0.1);
\draw[fill=black](z1'){node[right]{$z'_1$}} circle (0.1);
\draw[thick](z1)--(z2)--(z1'); 
\coordinate (zp) at (4,-2.5);
\draw[fill=black](zp){node[right]{$z_p$}} circle (0.1);
\draw[thick](zp1)--(zp); 
\draw(z1)--(x1);
\draw(z1').. controls (2.5,4) and (2,3.5)..(x2);
\path(zp1) edge [bend left=50] (x4);
\draw(zp).. controls (3,-1.5) and (3,2) ..(x1);
\foreach \x in {3,5} \draw[fill=white](x\x) circle (0.1);
\node[below] at (2,-3.5){(ii) $d_T(z_{p-1})=2$, $y_{p-1}=x_4$};
\end{scope}

\begin{scope}[xshift=20cm]
\foreach \x in {1,2,...,5}
{\coordinate (x\x) at (162-72*\x:2);
\draw[fill=black](x\x) circle (0.1);
\node at (162-72*\x:1.5) {$x_\x$};}
\draw[thick](x1)--(x2)--(x3)--(x4)--(x5)--(x1); 
\coordinate (z2) at (4,1);
\draw[fill=black](z2){node[right]{$z_2$}} circle (0.1);
\coordinate (zp1) at (4,-1);
\draw[fill=black](zp1){node[right]{$z_{p-1}$}} circle (0.1);
\draw[dotted,thick](z2)--(zp1); 
\coordinate (z1) at (3,2.5); \coordinate (z1') at (5,2.5);
\draw[fill=black](z1){node[right]{$z_1$}} circle (0.1);
\draw[fill=black](z1'){node[right]{$z'_1$}} circle (0.1);
\draw[thick](z1)--(z2)--(z1'); 
\coordinate (zp) at (4,-2.5);
\draw[fill=black](zp){node[right]{$z_p$}} circle (0.1);
\draw[thick](zp1)--(zp); 
\draw(z1)--(x1);
\draw(z1').. controls (2.5,4) and (2,3.5)..(x2);
\path(zp1) edge [bend right=30] (x1);
\path(zp) edge [bend left=30] (x4);
\foreach \x in {3,5} \draw[fill=white](x\x) circle (0.1);
\node[below] at (2,-3.5){(iii) $d_T(z_{p-1})=2$, $y_{p-1}=x_1$};
\end{scope}

\end{tikzpicture}

\caption{$|C|=5$ and $d_T(z_2)=3$.}\label{figure: 5-cycle-(1)}
\end{figure}
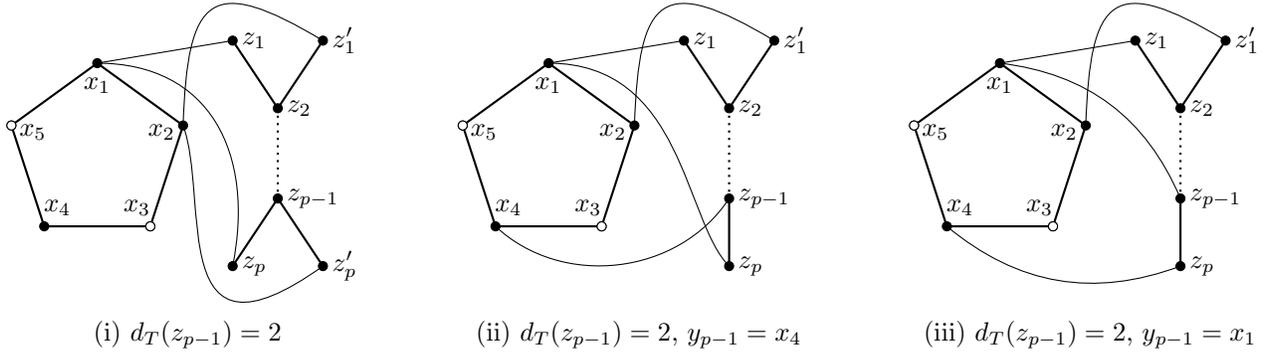


\begin{subsubcase}
$d_T(z_{p-1})=3$.
\end{subsubcase}

Recall that $d_C(y_p,y'_p)=1$ and $N_C(T)\subseteq\{x_1,x_2,x_4\}$. We have that $\{y_p,y'_p\}=\{x_1,x_2\}$, say $y_p=x_1,y'_p=x_2$,
see Figure \ref{figure: 5-cycle-(1)} (i). 
Now $z_2z_1x_1z_pz_{p-1}$, $z_2z'_1x_2x_1z_pz_{p-1}$ are two $(z_2,z_{p-1})$-paths internally-disjoint with $P[z_2,z_{p-1}]$, and of lengths 4, 5, respectively. It follows that $|P[z_2,z_{p-1}]|\equiv 1\bmod 3$ and $|P|\equiv 0\bmod 3$. If $|P|=3$, then $|V(T)|=6$ and $\rho(V(T))=9$, contradicting (3). So we conclude that $|P|\geqslant 6$.

If $d_G(z_3)=2$, then $\rho(\{z_1,z'_1,z_2,z_3\})=6$, a contradiction. 
Thus we assume that $d_G(z_3)\geqslant 3$. 
It follows that there exists a path $Q$ from $z_3$ to $C$ with all internal vertices not in $P$. Let $y_0$ be the terminus of $Q$. Recall that $y_0\in\{x_1,x_2,x_4\}$. 

If $y_0=x_1$, 
then $P[z_3,z_p]z_px_1$, $z_3z_2z_1x_1$, $z_3z_2z'_1x_2x_1$ are three $(z_3,x_1)$-paths internally-disjoint with $Q$, and of lengths $|P[z_3,z_p]|+1$, 3, 4, respectively, 
which are different modulo 3 (notice that $|P[z_3,z_p]|\equiv 1\bmod 3$), a contradiction. 
If $y_0=x_2$, 
then we can prove similarly. 
Now assume that $y_0=x_4$. 
Now $P[z_3,z_p]z_px_1x_5x_4$, $z_3z_2z_1x_1x_5x_4$, $z_3z_2z_1x_1x_2x_3x_4$ are three $(z_3,x_4)$-paths internally-disjoint with $Q$, and of lengths $|P[z_3,z_p]|+3$, 5, 6, respectively, which are different modulo 3, a contradiction.


\begin{subsubcase}
$d_T(z_{p-1})=2$.
\end{subsubcase}

Recall that $d_C(y_{p-1},y_p)=2$ and $N_C(T)\subseteq\{x_1,x_2,x_4\}$. 
We can assume w.l.o.g. that 
$\{y_{p-1},y_p\}=\{x_1,x_4\}$. 
If $y_{p-1}=x_4$, $y_p=x_1$, see Figure \ref{figure: 5-cycle-(1)} (ii), 
then $z_2z_1x_1z_pz_{p-1}$, $z_2z'_1x_2x_1z_pz_{p-1}$, $z_2z_1x_1x_2x_3x_4z_{p-1}$ 
are three $(z_2,z_{p-1})$-paths internally-disjoint with $P[z_2,z_{p-1}]$, 
and of lengths 4, 5, 6, respectively, a contradiction. 
So we conclude that $y_{p-1}=x_1,y_p=x_4$,
see Figure \ref{figure: 5-cycle-(1)} (iii). 

Now $z_2z_1x_1z_{p-1}$, $z_2z'_1x_2x_1z_{p-1}$ are two $(z_2,z_{p-1})$-paths internally-disjoint with $P[z_2,z_{p-1}]$, and of lengths 3, 4, respectively. 
It follows that $|P[z_2,z_{p-1}]|\equiv 2\bmod 3$ and $|P|\equiv 1\bmod 3$. 
By (3), we have that $d_G(z_3)\geqslant 3$. 
It follows that there exists a path $Q$ from $z_3$ to $C$ 
with all internal vertices not in $P$. 
Let $y_0$ be the terminus of $Q$. 
Then $y_0\in\{x_1,x_2,x_4\}$. 

If $y_0=x_1$, 
then $P[z_3,z_{p-1}]z_{p-1}x_1$, $z_3z_2z_1x_1$, $z_3z_2z'_1x_2x_1$ are three $(z_3,x_1)$-paths internally-disjoint with $Q$, and of lengths $|P[z_3,z_{p-1}]|+1$, 3, 4, respectively, which are different modulo 3 (notice that $|P[z_3,z_{p-1}]|\equiv 1\bmod 3$), a contradiction. 
If $y_0=x_2$, 
then $x_2z'_1z_2z_3Qx_2$ and $x_1z_{p-1}z_px_4x_5x_1$ are two disjoint cycles, contradicting (2). 
If $y_0=x_4$, 
then $x_1z_1z_2z'_1x_2x_1$ and $x_4z_pP[z_p,z_3]z_3Qx_4$ are two disjoint cycles, a contradiction.

\begin{subcase}
$d_T(z_2)=2$.  
\end{subcase}

Recall that $d_C(y_1,y_2)=2$. 
We can assume w.l.o.g. that $y_1=x_1$, $y_2=x_3$,
see Figure \ref{figure: 5-cycle-(2)}. 
If $x_1$ has a neighbor $z_0$ in $T$ other than $z_1$, then let $Q$ be the path from $z_2$ to $z_0$ in $T$. Now $z_2z_1x_1$, $z_2x_3x_2x_1$, $z_2x_3x_4x_5x_1$ are three $(z_2,x_1)$-paths internally-disjoint with $Qz_0x_1$, and of lengths 2, 3, 4, respectively, a contradiction. So we conclude that $x_1$ has no neighbors in $V(T)\backslash\{z_1\}$. 

\begin{figure}[ht]
\centering
\begin{tikzpicture}[scale=0.6]
\begin{scope}
\foreach \x in {1,2,...,5}
{\coordinate (x\x) at (162-72*\x:2);
\draw[fill=black](x\x) circle (0.1);
\node at (162-72*\x:1.5) {$x_\x$};}
\draw[thick](x1)--(x2)--(x3)--(x4)--(x5)--(x1); 
\coordinate (z2) at (4,1);
\draw[fill=black](z2){node[right]{$z_2$}} circle (0.1);
\coordinate (zp1) at (4,-1);
\draw[fill=black](zp1){node[right]{$z_{p-1}$}} circle (0.1);
\draw[dotted,thick](z2)--(zp1); 
\coordinate (z1) at (4,2.5);
\draw[fill=black](z1){node[right]{$z_1$}} circle (0.1);
\draw[thick](z1)--(z2); 
\coordinate (zp) at (4,-2.5);
\draw[fill=black](zp){node[right]{$z_p$}} circle (0.1);
\draw[thick](zp1)--(zp); 
\draw(z1)--(x1);
\path(z2) edge [bend left=20] (x3);
\path(zp1) edge [bend left=20] (x3);
\draw(zp)..controls (0,-3) and (-3,-3)..(x5);
\draw[fill=white](x1) circle (0.1);
\node[below] at (1.5,-3){(i) $y_{p-1}=x_3$, $y_p=x_5$};
\end{scope}

\begin{scope}[xshift=10cm]
\foreach \x in {1,2,...,5}
{\coordinate (x\x) at (162-72*\x:2);
\draw[fill=black](x\x) circle (0.1);
\node at (162-72*\x:1.5) {$x_\x$};}
\draw[thick](x1)--(x2)--(x3)--(x4)--(x5)--(x1); 
\coordinate (z2) at (4,1);
\draw[fill=black](z2){node[right]{$z_2$}} circle (0.1);
\coordinate (zp1) at (4,-1);
\draw[fill=black](zp1){node[right]{$z_{p-1}$}} circle (0.1);
\draw[dotted,thick](z2)--(zp1); 
\coordinate (z1) at (4,2.5);
\draw[fill=black](z1){node[right]{$z_1$}} circle (0.1);
\draw[thick](z1)--(z2); 
\coordinate (zp) at (4,-2.5);
\draw[fill=black](zp){node[right]{$z_p$}} circle (0.1);
\draw[thick](zp1)--(zp); 
\draw(z1)--(x1);
\path(z2) edge [bend left=20] (x3);
\draw(zp1)..controls (-1,-4) and (-3,-2)..(x5);
\draw(zp)--(x3);
\draw[fill=white](x1) circle (0.1);
\node[below] at (1.5,-3){(ii) $y_{p-1}=x_5$, $y_p=x_3$};
\end{scope}

\begin{scope}[xshift=20cm]
\foreach \x in {1,2,...,5}
{\coordinate (x\x) at (162-72*\x:2);
\draw[fill=black](x\x) circle (0.1);
\node at (162-72*\x:1.5) {$x_\x$};}
\draw[thick](x1)--(x2)--(x3)--(x4)--(x5)--(x1); 
\coordinate (z2) at (4,1);
\draw[fill=black](z2){node[right]{$z_2$}} circle (0.1);
\coordinate (zp1) at (4,-1);
\draw[fill=black](zp1){node[right]{$z_{p-1}$}} circle (0.1);
\draw[dotted,thick](z2)--(zp1); 
\coordinate (z1) at (4,2.5);
\draw[fill=black](z1){node[right]{$z_1$}} circle (0.1);
\draw[thick](z1)--(z2); 
\coordinate (zp) at (4,-2.5);
\draw[fill=black](zp){node[right]{$z_p$}} circle (0.1);
\draw[thick](zp1)--(zp); 
\draw(z1)--(x1);
\path(z2) edge [bend left=20] (x3);
\draw(zp1)--(x2);
\draw(zp)..controls (0,-3) and (-3,-3)..(x5);
\draw[fill=white](x1) circle (0.1);
\node[below] at (1.5,-3){(iii) $y_{p-1}=x_2$, $y_p=x_5$};
\end{scope}

\begin{scope}[yshift=-7.5cm]
\foreach \x in {1,2,...,5}
{\coordinate (x\x) at (162-72*\x:2);
\draw[fill=black](x\x) circle (0.1);
\node at (162-72*\x:1.5) {$x_\x$};}
\draw[thick](x1)--(x2)--(x3)--(x4)--(x5)--(x1); 
\coordinate (z2) at (4,1);
\draw[fill=black](z2){node[right]{$z_2$}} circle (0.1);
\coordinate (zp1) at (4,-1);
\draw[fill=black](zp1){node[right]{$z_{p-1}$}} circle (0.1);
\draw[dotted,thick](z2)--(zp1); 
\coordinate (z1) at (4,2.5);
\draw[fill=black](z1){node[right]{$z_1$}} circle (0.1);
\draw[thick](z1)--(z2); 
\coordinate (zp) at (4,-2.5);
\draw[fill=black](zp){node[right]{$z_p$}} circle (0.1);
\draw[thick](zp1)--(zp); 
\draw(z1)--(x1);
\path(z2) edge [bend left=20] (x3);
\path(zp1) edge [bend left=40] (x4);
\path(zp) edge [bend left=20] (x2);
\draw[fill=white](x1) circle (0.1);
\node[below] at (1.5,-3){(iv) $y_{p-1}=x_4$, $y_p=x_2$};
\end{scope}

\begin{scope}[xshift=10cm,yshift=-7.5cm]
\foreach \x in {1,2,...,5}
{\coordinate (x\x) at (162-72*\x:2);
\draw[fill=black](x\x) circle (0.1);
\node at (162-72*\x:1.5) {$x_\x$};}
\draw[thick](x1)--(x2)--(x3)--(x4)--(x5)--(x1); 
\coordinate (z2) at (4,1);
\draw[fill=black](z2){node[right]{$z_2$}} circle (0.1);
\coordinate (zp1) at (4,-1);
\draw[fill=black](zp1){node[right]{$z_{p-1}$}} circle (0.1);
\draw[dotted,thick](z2)--(zp1); 
\coordinate (z1) at (4,2.5);
\draw[fill=black](z1){node[right]{$z_1$}} circle (0.1);
\draw[thick](z1)--(z2); 
\coordinate (zp) at (4,-2.5);
\draw[fill=black](zp){node[right]{$z_p$}} circle (0.1);
\draw[thick](zp1)--(zp); 
\draw(z1)--(x1);
\path(z2) edge [bend left=20] (x3);
\draw(zp1)..controls (-1,-4) and (-3,-2)..(x5);
\path(zp) edge [bend left=20] (x2);
\draw[fill=white](x1) circle (0.1);
\node[below] at (1.5,-3){(v) $y_{p-1}=x_5$, $y_p=x_2$};
\end{scope}

\begin{scope}[xshift=20cm,yshift=-7.5cm]
\foreach \x in {1,2,...,5}
{\coordinate (x\x) at (162-72*\x:2);
\draw[fill=black](x\x) circle (0.1);
\node at (162-72*\x:1.5) {$x_\x$};}
\draw[thick](x1)--(x2)--(x3)--(x4)--(x5)--(x1); 
\coordinate (z2) at (4,1);
\draw[fill=black](z2){node[right]{$z_2$}} circle (0.1);
\coordinate (zp1) at (4,-1);
\draw[fill=black](zp1){node[right]{$z_{p-1}$}} circle (0.1);
\draw[dotted,thick](z2)--(zp1); 
\coordinate (z1) at (4,2.5);
\draw[fill=black](z1){node[right]{$z_1$}} circle (0.1);
\draw[thick](z1)--(z2); 
\coordinate (zp) at (4,-2.5);
\draw[fill=black](zp){node[right]{$z_p$}} circle (0.1);
\draw[thick](zp1)--(zp); 
\draw(z1)--(x1);
\path(z2) edge [bend left=20] (x3);
\draw(zp1)--(x2);
\path(zp) edge [bend left=20] (x4);
\draw[fill=white](x1) circle (0.1);
\node[below] at (1.5,-3){(vi) $y_{p-1}=x_2$, $y_p=x_4$};
\end{scope}

\end{tikzpicture}
\caption{$|C|=5$ and $d_T(z_2)=2$.}
\label{figure: 5-cycle-(2)}
\end{figure}
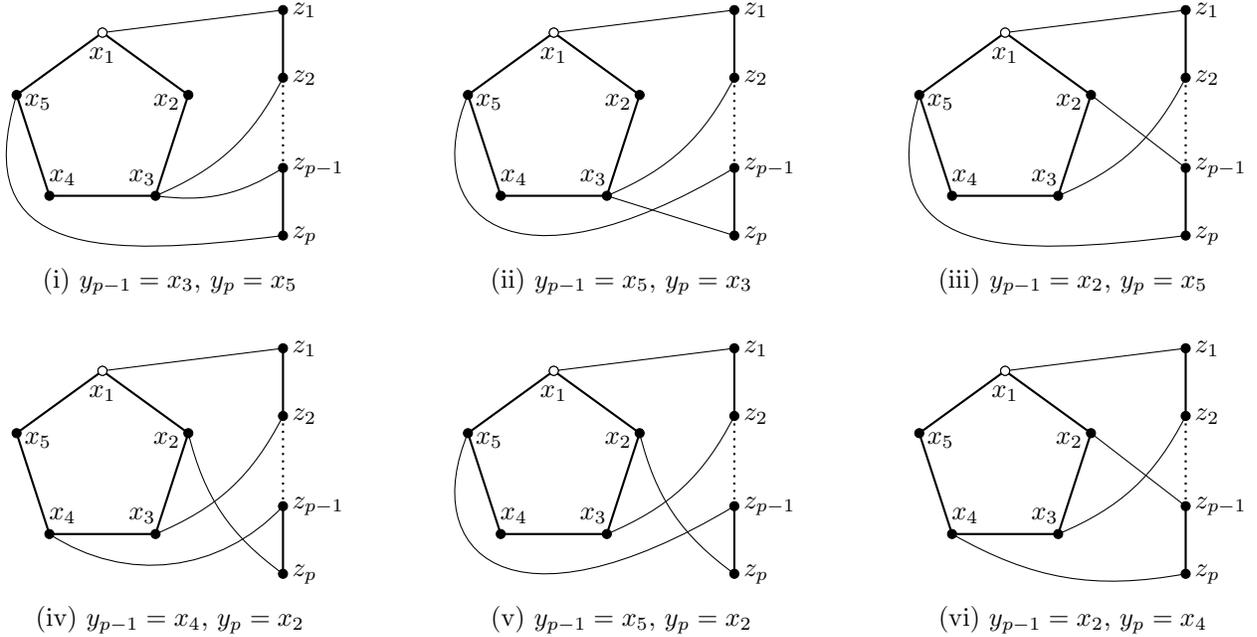

If $d_T(z_{p-1})=3$, then we can prove similarly as in Case \ref{Ca2.1}. So we assume that $d_T(z_{p-1})=2$. Recall that $d_C(y_{p-1},y_p)=2$. If $\{y_{p-1},y_p\}\cap\{y_1,y_2\}\neq\emptyset$, then $\{y_{p-1},y_p\}=\{x_3,x_5\}$.
If $y_{p-1}=x_3$ and $y_p=x_5$, 
then $x_1z_1z_2x_3z_{p-1}z_px_5x_1$ is a 7-cycle, see Figure \ref{figure: 5-cycle-(2)} (i); 
if $y_{p-1}=x_5$ and $y_p=x_3$, 
then $x_1z_1z_2x_3z_pz_{p-1}x_5x_1$ is a 7-cycle see Figure \ref{figure: 5-cycle-(2)} (ii), 
both contradicting (1). 
Thus $\{y_{p-1},y_p\}\cap\{y_1,y_2\}=\emptyset$
and $\{y_{p-1},y_p\}\in \{\{x_2,x_4\},\{x_2,x_5\}\}$.

If $y_{p-1}=x_2$, $y_p=x_5$, see Figure \ref{figure: 5-cycle-(2)} (iii), then $z_2x_3x_2z_{p-1}$, $z_2z_1x_1x_2z_{p-1}$, $z_2z_1x_1x_5z_pz_{p-1}$ are three $(z_2,z_{p-1})$-paths internally-disjoint with $P[z_2,z_{p-1}]$, and of lengths 3, 4, 5, respectively, a contradiction. If $y_{p-1}=x_4$, $y_p=x_2$, see Figure \ref{figure: 5-cycle-(2)} (iv), then we can prove similarly. So we conclude that either $y_{p-1}=x_5$, $y_p=x_2$, or $y_{p-1}=x_2$, $y_p=x_4$.

Suppose first that $y_{p-1}=x_5$, $y_p=x_2$, see Figure \ref{figure: 5-cycle-(2)} (v). Then $z_2x_3x_2z_pz_{p-1}$, $z_2z_1x_1x_2z_pz_{p-1}$ are two $(z_2,z_{p-1})$-paths internally-disjoint with $P[z_2,z_{p-1}]$, and of lengths 4, 5, respectively. It follows that $|P[z_2,z_{p-1}]|\equiv 1\bmod 3$ and $|P|\equiv 0\bmod 3$. We claim that there are no edges between $V(T)\backslash\{z_1,z_2,z_{p-1},z_p\}$ and $C$. Suppose otherwise. There will be a path $Q$ from $P(z_2,z_{p-1})$ to $C$ with all internal vertices not in $P$. Let $z_0,y_0$ be the origin and terminus of $Q$, respectively. Recall that $y_0\neq x_1$, and similarly, $y_0\neq x_2$. We can assume w.l.o.g. that $y_0\in\{x_3,x_4\}$. If $y_0=x_3$, then $x_3z_2P[z_2,z]zQx_3$ and $x_2z_pz_{p-1}x_5x_1x_2$ are two disjoint cycles, a contradiction. If $y_0=x_4$, then $x_3z_2P[z_2,z]zQx_4x_3$ and $x_2z_pz_{p-1}x_5x_1x_2$ are two disjoint cycles, a contradiction. Thus as we claimed, the edges $z_iy_i$ with $i\in \{1,2,p-1,p\}$ are the only edges between $P$ and $C$. It follows that $T=P$ is a path. Recall that $|P|\equiv 0\bmod 3$. If $|P|=3$, then $G[V(C)\cup V(T)]$ is isomorphic to $L_2$, as desired. If $|P|\geqslant 6$, then $\rho(V(T))=|P|+4\leqslant \frac{3}{2}|V(T)|$, a contradiction.

Suppose now that $y_{p-1}=x_2$, $y_p=x_4$, see Figure \ref{figure: 5-cycle-(2)} (vi). Then $z_2x_3x_2z_{p-1}$, $z_2z_1x_1x_2z_{p-1}$ are two $(z_2,z_{p-1})$-paths internally-disjoint with $P[z_2,z_{p-1}]$, and of lengths 3, 4, respectively. It follows that $|P[z_2,z_{p-1}]|\equiv 2\bmod 3$ and $|P|\equiv 1\bmod 3$. We claim that there are no edges between $V(T)\backslash\{z_1,z_2,z_{p-1},z_p\}$ and $C$. Suppose otherwise. There will be a path $Q$ from $P(z_2,z_{p-1})$ to $C$ with all internal vertices not in $P$. Let $z_0,y_0$ be the origin and terminus of $Q$, respectively. Recall that $y_0\neq x_1$, and similarly, $y_0\neq x_4$. We can assume w.l.o.g. that $y_0\in\{x_2,x_5\}$. If $y_0=x_2$, then $x_1z_1z_2x_3x_4x_5x_1$ and $x_2z_{p-1}P[z_{p-1},z]zQx_2$ are two disjoint cycles, a contradiction. If $y_0=x_5$, then $x_1z_1P[z_1,z]zQx_5x_1$ and $x_2z_{p-1}z_px_4x_3x_2$ are two disjoint cycles, a contradiction. Thus as we claimed, the edges $z_iy_i$ with $i\in \{1,2,p-1,p\}$ are the only edges between $P$ and $C$. It follows that $T=P$ is a path. Recall that $|P|\equiv 1\bmod 3$. If $|P|=4$, then $G[V(C)\cup V(T)]$ is isomorphic to $L_3$, as desired. If $|P|\geqslant 7$, then $\rho(V(T))=|P|+4\leqslant \frac{3}{2}|V(T)|$, a contradiction.

\begin{case}
$|C|=6$.  
\end{case}

\begin{subcase}\label{Ca3.1}
$d_T(z_2)=3$.
\end{subcase}

Recall that $d_C(y_1,y'_1)=2$, we can assume w.l.o.g. that $y_1=x_1$, $y'_1=x_3$.  
If $x_4$ has a neighbor $z_0$ in $T$, then let $Q$ be the path from $z_2$ to $z_0$ in $T$. Now $z_2z'_1x_3x_4$, $z_2z_1x_1x_2x_3x_4$, $z_2z'_1x_3x_2x_1x_6x_5x_4$ are three $(z_2,x_4)$-paths internally-disjoint with $Qz_0x_4$, and of lengths 3, 5, 7, respectively, a contradiction. So we conclude that $x_4$, and similarly, $x_6$, has no neighbors in $T$. If $x_1$ has a neighbor $z_0$ in $T$ other than $z_1$, then let $Q$ be the path from $z_2$ to $z_0$ in $T$. Now $z_2z_1x_1$, $z_2z'_1x_3x_2x_1$, $z_2z'_1x_3x_4x_5x_6x_1$ are three $(z_2,x_1)$-paths internally-disjoint with $Qz_0x_1$, and of lengths 2, 4, 6, respectively, a contradiction. So we conclude that $x_1$, and similarly, $x_3$, has no neighbors in $V(T)\backslash\{z_1,z'_1\}$. It follows that $N_C(V(T)\backslash\{z_1,z'_1\})\subseteq\{x_2,x_5\}$.

If $d_T(z_{p-1})=3$, then $\{y_p,y'_p\}\subseteq\{x_2,x_5\}$, implying that $d_C(y_p,y'_p)\in \{0,3\}$, contradicting Claim \ref{Cldcx1x2}. So we have $d_T(z_{p-1})=2$. Since $d_C(y_{p-1},y_p)=3$, we have $\{y_{p-1},y_p\}=\{x_2,x_5\}$. If $y_{p-1}=x_2$ and $y_p=x_5$, see Figure \ref{figure: 6-cycle} (i), then $x_1z_1z_2z'_1x_3x_2z_{p-1}z_px_5x_6x_1$ is a 10-cycle, contradicting (1). Similarly, if $y_{p-1}=x_5$ and $y_p=x_2$, then there is a 10-cycle, a contradiction again.

\begin{figure}[ht]
\centering
\begin{tikzpicture}[scale=0.6]
\begin{scope}
\foreach \x in {1,2,...,6}
{\coordinate (x\x) at (150-60*\x:2);
\draw[fill=black](x\x) circle (0.1);
\node at (150-60*\x:1.5) {$x_\x$};}
\draw[thick](x1)--(x2)--(x3)--(x4)--(x5)--(x6)--(x1); 
\coordinate (z2) at (4,1);
\draw[fill=black](z2){node[right]{$z_2$}} circle (0.1);
\coordinate (zp1) at (4,-1);
\draw[fill=black](zp1){node[right]{$z_{p-1}$}} circle (0.1);
\draw[dotted,thick](z2)--(zp1); 
\coordinate (z1) at (3,2.5); \coordinate (z1') at (5,2.5);
\draw[fill=black](z1){node[right]{$z_1$}} circle (0.1);
\draw[fill=black](z1'){node[right]{$z'_1$}} circle (0.1);
\draw[thick](z1)--(z2)--(z1'); 
\coordinate (zp) at (4,-2.5);
\draw[fill=black](zp){node[right]{$z_p$}} circle (0.1);
\draw[thick](zp1)--(zp); 
\path(z1) edge [bend right=10] (x1);
\draw(z1')..controls (3,4) and (2,3)..(2.5,1.5); \draw(2.5,1.5)..controls (3,0) and (2,-1)..(x3);

\path(zp1) edge [bend left=10] (x2);
\path(zp) edge [bend left=50] (x5);
\foreach \x in {1,3,4,6} \draw[fill=white](x\x) circle (0.1);
\node[below] at (2,-3.5){(i) $d_T(x_2)=3$};
\end{scope}

\begin{scope}[xshift=10cm]
\foreach \x in {1,2,...,6}
{\coordinate (x\x) at (150-60*\x:2);
\draw[fill=black](x\x) circle (0.1);
\node at (150-60*\x:1.5) {$x_\x$};}
\draw[thick](x1)--(x2)--(x3)--(x4)--(x5)--(x6)--(x1); 
\coordinate (z2) at (4,1);
\draw[fill=black](z2){node[right]{$z_2$}} circle (0.1);
\coordinate (zp1) at (4,-1);
\draw[fill=black](zp1){node[right]{$z_{p-1}$}} circle (0.1);
\draw[dotted,thick](z2)--(zp1); 
\coordinate (z1) at (4,2.5);
\draw[fill=black](z1){node[right]{$z_1$}} circle (0.1);
\draw[thick](z1)--(z2); 
\coordinate (zp) at (4,-2.5);
\draw[fill=black](zp){node[right]{$z_p$}} circle (0.1);
\draw[thick](zp1)--(zp); 
\path(z1) edge [bend right=10] (x1);
\draw(z2)..controls (3,0) and (3,-2)..(x4);
\draw(zp1)..controls (3,0) and (3,2)..(x1);
\path(zp) edge [bend left=10] (x4);
\foreach \x in {3,5} \draw[fill=white](x\x) circle (0.1);
\node[below] at (1.5,-3.5){(ii) $d_T(x_2)=2$, $y_{p-1}=x_1$};
\end{scope}

\begin{scope}[xshift=20cm]
\foreach \x in {1,2,...,6}
{\coordinate (x\x) at (150-60*\x:2);
\draw[fill=black](x\x) circle (0.1);
\node at (150-60*\x:1.5) {$x_\x$};}
\draw[thick](x1)--(x2)--(x3)--(x4)--(x5)--(x6)--(x1); 
\coordinate (z2) at (4,1);
\draw[fill=black](z2){node[right]{$z_2$}} circle (0.1);
\coordinate (zp1) at (4,-1);
\draw[fill=black](zp1){node[right]{$z_{p-1}$}} circle (0.1);
\draw[dotted,thick](z2)--(zp1); 
\coordinate (z1) at (4,2.5);
\draw[fill=black](z1){node[right]{$z_1$}} circle (0.1);
\draw[thick](z1)--(z2); 
\coordinate (zp) at (4,-2.5);
\draw[fill=black](zp){node[right]{$z_p$}} circle (0.1);
\draw[thick](zp1)--(zp); 
\path(z1) edge [bend right=10] (x1);
\draw(z2)..controls (3,0) and (3,-2)..(x4);
\path(zp1) edge [bend left=30] (x4);
\draw(zp)..controls (3,-1) and (3.5,2)..(x1);
\foreach \x in {3,5} \draw[fill=white](x\x) circle (0.1);
\node[below] at (1.5,-3.5){(iii) $d_T(x_2)=2$, $y_{p-1}=x_4$};
\end{scope}

\end{tikzpicture}
\caption{$|C|=6$.}
\label{figure: 6-cycle}
\end{figure}

\begin{subcase}
$d_T(z_2)=2$.
\end{subcase}

Recall that $d_C(y_1,y_2)=3$, we can assume w.l.o.g. that $y_1=x_1$ and $y_2=x_4$. If $x_3$ has a neighbor $z_0$ in $T$, then let $Q$ be the path from $z_2$ to $z_0$ in $T$. Now $z_2x_4x_3$, $z_2z_1x_1x_2x_3$, $z_2z_1x_1x_6x_5x_4x_3$ are three $(z_2,x_3)$-paths internally-disjoint with $Qz_0x_3$, and of lengths 2, 4, 6, respectively, a contradiction. So we conclude that $x_3$, and similarly, $x_5$, has no neighbors in $T$. 

If $d_T(z_{p-1})=3$, then we can prove similarly as in Case \ref{Ca3.1}. So we assume that $d_T(z_{p-1})=2$. Recall that $d_C(y_{p-1},y_p)=3$ and $\{y_{p-1},y_p\}\cap\{x_3,x_5\}=\emptyset$. We have that $\{y_{p-1},y_p\}=\{x_1,x_4\}$.  
If $y_{p-1}=x_1$ and $y_p=x_4$, see Figure \ref{figure: 6-cycle} (ii), 
then $z_2z_1x_1z_{p-1}$, $z_2x_4x_3x_2x_1z_{p-1}$, $z_2z_1x_1x_2x_3x_4z_pz_{p-1}$ are three $(z_2,z_{p-1})$-paths internally-disjoint with $P[z_2,z_{p-1}]$, and of lengths 3, 5, 7, respectively, a contradiction. 
If $y_{p-1}=x_4$ and $y_p=x_1$, see Figure \ref{figure: 6-cycle} (iii), 
then $z_2x_4z_{p-1}$, $z_2z_1x_1z_pz_{p-1}$, $z_2z_1x_1x_2x_3x_4z_{p-1}$ are three $(z_2,z_{p-1})$-paths internally-disjoint with $P[z_2,z_{p-1}]$, and of lengths 2, 4, 6, respectively, a contradiction.

\begin{case}
$|C|=8$.  
\end{case}

If $d_T(z_2)=2$, then $z_2$ has no neighbors in $C$ by Claim \ref{Cldcx1x2}, contradicting Claim \ref{Cldy2}. So we conclude that $d_T(z_2)=3$, and similarly, $d_T(z_{p-1})=3$. Recall that $d_C(y_1,y'_1)=4$, we can assume w.l.o.g. that $y_1=x_1$, $y'_1=x_5$.

If $x_2$ has a neighbor $z_0$ in $T$, then let $Q$ be the path from $z_2$ to $z_0$ in $T$. Now $z_2z_1x_1x_2$, $z_2z'_1x_5x_4x_3x_2$, $z_2z'_1x_5x_6x_7x_8x_1x_2$ are three $(z_2,x_2)$-paths internally-disjoint with $Qz_0x_2$, and of lengths 3, 5, 7, respectively, a contradiction. So we conclude that $x_2$, and similarly, $x_4,x_6,x_8$, have no neighbors in $T$. Recall that $d_C(y_p,y'_p)=4$. We have that either $\{y_p,y'_p\}=\{x_1,x_5\}$ or $\{y_p,y'_p\}=\{x_3,x_7\}$.

\begin{figure}[ht]
\centering
\begin{tikzpicture}[scale=0.6]

\begin{scope}
\foreach \x in {1,2,...,8}
{\coordinate (x\x) at (135-45*\x:2);
\draw[fill=black](x\x) circle (0.1);
\node at (135-45*\x:1.5) {$x_\x$};}
\draw[thick](x1)--(x2)--(x3)--(x4)--(x5)--(x6)--(x7)--(x8)--(x1); 
\coordinate (z2) at (4,1);
\draw[fill=black](z2){node[right]{$z_2$}} circle (0.1);
\coordinate (zp1) at (4,-1);
\draw[fill=black](zp1){node[right]{$z_{p-1}$}} circle (0.1);
\draw[dotted,thick](z2)--(zp1); 
\coordinate (z1) at (3,2.5); \coordinate (z1') at (5,2.5);
\draw[fill=black](z1){node[right]{$z_1$}} circle (0.1);
\draw[fill=black](z1'){node[right]{$z'_1$}} circle (0.1);
\draw[thick](z1)--(z2)--(z1'); 
\coordinate (zp) at (3,-2.5); \coordinate (zp') at (5,-2.5);
\draw[fill=black](zp){node[right]{$z_p$}} circle (0.1);
\draw[fill=black](zp'){node[right]{$z'_p$}} circle (0.1);
\draw[thick](zp)--(zp1)--(zp'); 
\path(z1) edge [bend right=10] (x1);
\draw(z1')..controls (3,4.5) and (-1,3)..(-2,2); 
\draw(-2,2)..controls (-3.5,0.5) and (-3,-3.5)..(x5);
\draw(zp).. controls (3.5,0) and (2.5,2.5)..(x1);
\path(zp') edge [bend left=40] (x5);
\foreach \x in {2,4,6,8} \draw[fill=white](x\x) circle (0.1);
\node[below] at (1.5,-3.5){(i) $\{y_{p-1},y_p\}=\{x_1,x_5\}$};
\end{scope}

\begin{scope}[xshift=11cm]
\foreach \x in {1,2,...,8}
{\coordinate (x\x) at (135-45*\x:2);
\draw[fill=black](x\x) circle (0.1);
\node at (135-45*\x:1.5) {$x_\x$};}
\draw[thick](x1)--(x2)--(x3)--(x4)--(x5)--(x6)--(x7)--(x8)--(x1); 
\coordinate (z2) at (4,1);
\draw[fill=black](z2){node[right]{$z_2$}} circle (0.1);
\coordinate (zp1) at (4,-1);
\draw[fill=black](zp1){node[right]{$z_{p-1}$}} circle (0.1);
\draw[dotted,thick](z2)--(zp1); 
\coordinate (z1) at (3,2.5); \coordinate (z1') at (5,2.5);
\draw[fill=black](z1){node[right]{$z_1$}} circle (0.1);
\draw[fill=black](z1'){node[right]{$z'_1$}} circle (0.1);
\draw[thick](z1)--(z2)--(z1'); 
\coordinate (zp) at (3,-2.5); \coordinate (zp') at (5,-2.5);
\draw[fill=black](zp){node[right]{$z_p$}} circle (0.1);
\draw[fill=black](zp'){node[right]{$z'_p$}} circle (0.1);
\draw[thick](zp)--(zp1)--(zp'); 
\path(z1) edge [bend right=10] (x1);
\draw(z1')..controls (3,4.5) and (-1,3)..(-2,2); 
\draw(-2,2)..controls (-3.5,0.5) and (-3,-3.5)..(x5);
\path(zp) edge [bend right=20] (x3);
\draw(zp')..controls (3,-4) and (-3,-4)..(x7);
\foreach \x in {2,4,6,8} \draw[fill=white](x\x) circle (0.1);
\node[below] at (1.5,-3.5){(ii) $\{y_{p-1},y_p\}=\{x_3,x_7\}$};
\end{scope}

\end{tikzpicture}
\caption{$|C|=10$.}
\label{figure: 8-cycle}
\end{figure}

Suppose first that $\{y_p,y'_p\}=\{x_1,x_5\}$, see Figure \ref{figure: 8-cycle} (i). We can assume w.l.o.g. that $y_p=x_1$, $y'_p=x_5$. Now $z_2z_1x_1z_pz_{p-1}$, $z_2z'_1x_5x_4x_3x_2x_1z_pz_{p-1}$ are two $(z_2,z_{p-1})$-paths internally-disjoint with $P[z_2,z_{p-1}]$, and of lengths 4, 8, respectively. It follows that $|P[z_2,z_{p-1}]|\equiv 1\bmod 3$ and $|P|\equiv 0\bmod 3$. By (3), we have that $d_G(z_3)\geqslant 3$. It follows that there exists a path $Q$ from $z_3$ to $C$ with all internal vertices not in $P$. Let $y_0$ be the terminus of $Q$. Then $y_0\in\{x_1,x_3,x_5,x_7\}$. 

If $y_0=x_1$, then $P[z_3,z_p]z_px_1$, $z_3z_2z_1x_1$, $z_3z_2z'_1x_5x_4x_3x_2x_1$ are three $(z_3,x_1)$-paths internally-disjoint with $Q$, and of lengths $|P[z_3,z_p]|+1$, 3, 7, respectively, which are different modulo 3 (notice that $|P[z_3,z_p]|\equiv 1\bmod 3$), a contradiction. If $y_0=x_3$, then $P[z_3,z_p]z_px_1x_2x_3$, $z_3z_2z_1x_1x_2x_3$, $z_3z_2z_1x_1x_8x_7x_6x_5x_4x_3$ are three $(z_3,x_3)$-paths internally-disjoint with $Q$, and of lengths $|P[z_3,z_p]|+3$, 5, 9, respectively, which are different modulo 3, a contradiction. The case $y_0=x_5$ or $x_7$ can be proved similarly. 

Now suppose that $\{y_p,y'_p\}=\{x_3,x_7\}$, see Figure \ref{figure: 8-cycle} (ii). We can assume w.l.o.g. that $y_p=x_3$, $y'_p=x_7$. We claim that there are no edges between $V(T)\backslash\{z_1,z'_2,z_p,z'_p\}$ and $C$. Suppose otherwise. There will be a path $Q$ from $P(z_2,z_{p-1})$ to $C$ with all internal vertices not in $P$. Let $z_0,y_0$ be the origin and terminus of $Q$, respectively. Recall that $y_0\in\{x_1,x_3,x_5,x_7\}$. We can assume w.l.o.g. that $y_0=x_1$, then $x_1z_1z_2Qx_1$ and $x_3z_pz_{p-1}z'_px_7x_6x_5x_4x_3$ are two disjoint cycles, a contradiction. Thus as we claimed, $z_1x_1, z_px_3, z'_1x_5, z'_px_7$ are the only edges between $P$ and $C$. It follows that $V(T)=V(P)\cup \{z'_1,z'_p\}$. Notice that $|V(T)|\geqslant 6$. We have $\rho(V(T))=|V(T)|+3\leqslant \frac{3}{2}|V(T)|$, a contradiction. 
\end{proof}

\section{Proof of Theorem \ref{thm: main}}\label{section: proof of main thm}

Suppose that $G$ is an $n$-vertex graph containing no $(1\bmod 3)$-cycles. We will show that $e(G)\leqslant 15q+\frac{3}{2}r$. We proceed our proof by induction. If $n\leqslant 6$, then $G$ contains no $(1\bmod 3)$-cycles if and only if $G$ contains no 4-cycles, and our assertion is true by the Tru\'an number $ex(n,C_4)$ (e.g., see \cite{CFS89}). So we will assume that $n\geqslant 7$. 

\setcounter{claim}{0}
\begin{claim}\label{claim: 2-con}
$G$ is $2$-connected.
\end{claim}

\begin{proof}
By $n-1=9q+r$, $0\leqslant r\leqslant 8$, we have 
\[
15q+\frac{3}{2}r
=15q+\frac{3}{2}(n-1-9q)
=\frac{3}{2}(n+q-1).    
\]
We assume that $G$ is connected; for otherwise we can add $\omega(G)-1$ edges between components of $G$ to obtain a connected graph, where $\omega(G)$ is the number of components of $G$ (notice that the added edges is not contained in any cycle). Suppose that $G$ is not 2-connected. Then there exist two nontrivial subgraphs $G_1,G_2$ of $G$ satisfying that $G_1\cup G_2=G$ and $V(G_1)\cap V(G_2)=\{x\}$. Let $n_i=|V(G_i)|$ for $i\in \{1,2\}$. Then $n+1=n_1+n_2$. Denote by $n_1-1=9q_1+r_1$, $n_2-1=9q_2+r_2$, here $0\leqslant r_1,r_2\leqslant 8$. Then either 
(i) $q=q_1+q_2$, $r=r_1+r_2$; or (ii) $q-1=q_1+q_2$, $r+9=r_1+r_2$. It follows that $q_1+q_2\leqslant q$.
By assumption, $G$ contains no $(1\bmod 3)$-cycles, and so does $G_1$ and $G_2$. By induction hypothesis,
we have $e(G_i)\leqslant \frac{3}{2}(n_i+q_i-1)$ for $i\in \{1,2\}$. Thus
\[
e(G)
=e(G_1)+e(G_2)
\leqslant \frac{3}{2}(n_1+q_1-1)+\frac{3}{2}(n_2+q_2-1)
=\frac{3}{2}((n_1+n_2)+(q_1+q_2)-2)
\leqslant \frac{3}{2}(n+q-1),
\]
a contradiction. 
\end{proof}

\begin{claim}\label{claim: rho}
$\rho(U)>\frac{3}{2}|U|$ for any nonempty subset $U$ of $V(G)$.
\end{claim}

\begin{proof}
Assume the opposite that $\rho(U)\leqslant \frac{3}{2}|U|$ for some nonempty $U\subseteq V(G)$. Let $n_1=|V(G-U)|$. Denote by $n_1-1=9q_1+r_1$, where $0\leqslant r_1\leqslant 8$. It follows that $|U|=n-n_1$ and $q_1\leqslant q$. Since $G-U$ contains no $(1\bmod 3)$-cycles, by induction hypothesis, we have $e(G-U)\leqslant \frac{3}{2}(n_1+q_1-1)$. Thus 
$$e(G)=e(G-U)+\rho(U)\leqslant \frac{3}{2}(n_1+q_1-1)+\frac{3}{2}(n-n_1)=\frac{3}{2}(n+q_1-1)\leqslant \frac{3}{2}(n+q-1),$$
a contradiction.
\end{proof}

From Claim \ref{claim: 2-con} and Theorem \ref{thm: dean et al.}, 
we can get that $\delta(G)=2$. 
Claim \ref{claim: rho} implies that $N_2(G)$ is an independent set, i.e., any two vertices of degree 2 are not adjacent to each other. 

\begin{claim}\label{claim: essentially-3-con}
$G$ is essentially $3$-connected. Moreover, for any vertex cut $\{x,y\}$ of $G$, we have $xy\notin E(G)$
and $G-\{x,y\}$ has exactly two components, one of which is nontrivial and the other one is trivial.
\end{claim}

\begin{proof}
Let $\{x,y\}$ be an arbitrary vertex cut of $G$. If $G-\{x,y\}$ has two nontrivial components $H_1, H_2$,
then the two subgraphs $G_i=G[V(H_i)\cup\{x,y\}]$, $i=1,2$, satisfy that $|V(G_1)|, |V(G_2)|\geqslant 4$ and $V(G_1)\cap V(G_2)=\{x,y\}$. Since $G$ is 2-connected, we see that $G_i+xy$ is 2-connected. By Claim \ref{claim: rho}, $N_2(G_i)\backslash\{x,y\}$ is an independent set. By assumption, $G_i$ contains no $(1\bmod 3)$-cycles. Then by Lemma \ref{lemma: 2 paths mod 3}, $G_i$ contains two $(x,y)$-paths of different lengths modulo 3, say of lengths $a_i$, $a_i+1\bmod 3$, $i=1,2$. This implies that $G$ contains three cycles of lengths $a_1+a_2,\ a_1+a_2+1,\ a_1+a_2+2\bmod 3$, respectively, one of which is a $(1\bmod 3)$-cycle, a contradiction. So we conclude that $G-\{x,y\}$ has at most one nontrivial component, which implies that $G$ is essentially 3-connected.

If $G-\{x,y\}$ has two or more trivial components, then the unique vertex in any trivial component is adjacent to both $x$ and $y$, and $G$ contains a 4-cycle passing through $x,y$, a contradiction. This implies that $G-\{x,y\}$ has exactly one trivial component and one nontrivial component.

Now we show that $xy\notin E(G)$. Suppose that $xy\in E(G)$. Let $z$ be the unique vertex in the trivial component of $G-\{x,y\}$. Let $G_1$ be a graph obtained from $G$ by removing the vertex $z$ and then removing the edge $xy$. By Lemma \ref{lemma: 2 paths mod 3}, $G_1$ contains two $(x,y)$-paths of different lengths modulo 3. Together with $xy$ or $xzy$, we will get a $(1\bmod 3)$-cycle in $G$, a contradiction.
\end{proof}

\begin{claim}\label{claim: 2 disjoint cycles}
$G$ contains two disjoint cycles.
\end{claim}

\begin{proof}
Assume the opposite that $G$ does not contain two disjoint cycles. Let $C$ be an arbitrary shortest cycle of $G$, and let $H$ be an arbitrary component of $G-C$. By Claim \ref{claim: 2-con}, Claim \ref{claim: rho} and Lemma \ref{lemma: L123}, 
we can get that $|C|=5$ and $G[V(C)\cup V(H)]$ is isomorphic to $L_1, L_2$ or $L_3$. 
In view of the three graphs $L_1,L_2,L_3$, 
we can also get that $3\leqslant |V(H)|\leqslant 5$ and $|N_C(H)|\geqslant 3$. 

Notice that each of $L_1,L_2,L_3$ has a second 5-cycle $C'$ which has exactly one common edge with $C$, say $V(C)\cap V(C')=\{x,y\}$ and $E(C)\cap E(C')=\{xy\}$. 

If $G-C$ has exactly one component $H$, then $G=G[V(C)\cup V(H)]$ is isomorphic to $L_1, L_2$ or $L_3$,
a contradiction to the assumption that $e(G)>15q+\frac{3}{2}r$. So $G-C$ has a second component $H'$. In this case $G-C'$ has a component $H''$ that contains $V(H')\cup(V(C)\backslash\{x,y\})$, which has order at least 6, contradicting Lemma \ref{lemma: L123}.
\end{proof}

\begin{claim}\label{claim: adjacent to two cycles}
$G$ contains two disjoint induced cycles $C_1,C_2$ such that every component of $G-C_1\cup C_2$ has neighbors in both $C_1$ and $C_2$.
\end{claim}

\begin{proof}
The existence of two disjoint induced cycles follows directly from Claim \ref{claim: 2 disjoint cycles}. Let $C_1,C_2$ be two such cycles of $G$ satisfying that the component of $G-C_1$ containing $C_2$ is as large as possible, and then the component of $G-C_2$ containing $C_1$ is as large as possible. By Lemma \ref{lemma: two cycles}, for each $i\in \{1,2\}$, the cycle $C_i$ has only one bridge in $G$; since otherwise $G$ has a vertex cut $\{x,y\}$ such that $G-\{x,y\}$ has at least three components. It follows that $C_i$ has no chord and $G-C_i$ has exactly one component.

Let $H$ be an arbitrary component of $G-C_1\cup C_2$. Since $G$ is 2-connected, we have $E(H,C_1\cup C_2)\neq \emptyset$. If $N_{C_2}(H)=\emptyset$, then $G-C_1$ will be disconnected, a contradiction. So we have that $N_{C_2}(H)\neq\emptyset$, and similarly, $N_{C_1}(H)\neq\emptyset$.
\end{proof}

By Claim \ref{claim: adjacent to two cycles}, let $C_1,C_2$ be two disjoint induced cycles such that every component of $G-C_1\cup C_2$ has neighbors in both $C_1$ and $C_2$. Two disjoint paths from $C_1$ to $C_2$ will be called \textit{clashing} if their origins are mod-non-diagonal in $C_1$ and their termini are mod-non-diagonal in $C_2$. From Lemma \ref{LeC1C2P1P2}, we see that $G$ contains no two clashing paths from $C_1$ to $C_2$.

By Claim \ref{claim: essentially-3-con}, $G$ is essentially 3-connected. So there exist three disjoint paths from $C_1$ to $C_2$, say $P^1,P^2,P^3$, with $V(C_i)\cap V(P^j)=\{x_i^j\}$, $i=1,2$ and $j=1,2,3$. We give orientations for $C_i$, $i=1,2$, such that $x_i^1,x_i^2,x_i^3$ appear in this order along $C_i$, see Figure \ref{fig:3paths}. For convenience, we denote by 
$$\begin{array}{lll}
|\overrightarrow{C_1}[x_1^1,x_1^2]|=\alpha_1, &|\overrightarrow{C_1}[x_1^2,x_1^3]|=\beta_1, 
&|\overrightarrow{C_1}[x_1^3,x_1^1]|=\gamma_1;\\
|\overrightarrow{C_2}[x_2^1,x_2^2]|=\alpha_2, &|\overrightarrow{C_2}[x_2^2,x_2^3]|=\beta_2, 
&|\overrightarrow{C_2}[x_2^3,x_2^1]|=\gamma_2.
\end{array}$$
Note that $|C_1|=\alpha_1+\beta_1+\gamma_1$ and $|C_2|=\alpha_2+\beta_2+\gamma_2$. Since $G$ contains no $(1\bmod 3)$-cycles, we have $|C_i|\not\equiv 1\bmod 3$ for any $i=1,2$. 

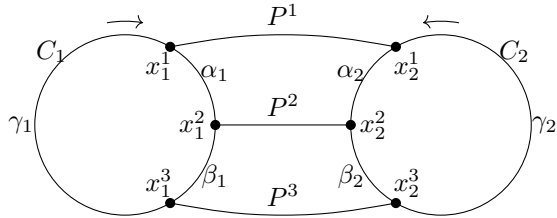
\begin{figure}[ht]
\centering
\begin{tikzpicture}[scale=0.6]
\draw(0,0) circle (2); 
\draw(7,0){coordinate (o2)} circle (2);
\coordinate (x11) at (60:2); \node at (60:1.5) {$x_1^1$};
\coordinate (x12) at (0:2); \node at (0:1.5) {$x_1^2$};
\coordinate (x13) at (300:2); \node at (300:1.5) {$x_1^3$};
\foreach \x in {1,2,3} \draw[fill=black] (x1\x) circle (0.1);
\node at (30:2.3) {$\alpha_1$};
\node at (330:2.3) {$\beta_1$};
\node at (180:2.3) {$\gamma_1$}; 
\path(o2)--+(120:2) coordinate (x21); \path(o2)--+(120:1.5){node{$x_2^1$}};
\path(o2)--+(180:2) coordinate (x22); \path(o2)--+(180:1.5){node{$x_2^2$}};
\path(o2)--+(240:2) coordinate (x23); \path(o2)--+(240:1.5){node{$x_2^3$}};
\foreach \x in {1,2,3} \draw[fill=black] (x2\x) circle (0.1);
\path(o2)--+(150:2.3){node{$\alpha_2$}};
\path(o2)--+(210:2.3){node{$\beta_2$}};
\path(o2)--+(0:2.3){node{$\gamma_2$}}; 
\coordinate (p1) at (3.5,2); \coordinate (p2) at (3.5,0); \coordinate (p3) at (3.5,-2);
\draw(x11) to [out=10, in=180] (p1); \draw(p1) to [out=0, in=170] (x21);
\draw(x12)--(x22);
\draw(x13) to [out=-10, in=180] (p3); \draw(p3) to [out=0, in=-170] (x23); 
\node at (135:2.3) {$C_1$}; \path(o2)--+(45:2.3){node{$C_2$}};
\foreach \x in {1,2,3} \node[above] at (p\x) {$P^\x$};
\draw[->] (100:2.3) arc (100:80:2.3);
\path(7,0)--+(80:2.3) coordinate (originalofarc);
\draw[->] (originalofarc) arc (80:100:2.3);
\end{tikzpicture}
\caption{Cycles $C_1,C_2$ and paths $P^1,P^2,P^3$.}\label{fig:3paths}
\end{figure}

In the following of the proof, we use $a\equiv b$ to refer that $a\equiv b\bmod 3$, and $(a_1,a_2,a_3)\equiv(b_1,b_2,b_3)$ means that $a_i\equiv b_i\bmod 3$ for $i=1,2,3$. For convenience, we set $M=C_1\cup C_2\cup P^1\cup P^2\cup P^3$.

\begin{claim}\label{ClEitherC1C2}
Either $|C_1|\equiv 2$ or $|C_2|\equiv 2$.
\end{claim}

\begin{proof}
Recall that $|C_i|\not\equiv 1$, $i=1,2$. 
Assume the opposite that $|C_1|\equiv|C_2|\equiv 0$. 
Then $\alpha_1+\beta_1+\gamma_1\equiv 0$, 
implying that 
$(\alpha_1,\beta_1,\gamma_1)\equiv(0,0,0),\ (1,1,1),\ (2,2,2)$ or $(0,1,2)$, up to symmetric.

\medskip\noindent\textbf{Case A.} $(\alpha_1,\beta_1,\gamma_1)\equiv(1,1,1),\ (2,2,2)$ or $(0,1,2)$.

In this case both $\{x_1^1,x_1^3\}$ and $\{x_1^2,x_1^3\}$ are mod-non-diagonal pairs in $C_1$. If $\{x_2^1,x_2^3\}$ or $\{x_2^2,x_2^3\}$ is a mod-non-diagonal pair in $C_2$, then either $P^1,P^3$ or $P^2,P^3$ are two clashing paths, a contradiction. 
So both $\{x_2^1,x_2^3\}$ and $\{x_2^2,x_2^3\}$ are mod-diagonal pairs in $C_2$. That is, $\beta_2\equiv\gamma_2\equiv 0$, and then $\alpha_2\equiv 0$ since $|C_2|\equiv 0$.
 
By $\alpha_2\equiv 0$, there exist at least two vertices in $\overrightarrow{C_2}(x_2^1,x_2^2)$.
Let $y_2=x_2^{1+}$ and $z_2=x_2^{1++}$ be the next two vertices after $x_2^1$ along $C_2$. By Claim \ref{claim: rho}, we have $\max\{d_G(y_2),d_G(z_2)\}\geqslant 3$. Since $C_2$ is an induced cycle, we have $N_G(\{y_2,z_2\})\backslash V(C_2)\neq\emptyset$. By the choice of $C_1,C_2$, there exists a path $Q_1$ from $V(M)\backslash V(C_2)$ to $\{y_2,z_2\}$ with all internal vertices in $V(G)\backslash V(M)$. Let $u_1\in V(M)\backslash V(C_2)$, $v_Q\in\{y_2,z_2\}$ be the origin and the terminus of $Q_1$, respectively. We set
\[
  Q=\left\{\begin{array}{l}
  Q_1,\\
  P^1[x_1^1,u_1]u_1Q_1,\\
  P^2[x_1^2,u_1]u_1Q_1,\\
  P^3[x_1^3,u_1]u_1Q_1,
\end{array}\right.\mbox{ and }u_Q=\left\{\begin{array}{ll}
  u_1,   & \mbox{if }u_1\in V(C_1),\\ 
  x_1^1, & \mbox{if }u_1\in V(P^1)\backslash\{x_1^1,x_2^1\},\\
  x_1^2, & \mbox{if }u_1\in V(P^2)\backslash\{x_1^2,x_2^2\},\\
  x_1^3, & \mbox{if }u_1\in V(P^3)\backslash\{x_1^3,x_2^3\}.
\end{array}\right.
\]
We remark that $Q$ and $P^i$ ($i=1,2,3$) are disjoint unless $u_Q=x_1^i$, see Figure \ref{FiQ4type}.

\begin{figure}[ht]
\centering
\begin{tikzpicture}[scale=0.6]

\begin{scope}
\draw(0,0) circle (2); 
\draw(7,0){coordinate (o2)} circle (2); 
\coordinate (x11) at (60:2); \node at (60:1.5) {$x_1^1$}; 
\coordinate (x12) at (0:2); \node at (0:1.5) {$x_1^2$}; 
\coordinate (x13) at (300:2); \node at (300:1.5) {$x_1^3$}; 
\foreach \x in {1,2,3} \draw[fill=black] (x1\x) circle (0.1);
\path(o2)--+(120:2) coordinate (x21); \path(o2)--+(120:1.5){node{$x_2^1$}};
\path(o2)--+(180:2) coordinate (x22); \path(o2)--+(180:1.5){node{$x_2^2$}};
\path(o2)--+(240:2) coordinate (x23); \path(o2)--+(240:1.5){node{$x_2^3$}};
\foreach \x in {1,2,3} \draw[fill=black] (x2\x) circle (0.1);
\coordinate (p1) at (3.5,2); \coordinate (p2) at (3.5,0); \coordinate (p3) at (3.5,-2);
\draw(x11) to [out=10, in=180] (p1); \draw(p1) to [out=0, in=170] (x21);
\draw(x12)--(x22);
\draw(x13) to [out=-10, in=180] (p3); \draw(p3) to [out=0, in=-170] (x23); 
\path(o2)--+(140:2) {coordinate (y2)}; \path(o2)--+(140:1.5){node{$y_2$}};
\draw[fill=black] (y2) {node[left] {$v_Q$}} circle (0.1);
\path(o2)--+(160:2) coordinate (z2); \path(o2)--+(160:1.5){node{$z_2$}};
\draw[fill=black] (z2) circle (0.1);
\draw[thick] (x21) arc (120:160:2);
\draw[fill=black] (100:2) {coordinate (u1)} circle (0.1);
\node[above left] at (95:1.8) {$u_Q=u_1$};
\draw[blue] (y2) to [bend right=50] (u1);
\node[above] at (4,2.5) {$Q_1$};
\node[below] at (3.5,-2.2) {(i) $u_1\in V(C_1)$};
\end{scope}

\begin{scope}[xshift=13cm]
\draw(0,0) circle (2); 
\draw(7,0){coordinate (o2)} circle (2); 
\coordinate (x11) at (60:2); \node at (60:1.5) {$x_1^1$}; 
\coordinate (x12) at (0:2); \node at (0:1.5) {$x_1^2$}; 
\coordinate (x13) at (300:2); \node at (300:1.5) {$x_1^3$}; 
\foreach \x in {1,2,3} \draw[fill=black] (x1\x) circle (0.1);
\path(o2)--+(120:2) coordinate (x21); \path(o2)--+(120:1.5){node{$x_2^1$}};
\path(o2)--+(180:2) coordinate (x22); \path(o2)--+(180:1.5){node{$x_2^2$}};
\path(o2)--+(240:2) coordinate (x23); \path(o2)--+(240:1.5){node{$x_2^3$}};
\foreach \x in {1,2,3} \draw[fill=black] (x2\x) circle (0.1);
\coordinate (p1) at (3.5,2); \coordinate (p2) at (3.5,0); \coordinate (p3) at (3.5,-2);
\draw[blue] (x11) to [out=10, in=180] (p1); \draw(p1) to [out=0, in=170] (x21);
\draw(x12)--(x22);
\draw(x13) to [out=-10, in=180] (p3); \draw(p3) to [out=0, in=-170] (x23); 
\path(o2)--+(140:2) {coordinate (y2)}; \path(o2)--+(140:1.5){node{$y_2$}};
\draw[fill=black] (y2) circle (0.1);
\path(o2)--+(160:2) coordinate (z2); \path(o2)--+(160:1.5){node{$z_2$}};
\draw[fill=black] (z2) circle (0.1);  \draw[thick] (x21) arc (120:160:2);
\node[left] at (5.4,1) {$v_Q$};
\draw[fill=black] (p1) {coordinate (u1)} {node[above] {$u_1$}} circle (0.1);
\node[above] at (x11) {$u_Q$};
\draw[blue] (z2) to [bend left=30] (u1);
\node at (3.7,1) {$Q_1$};
\node[below] at (3.5,-2.2) {(ii) $u_1\in V(P^1)\backslash\{x_1^1,x_2^1\}$};
\end{scope}

\begin{scope}[yshift=-6cm]
\draw(0,0) circle (2); 
\draw(7,0){coordinate (o2)} circle (2); 
\coordinate (x11) at (60:2); \node at (60:1.5) {$x_1^1$}; 
\coordinate (x12) at (0:2); \node at (0:1.5) {$x_1^2$}; 
\coordinate (x13) at (300:2); \node at (300:1.5) {$x_1^3$}; 
\foreach \x in {1,2,3} \draw[fill=black] (x1\x) circle (0.1);
\path(o2)--+(120:2) coordinate (x21); \path(o2)--+(120:1.5){node{$x_2^1$}};
\path(o2)--+(180:2) coordinate (x22); \path(o2)--+(180:1.5){node{$x_2^2$}};
\path(o2)--+(240:2) coordinate (x23); \path(o2)--+(240:1.5){node{$x_2^3$}};
\foreach \x in {1,2,3} \draw[fill=black] (x2\x) circle (0.1);
\coordinate (p1) at (3.5,2); \coordinate (p2) at (3.5,0); \coordinate (p3) at (3.5,-2);
\draw(x11) to [out=10, in=180] (p1); \draw(p1) to [out=0, in=170] (x21);
\draw[blue] (x12)--(p2); \draw(p2)--(x22);
\draw(x13) to [out=-10, in=180] (p3); \draw(p3) to [out=0, in=-170] (x23); 
\path(o2)--+(140:2) {coordinate (y2)}; \path(o2)--+(140:1.5){node{$y_2$}};
\draw[fill=black] (y2) circle (0.1);
\path(o2)--+(160:2) coordinate (z2); \path(o2)--+(160:1.5){node{$z_2$}};
\draw[fill=black] (z2) circle (0.1);  \draw[thick] (x21) arc (120:160:2);
\node[left] at (5.7,1.6) {$v_Q$};
\draw[fill=black] (p2) {coordinate (u1)} {node[below] {$u_1$}} circle (0.1);
\node[above right] at (1.8,-0.2) {$u_Q$};
\draw[blue] (y2) to [bend right=40] (u1);
\node at (3.5,1) {$Q_1$};
\node[below] at (3.5,-2.2) {(iii) $u_1\in V(P^2)\backslash\{x_1^2,x_2^2\}$};
\end{scope}

\begin{scope}[xshift=13cm, yshift=-6cm]
\draw(0,0) circle (2); 
\draw(7,0){coordinate (o2)} circle (2); 
\coordinate (x11) at (60:2); \node at (60:1.5) {$x_1^1$}; 
\coordinate (x12) at (0:2); \node at (0:1.5) {$x_1^2$}; 
\coordinate (x13) at (300:2); \node at (300:1.5) {$x_1^3$}; 
\foreach \x in {1,2,3} \draw[fill=black] (x1\x) circle (0.1);
\path(o2)--+(120:2) coordinate (x21); \path(o2)--+(120:1.5){node{$x_2^1$}};
\path(o2)--+(180:2) coordinate (x22); \path(o2)--+(180:1.5){node{$x_2^2$}};
\path(o2)--+(240:2) coordinate (x23); \path(o2)--+(240:1.5){node{$x_2^3$}};
\foreach \x in {1,2,3} \draw[fill=black] (x2\x) circle (0.1);
\coordinate (p1) at (3.5,2); \coordinate (p2) at (3.5,0); \coordinate (p3) at (3.5,-2);
\draw(x11) to [out=10, in=180] (p1); \draw(p1) to [out=0, in=170] (x21);
\draw(x12)--(x22);
\draw[blue] (x13) to [out=-10, in=180] (p3); \draw(p3) to [out=0, in=-170] (x23); 
\path(o2)--+(140:2) {coordinate (y2)}; \path(o2)--+(140:1.5){node{$y_2$}};
\draw[fill=black] (y2) circle (0.1);
\path(o2)--+(160:2) coordinate (z2); \path(o2)--+(160:1.5){node{$z_2$}};
\draw[fill=black] (z2) circle (0.1);  \draw[thick] (x21) arc (120:160:2);
\node[left] at (5.4,1) {$v_Q$};
\draw[fill=black] (p3) {coordinate (u1)} circle (0.1);  \node[above left] at (3.7,-2.1) {$u_1$};
\node[below] at (303:2) {$u_Q$};
\draw[blue] (z2) to [bend right=30] (u1);
\node at (4.2,-0.8) {$Q_1$};
\node[below] at (3.5,-2.2) {(iv) $u_1\in V(P^3)\backslash\{x_1^3,x_2^3\}$};
\end{scope}

\end{tikzpicture}
\caption{The path $Q$ in Claim \ref{ClEitherC1C2} (the blue path).}\label{FiQ4type}
\end{figure}
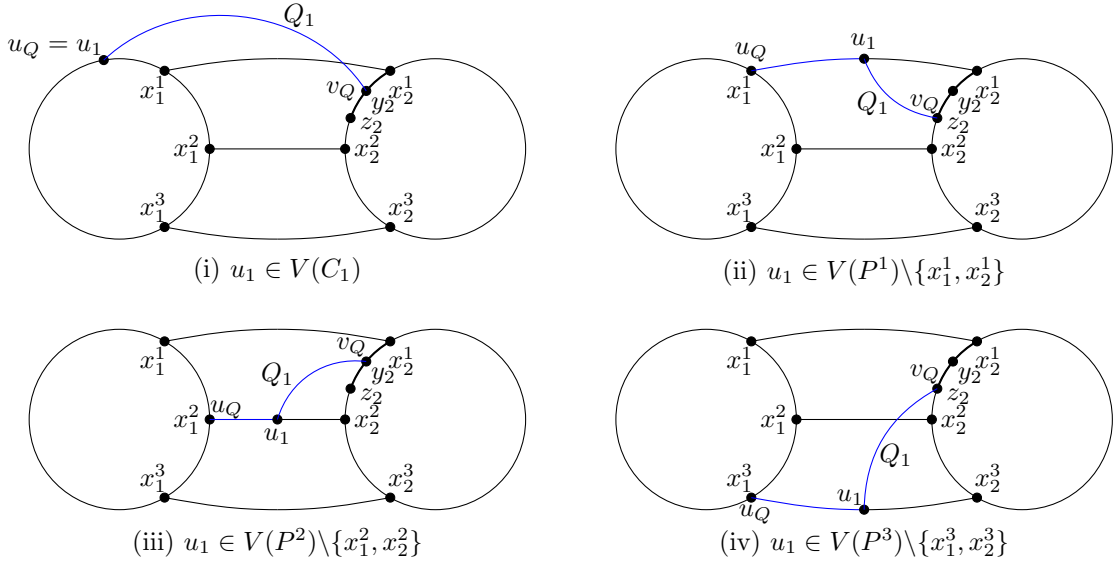

We claim that either $u_Q,x_1^2$ or $u_Q,x_1^3$ are mod-non-diagonal in $C_1$. This is true for $u_Q\in\{x_1^2,x_1^3\}$ since $x_1^2,x_1^3$ are mod-non-diagonal, and if $u_Q\in V(C_1)\backslash\{x_1^2,x_1^3\}$, then it can be deduced from Lemma \ref{LenDiagonal}. In any case we have that either $u_Q,x_1^2$ or $u_Q,x_1^3$ are mod-non-diagonal in $C_1$. Also note that both $v_Q,x_2^2$ and $v_Q,x_2^3$ are mod-non-diagonal in $C_2$. It follows that either $Q,P^2$ or $Q,P^3$ are two clashing paths, a contradiction.

\medskip\noindent\textbf{Case B.} $(\alpha_1,\beta_1,\gamma_1)\equiv (0,0,0)$.

Since $\alpha_1\equiv 0$, there exist at least two vertices in $\overrightarrow{C_1}(x_1^1,x_1^2)$. Let $y_1=x_1^{1+}$ and $z_1=x_1^{1++}$. By Claim \ref{claim: rho}, there exists a path $R_1$ from $\{y_1,z_1\}$ to $V(M)\backslash V(C_1)$ with all internal vertices in $V(G)\backslash V(M)$. Let $u_R\in\{y_1,z_1\}$ be the origin of $R_1$ and $v_1\in V(M)\backslash V(C_1)$ be the terminus of $R_1$. If $v_1\in V(C_2)$, then let $R=R_1$ and $v_R=v_1$; if $v_1\in V(P^i)\backslash\{x_1^i,x_2^i\}$ for some $i=1,2,3$, then let $R=R_1v_1P^i[v_1,x_2^i]$ and $v_R=x_2^i$.

Now $R$ is disjoint with at least two paths of $P^1,P^2,P^3$. Notice that $|\overrightarrow{C_1}[u_R,x_1^i]|\not\equiv 0$, for all $i=1,2,3$. It follows that there are three disjoint paths from $C_1$ to $C_2$ such that their origins separate $C_1$ into three segments of lengths $0,1,2\bmod 3$. Thus we can prove the assertion as in Case A.
\end{proof}

\begin{claim}\label{ClBothC1C2}
    Both $|C_1|\equiv 2$ and $|C_2|\equiv 2$.
\end{claim}

\begin{proof}
Assume w.l.o.g. that $|C_1|\equiv 0$ and $|C_2|\equiv 2$. 
If every two vertices in $\{x_2^1,x_2^2,x_2^3\}$ are mod-diagonal in $C_2$, 
then $\alpha_2\equiv\beta_2\equiv\gamma_2\equiv 1$ 
by Lemma \ref{LenDiagonal},
contradicting that $|C_2|\equiv 2$. So we conclude that there are two vertices in $\{x_2^1,x_2^2,x_2^3\}$ that are mod-non-diagonal in $C_2$. If $(\alpha_1,\beta_1,\gamma_1)\equiv(1,1,1)$ or $(2,2,2)$, then every two vertices in $\{x_1^1,x_1^2,x_1^3\}$ are mod-non-diagonal in $C_1$, and two paths of $P^1,P^2,P^3$ are clashing, a contradiction. So we conclude that $(\alpha_1,\beta_1,\gamma_1)\equiv(0,0,0)$ or $(0,1,2)$, up to symmetric.

\medskip\noindent\textbf{Case A.} $(\alpha_1,\beta_1,\gamma_1)\equiv(0,1,2)$.

In this case both $x_1^1,x_1^3$ and $x_1^2,x_1^3$ are mod-non-diagonal in $C_1$. 
This implies that both $x_2^1,x_2^3$ and $x_2^2,x_2^3$ are mod-diagonal in $C_2$. That is, $\beta_2\equiv\gamma_2\equiv 1$, which follows that $\alpha_2\equiv 0$.

Let $y_2=x_2^{1+}$ and $z_2=x_2^{1++}$. There exists a path $Q_1$ from $V(M)\backslash V(C_2)$ to $\{y_2,z_2\}$ with all internal vertices in $V(G)\backslash V(M)$. Let $v_Q\in\{y_2,z_2\}$ be the terminus of $Q_1$ and $u_1\in V(M)\backslash V(C_2)$ be the origin of $Q_1$. We define the path $Q$ as in Claim \ref{ClEitherC1C2}, and let $u_Q$ be the origin of $Q$. So $Q$ is a path from $C_1$ to $C_2$ and $Q$ is disjoint with $P^i$ unless $u_Q=x_1^i$, $i=1,2,3$.

Notice that $v_Q,x_2^3$ are mod-non-diagonal in $C_2$. If $u_Q,x_1^3$ are mod-non-diagonal in $C_1$, then $Q,P^3$ are two clashing paths, a contradiction. So we conclude that $u_Q,x_1^3$ are mod-diagonal in $C_1$ or $u_Q=x_1^3$. Specially, $u_Q\notin\{x_1^1,x_1^2\}$. Since $|\overrightarrow{C_1}[u_Q,x_1^3]|\equiv 0$, we see that $|\overrightarrow{C_1}[u_Q,x_1^1]|\not\equiv 0$ and $|\overrightarrow{C_1}[u_Q,x_1^2]|\not\equiv 0$.
Now both $u_Q,x_1^1$ and $u_Q,x_1^2$ are mod-non-diagonal in $C_1$. 

Notice that $|\overrightarrow{C_2}[x_2^1,v_Q]|+|\overrightarrow{C_2}[v_Q,x_2^2]|\equiv 0$. This implies that either $\{v_Q,x_2^1\}$ or $\{v_Q,x_2^2\}$ is a mod-non-diagonal pair in $C_2$. Thus either $Q,P^1$ or $Q,P^2$ are two clashing paths, a contradiction.

\medskip\noindent\textbf{Case B.} $(\alpha_1,\beta_1,\gamma_1)\equiv(0,0,0)$.

This case can be deduced through similar analysis to that in Claim \ref{ClEitherC1C2}. 
\end{proof}

By Claim \ref{ClBothC1C2}, $|C_1|\equiv|C_2|\equiv 2$. It follows that $(\alpha_1,\beta_1,\gamma_1)\equiv(0,1,1),(0,0,2)$ or $(1,2,2)$, up to symmetric.

\begin{claim}\label{ClQRequiv2}
If $Q,R$ are two disjoint paths from $C_1$ to $C_2$, such that either their origins are mod-non-diagonal in $C_1$, or their termini are mod-non-diagonal in $C_2$, then $|Q|+|R|\equiv 2$. 
\end{claim}

\begin{proof}
Let $u_Q,u_R$ be the origins of $Q,R$ and $v_Q,v_R$ their termini. 
If $u_Q,u_R$ are mod-non-diagonal in $C_1$, then $v_Q,v_R$ are mod-diagonal in $C_2$; otherwise $Q,R$ are two clashing paths. That is, $|\overrightarrow{C_2}[v_Q,v_R]|\equiv|\overleftarrow{C_2}[v_Q,v_R]|\equiv 1$. We can assume w.l.o.g. that $|\overrightarrow{C_1}[u_Q,u_R]|\equiv 0$ and $|\overleftarrow{C_1}[u_Q,u_R]|\equiv 2$, see Figure \ref{FiTwoPathsQR}. Consider the two cycles 
\[
\overrightarrow{C_1}[u_Q,u_R]u_RRv_R\overrightarrow{C_2}[v_R,v_Q]v_QQu_Q,~
\overleftarrow{C_1}[u_Q,u_R]u_RRv_R\overrightarrow{C_2}[v_R,v_Q]v_QQu_Q. 
\]
They are not $(1\bmod 3)$-cycles, implying that $0+|Q|+1+|R|\not\equiv 1$ and $2+|Q|+1+|R|\not\equiv 1$. So $|Q|+|R|\equiv 2$, as desired. The second case can be proved similarly.
\end{proof}

\begin{figure}[ht]
\centering
\begin{tikzpicture}[scale=0.6]
\draw(0,0) circle (2); 
\draw(7,0){coordinate (o2)} circle (2);
\coordinate (x1) at (45:2); \node at (45:1.5) {$u_Q$};
\coordinate (x2) at (315:2); \node at (315:1.5) {$u_R$};
\foreach \x in {1,2} \draw[fill=black] (x\x) circle (0.1);
\node[left] at (180:2) {2}; \node[right] at (0:2) {0}; 
\path(o2)--+(135:2) coordinate (y1); \path(o2)--+(135:1.5){node{$v_Q$}};
\path(o2)--+(225:2) coordinate (y2); \path(o2)--+(225:1.5){node{$v_R$}};
\foreach \x in {1,2} \draw[fill=black] (y\x) circle (0.1);
\path(o2)--+(180:2){node[left]{1}}; \path(o2)--+(0:2){node[right]{1}};
\path(x1) edge [bend left=20] (y1); \path(x2) edge [bend right=20] (y2);
\node[above] at (3.5,1.8) {$Q$}; \node[above] at (3.5,-1.8) {$R$};
\end{tikzpicture}
\caption{Two paths $Q,R$ from $C_1$ to $C_2$.}
\label{FiTwoPathsQR}
\end{figure}
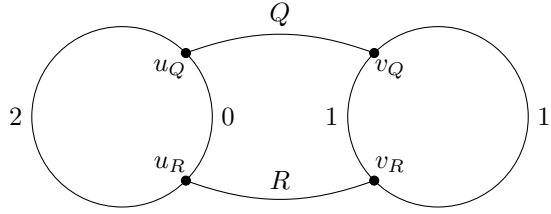

\begin{claim}\label{Cl011vs122}
Either $(\alpha_1,\beta_1,\gamma_1)\equiv(0,1,1)$, $(\alpha_2,\beta_2,\gamma_2)\equiv(1,2,2)$, or $(\alpha_1,\beta_1,\gamma_1)\equiv(1,2,2)$, $(\alpha_2,\beta_2,\gamma_2)\equiv(0,1,1)$, up to symmetric.
\end{claim}

\begin{proof}
First suppose that $(\alpha_1,\beta_1,\gamma_1)\equiv(0,0,2)$. 
Then every two vertices in $\{x_1^1,x_1^2,x_1^3\}$ are mod-non-diagonal in $C_1$. Since $|C_2|\equiv 2$, 
there exist two vertices in $\{x_2^1,x_2^2,x_2^3\}$ which are mod-non-diagonal in $C_2$. It follows that two paths of $P^1,P^2,P^3$ are clashing, a contradiction. 

Second suppose that $(\alpha_1,\beta_1,\gamma_1)\equiv(1,2,2)$. 
In this case both $\{x_1^1,x_1^3\}$ and $\{x_1^2,x_1^3\}$ are mod-non-diagonal pairs in $C_1$. This implies that both $\{x_2^1,x_2^3\}$ and $\{x_2^2,x_2^3\}$ are mod-diagonal pairs in $C_2$. That is, $\beta_2\equiv\gamma_2\equiv 1$, which follows that $\alpha_2\equiv 0$, as desired.

Finally we suppose that $(\alpha_1,\beta_1,\gamma_1)\equiv(0,1,1)$. 
Since $x_1^1,x_1^2$ are mod-non-diagonal in $C_1$, 
we have that $x_2^1,x_2^2$ are mod-diagonal in $C_2$, i.e., $\alpha_2\equiv 1$. 
If $(\alpha_2,\beta_2,\gamma_2)\not\equiv(1,2,2)$, 
then we can assume w.l.o.g. that $\beta_2\equiv 0$ and $\gamma_2\equiv 1$. 

By $\beta_2\equiv 0$, 
there exist at least two vertices in $\overrightarrow{C_2}(x_2^2,x_2^3)$. 
We define a path $Q$ from $C_1$ to $C_2$ as in Claim \ref{ClEitherC1C2}, and let $u_Q\in V(C_1)$, $v_Q\in\{x_2^{2+},x_2^{2++}\}$ be the origin and terminus of $Q$. We notice that $Q$ and $P^i$ are disjoint unless $u_Q=x_1^i$ for $i=1,2,3$. By symmetric, we see that $\overleftarrow{C_1}(x_1^2,x_1^1)$ contains at least two vertices, and there exists a path $R$ from $C_1$ to $C_2$ with origin $u_R\in\{x_1^{2-},x_1^{2--}\}$ and terminus $v_R\in V(C_2)$ such that $R$ and $P^i$ are disjoint unless $v_R=x_2^i$ for $i=1,2,3$. 

\begin{subclaim}\label{ClvQv2uRu2}
    $v_Q=x_2^{2++}$, and $u_R=x_1^{2--}$.
\end{subclaim}

\begin{proof}
Suppose that $v_Q=x_2^{2+}$. 
Then both $\{v_Q,x_2^1\}$ and $\{v_Q,x_2^3\}$ are mod-non-diagonal pairs in $C_2$. If $u_Q\notin\{x_1^1,x_1^3\}$, 
then $Q$ is disjoint with $P^1$ and $P^3$. 
Since $|\overrightarrow{C_1}[x_1^1,x_1^3]|\equiv|\overrightarrow{C_1}[x_1^3,x_1^1]|\equiv 1$, 
by Lemma \ref{LenDiagonal}, 
every vertex in $V(C_1)\backslash\{x_1^1,x_1^3\}$ is mod-non-diagonal with either $x_1^1$ or $x_1^3$. 
It follows that either $Q,P^1$ or $Q,P^3$ are two clashing paths, a contradiction. 
Now we conclude that $u_Q=x_1^1$ or $x_1^3$, see Figure \ref{FivQx22+}. 
By Claim \ref{ClQRequiv2}, $|P^1|+|P^2|\equiv 2$ and $|P^2|+|P^3|\equiv 2$. 

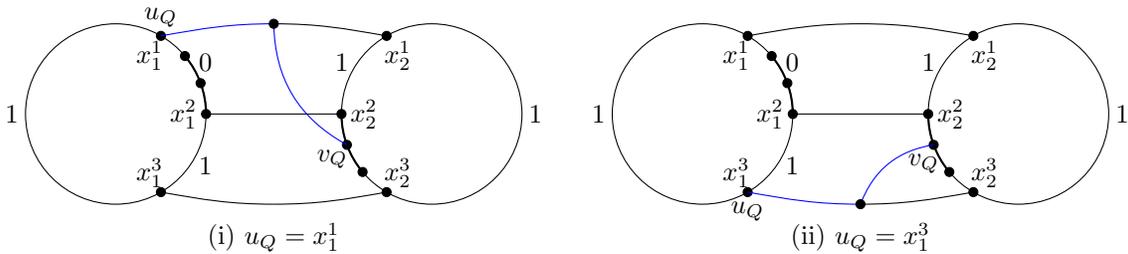
\begin{figure}[ht]
\centering
\begin{tikzpicture}[scale=0.6]

\begin{scope}
\draw(0,0) circle (2); 
\draw(7,0){coordinate (o2)} circle (2); 
\coordinate (x11) at (60:2); \node at (60:1.5) {$x_1^1$}; 
\coordinate (x12) at (0:2); \node at (0:1.5) {$x_1^2$}; 
\coordinate (x13) at (300:2); \node at (300:1.5) {$x_1^3$}; 
\foreach \x in {1,2,3} \draw[fill=black] (x1\x) circle (0.1);
\path(o2)--+(120:2) coordinate (x21); \path(o2)--+(120:1.5){node{$x_2^1$}};
\path(o2)--+(180:2) coordinate (x22); \path(o2)--+(180:1.5){node{$x_2^2$}};
\path(o2)--+(240:2) coordinate (x23); \path(o2)--+(240:1.5){node{$x_2^3$}};
\foreach \x in {1,2,3} \draw[fill=black] (x2\x) circle (0.1);
\coordinate (p1) at (3.5,2); \coordinate (p2) at (3.5,0); \coordinate (p3) at (3.5,-2);
\draw[blue] (x11) to [out=10, in=180] (p1); \draw(p1) to [out=0, in=170] (x21);
\draw(x12)--(x22);
\draw(x13) to [out=-10, in=180] (p3); \draw(p3) to [out=0, in=-170] (x23); 
\node at (30:2.3) {0}; \node at (-30:2.3) {1}; \node at (180:2.3) {1};
\path(o2)--+(150:2.3){node{1}}; \path(o2)--+(0:2.3){node{1}}; 
\path(o2)--+(200:2) {coordinate (y2)}; 
\draw[fill=black] (y2) circle (0.1);
\path(o2)--+(220:2) coordinate (z2); 
\draw[fill=black] (z2) circle (0.1);
\draw[thick] (x22) arc (180:220:2); 
\draw[fill=black] (20:2) {coordinate (y1)} circle (0.1); 
\draw[fill=black] (40:2) {coordinate (z1)} circle (0.1); 
\draw[thick] (x12) arc (0:40:2); 
\draw[blue] (y2) to [bend left=30] (p1);
\draw[fill=black] (p1) circle (0.1);
\node[above] at (x11) {$u_Q$}; \node at (4.8,-1) {$v_Q$};
\node[below] at (3.5,-2.2) {(i) $u_Q=x_1^1$};
\end{scope}

\begin{scope}[xshift=13cm]
\draw(0,0) circle (2); 
\draw(7,0){coordinate (o2)} circle (2); 
\coordinate (x11) at (60:2); \node at (60:1.5) {$x_1^1$}; 
\coordinate (x12) at (0:2); \node at (0:1.5) {$x_1^2$}; 
\coordinate (x13) at (300:2); \node at (300:1.5) {$x_1^3$}; 
\foreach \x in {1,2,3} \draw[fill=black] (x1\x) circle (0.1);
\path(o2)--+(120:2) coordinate (x21); \path(o2)--+(120:1.5){node{$x_2^1$}};
\path(o2)--+(180:2) coordinate (x22); \path(o2)--+(180:1.5){node{$x_2^2$}};
\path(o2)--+(240:2) coordinate (x23); \path(o2)--+(240:1.5){node{$x_2^3$}};
\foreach \x in {1,2,3} \draw[fill=black] (x2\x) circle (0.1);
\coordinate (p1) at (3.5,2); \coordinate (p2) at (3.5,0); \coordinate (p3) at (3.5,-2);
\draw(x11) to [out=10, in=180] (p1); \draw(p1) to [out=0, in=170] (x21);
\draw(x12)--(x22);
\draw[blue] (x13) to [out=-10, in=180] (p3); \draw(p3) to [out=0, in=-170] (x23); 
\node at (30:2.3) {0}; \node at (-30:2.3) {1}; \node at (180:2.3) {1};
\path(o2)--+(150:2.3){node{1}}; \path(o2)--+(0:2.3){node{1}}; 
\path(o2)--+(200:2) {coordinate (y2)}; 
\draw[fill=black] (y2) circle (0.1);
\path(o2)--+(220:2) coordinate (z2); 
\draw[fill=black] (z2) circle (0.1);
\draw[thick] (x22) arc (180:220:2); 
\draw[fill=black] (20:2) {coordinate (y1)} circle (0.1); 
\draw[fill=black] (40:2) {coordinate (z1)} circle (0.1); 
\draw[thick] (x12) arc (0:40:2); 
\draw[blue] (y2) to [bend right=30] (p3);
\draw[fill=black] (p3) circle (0.1);
\node[below] at (x13) {$u_Q$}; \node at (4.9,-1.1) {$v_Q$};
\node[below] at (3.5,-2.2) {(ii) $u_Q=x_1^3$};
\end{scope}

\end{tikzpicture}
\caption{The path $Q$ in Claim \ref{ClvQv2uRu2}.}\label{FivQx22+}
\end{figure}

If $u_Q=x_1^1$, then $Q$ is disjoint with $P^2,P^3$, see Figure \ref{FivQx22+} (i). By Claim \ref{ClQRequiv2}, $|Q|+|P^2|\equiv|Q|+|P^3|\equiv 2$. It follows that $|P^2|\equiv|P^3|$ and $|P^1|+|P^3|\equiv 2$. Now $\overrightarrow{C_1}[x_1^1,x_1^3]x_1^3P^3x_2^3\overrightarrow{C_2}[x_2^3,x_2^1]x_2^1P^1x_1^1$ is a $(1\bmod 3)$-cycle, a contradiction.

Now assume that $u_Q=x_1^3$, see Figure \ref{FivQx22+} (ii). Then $Q\cup P^3$ contains a path $S$ from $v_Q$ to $x_2^3$ with all internal vertices in $(V(G)\backslash V(M))\cup V(P^3)$. 
Now there exist three $(v_Q,x_2^3)$-paths 
\[
\overrightarrow{C_2}[v_Q,x_2^3],~
\overleftarrow{C_2}[v_Q,x_2^3],~
v_Qx_2^2P^2x_1^2\overleftarrow{C_1}[x_1^2,x_1^1]x_1^1P^1x_2^1\overleftarrow{C_1}[x_2^1,x_2^3],
\]
which are internally-disjoint with $S$, and of lengths $2,\ 0,\ 1\bmod 3$, respectively, a contradiction. 

The second assertion can be proved similarly. 
\end{proof}

\begin{subclaim}\label{CluQx1diagonal}
$u_Q\in V(\overrightarrow{C_1}(x_1^2,x_1^1))$ 
and $\{u_Q, x_1^1\}$ is a mod-diagonal pair in $C_1$; 
$v_R\in V(\overrightarrow{C_2}(x_2^3,x_2^2))$ 
and $\{v_R, x_2^3\}$ is a mod-diagonal pair in $C_2$.
\end{subclaim}

\begin{proof}
By Claim \ref{ClvQv2uRu2}, $v_Q=x_2^{2++}$, 
implying that both $\{v_Q,x_2^1\}$ and $\{v_Q,x_2^2\}$ are mod-non-diagonal pairs in $C_2$. It follows that neither $\{u_Q,x_1^1\}$ nor $\{u_Q,x_1^2\}$ is a mod-non-diagonal pair in $C_1$. Notice that $x_1^1,x_1^2$ are mod-non-diagonal, and every vertex in $\overrightarrow{C_1}(x_1^1,x_1^2)$ is mod-non-diagonal with either $x_1^1$ or $x_1^2$ in $C_1$. We conclude that $u_Q\in V(\overrightarrow{C_1}(x_1^2,x_1^1))$ and $\{u_Q, x_1^1\}$ is a mod-diagonal pair in $C_1$. 
The second assertion can be proved similarly.
\end{proof}

We remark that possibly $V(Q)\cap V(R)\neq\emptyset$. If this happens, then $Q\cup R$ contains a path $S$ from $u_R$ to $v_Q$, see Figure \ref{FiQRS} (i). Moreover, $S$ is disjoint with $P^2$ since both $Q$ and $R$ are disjoint with $P^2$. Notice that $u_R,x_1^2$ are mod-non-diagonal in $C_1$ and $v_Q,x_2^2$ are mod-non-diagonal in $C_2$. It follows that $S,P^2$ are two clashing paths, a contradiction.

\begin{figure}[ht]
\centering
\begin{tikzpicture}[scale=0.6]

\begin{scope}
\draw(0,0) circle (2); 
\draw(7,0){coordinate (o2)} circle (2); 
\coordinate (x11) at (60:2); \node at (60:1.5) {$x_1^1$}; 
\coordinate (x12) at (0:2); \node at (0:1.5) {$x_1^2$}; 
\coordinate (x13) at (300:2); \node at (300:1.5) {$x_1^3$}; 
\foreach \x in {1,2,3} \draw[fill=black] (x1\x) circle (0.1);
\path(o2)--+(120:2) coordinate (x21); \path(o2)--+(120:1.5){node{$x_2^1$}};
\path(o2)--+(180:2) coordinate (x22); \path(o2)--+(180:1.5){node{$x_2^2$}};
\path(o2)--+(240:2) coordinate (x23); \path(o2)--+(240:1.5){node{$x_2^3$}};
\foreach \x in {1,2,3} \draw[fill=black] (x2\x) circle (0.1);
\coordinate (p1) at (3.5,2); \coordinate (p2) at (3.5,0); \coordinate (p3) at (3.5,-2);
\draw(x11) to [out=10, in=180] (p1); \draw(p1) to [out=0, in=170] (x21);
\draw(x12)--(x22);
\draw(x13) to [out=-10, in=180] (p3); \draw(p3) to [out=0, in=-170] (x23); 
\node at (-30:2.3) {1}; \node at (180:2.3) {1};
\path(o2)--+(150:2.3){node{1}}; \path(o2)--+(0:2.3){node{1}}; 
\path(o2)--+(200:2) {coordinate (y2)}; 
\draw[fill=black] (y2) circle (0.1);
\path(o2)--+(220:2) coordinate (z2); 
\draw[fill=black] (z2) circle (0.1);
\draw[thick] (x22) arc (180:220:2); 
\draw[fill=black] (20:2) {coordinate (y1)} circle (0.1); 
\draw[fill=black] (40:2) {coordinate (z1)} circle (0.1); 
\draw[thick] (x12) arc (0:40:2); 
\draw[blue] (z1)..controls (3.5,1.5) and (3.5,-1.5)..(z2);
\node at (1.8,1.6) {$u_R$}; \node at (5.2,-1.6) {$v_Q$};
\node[below] at (3.5,-2.7) {(i) The path $S$ ($V(Q)\cap V(R)\neq\emptyset$)};
\end{scope}

\begin{scope}[xshift=13cm]
\draw(0,0) circle (2); 
\draw(7,0){coordinate (o2)} circle (2); 
\coordinate (x11) at (60:2); \node at (60:1.5) {$x_1^1$}; 
\coordinate (x12) at (0:2); \node at (0:1.5) {$x_1^2$}; 
\coordinate (x13) at (300:2); \node at (300:1.5) {$x_1^3$}; 
\foreach \x in {1,2,3} \draw[fill=black] (x1\x) circle (0.1);
\path(o2)--+(120:2) coordinate (x21); \path(o2)--+(120:1.5){node{$x_2^1$}};
\path(o2)--+(180:2) coordinate (x22); \path(o2)--+(180:1.5){node{$x_2^2$}};
\path(o2)--+(240:2) coordinate (x23); \path(o2)--+(240:1.5){node{$x_2^3$}};
\foreach \x in {1,2,3} \draw[fill=black] (x2\x) circle (0.1);
\coordinate (p1) at (3.5,2); \coordinate (p2) at (3.5,0); \coordinate (p3) at (3.5,-2);
\draw(x11) to [out=10, in=180] (p1); \draw(p1) to [out=0, in=170] (x21);
\draw(x12)--(x22);
\draw(x13) to [out=-10, in=180] (p3); \draw(p3) to [out=0, in=-170] (x23); 
\node at (-30:2.3) {1}; \node at (180:2.3) {1};
\path(o2)--+(150:2.3){node{1}}; \path(o2)--+(0:2.3){node{1}}; 
\path(o2)--+(200:2) {coordinate (y2)}; 
\draw[fill=black] (y2) circle (0.1);
\path(o2)--+(220:2) coordinate (z2); 
\draw[fill=black] (z2) circle (0.1);
\draw[thick] (x22) arc (180:220:2); 
\draw[fill=black] (20:2) {coordinate (y1)} circle (0.1); 
\draw[fill=black] (40:2) {coordinate (z1)} circle (0.1); 
\draw[thick] (x12) arc (0:40:2); 
\path(o2)--+(80:2) {coordinate (vR)}; \draw[fill=black] (vR) circle (0.1);
\draw[fill=black] (260:2) {coordinate (uQ)} circle (0.1);
\draw[blue] (z1) to [bend left=30] (vR);
\draw[blue] (z2) to [bend left=30] (uQ);
\node at (30:2.4) {$u_R$}; \node at (260:2.4) {$u_Q$}; 
\path(o2)--+(210:2.4) {node {$v_Q$}}; \path(o2)--+(80:2.4) {node {$v_R$}};
\node[below] at (3.5,-2.7) {(ii) The paths $Q$ and $R$ ($V(Q)\cap V(R)=\emptyset$)};
\end{scope}

\end{tikzpicture}
\caption{The paths $Q$, $R$ and $S$ in Claim \ref{Cl011vs122}.}\label{FiQRS}
\end{figure}
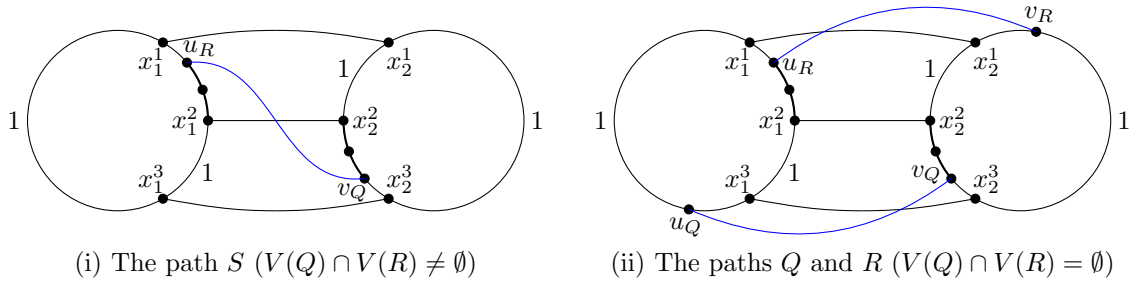

Now suppose that $Q,R$ are disjoint, see Figure \ref{FiQRS} (ii). By Claim \ref{CluQx1diagonal}, $u_Q\in V(\overrightarrow{C_1}(x_1^2,x_1^1))$ and $u_Q, x_1^1$ are mod-diagonal in $C_1$. Since $|\overrightarrow{C_1}[x_1^1,u_R]|\equiv 1$, we see that $u_Q,u_R$ are mod-non-diagonal in $C_1$. Similarly, $v_Q,v_R$ are mod-non-diagonal in $C_2$, implies that $Q,R$ are two clashing paths, a contradiction.
\end{proof}

\begin{claim}\label{ClQRExactlyOne}
Suppose that $Q,R$ are two disjoint paths from $C_1$ to $C_2$. 
If there exists a third path from $C_1$ to $C_2$ that is disjoint with $Q,R$, 
then $|Q|\equiv|R|\equiv 1$ and exactly one of the following holds: \\ 
\indent $(1)$ the origins of $Q,R$ are mod-diagonal in $C_1$, or \\
\indent $(2)$ the termini of $Q,R$ are mod-diagonal in $C_2$. 
\end{claim}

\begin{proof}
Let $S$ be a third path from $C_1$ to $C_2$. Let $S^1,S^2,S^3$ be a rearrangement of $Q,R,S$, with $V(C_i)\cap V(S^j)=\{y_i^j\}$, such that $y_1^1,y_1^2,y_1^3$ appear in this order along $C_1$ and $y_2^1,y_2^2,y_2^3$ appear in this order along $C_2$ (with proper orientations of $C_1,C_2$). Set $\alpha'_i=\overrightarrow{C_i}[y_i^1,y_i^2]$, $\beta'_i=\overrightarrow{C_i}[y_i^2,y_i^3]$, $\gamma'_i=\overrightarrow{C_i}[y_i^3,y_i^1]$, $i=1,2$. By a similar progress as in Claim \ref{Cl011vs122}, we conclude that, up to symmetric, either $(\alpha'_1,\beta'_1,\gamma'_1)\equiv(0,1,1)$, $(\alpha'_2,\beta'_2,\gamma'_2)\equiv(1,2,2)$, or $(\alpha'_1,\beta'_1,\gamma'_1)\equiv(1,2,2)$, $(\alpha'_2,\beta'_2,\gamma'_2)\equiv(0,1,1)$. It follows that for each two paths $S^i,S^j$, either $y_1^i,y_1^j$ are mod-diagonal in $C_1$ or $y_2^i,y_2^j$ are mod-diagonal in $C_2$, but not both. We conclude that exactly one of the statements (1) and (2) holds. 

By Claim \ref{ClQRequiv2}, we have that $|S^1|+|S^2|\equiv|S^1|+|S^3|\equiv|S^2|+|S^3|\equiv 2$. It follows that $|S^1|\equiv|S^2|\equiv|S^3|\equiv 1$. That is, $|Q|\equiv|R|\equiv 1$.
\end{proof}

If $Q,R$ are two disjoint paths from $C_1$ to $C_2$, both disjoint with a third path from $C_1$ to $C_2$, such that their origins are mod-diagonal in $C_1$ and their termini are mod-diagonal in $C_2$, then we call $Q,R$ two \textit{parallel} paths. So Claim \ref{ClQRExactlyOne} says that there are no two parallel paths from $C_1$ to $C_2$.

By Claim \ref{Cl011vs122} and by symmetric, we can assume that $(\alpha_1,\beta_1,\gamma_1)\equiv(0,1,1)$ and $(\alpha_2,\beta_2,\gamma_2)\equiv(1,2,2)$. By Claim \ref{ClQRExactlyOne}, we see that $|P^1|\equiv|P^2|\equiv|P^3|\equiv 1$.

\begin{claim}\label{ClP4}
There exists a path $R$ from $\overrightarrow{C_1}(x_1^1,x_1^2)$ to $\overrightarrow{C_2}(x_2^2,x_2^1)-x_2^3$ with all internal vertices in $V(G)\backslash V(M)$. 
Moreover, assuming that $u_R,v_R$ are the origin and terminus of $R$, 
if $v_R\in V(\overrightarrow{C_2}(x_2^2,x_2^3))$, 
then $|\overrightarrow{C_1}[x_1^1,u_R]|\equiv|\overrightarrow{C_2}[x_2^2,v_R]|\equiv 1$; 
and if $v_R\in V(\overrightarrow{C_2}(x_2^3,x_2^1))$, 
then $|\overrightarrow{C_1}[x_1^1,u_R]|\equiv 2$ and $|\overrightarrow{C_2}[x_2^3,v_R]|\equiv 1$.
\end{claim}

\begin{proof}
By $\alpha_1\equiv 0$, 
there exist at least two vertices in $\overrightarrow{C_1}(x_1^1,x_1^2)$. 
Similarly to Claim \ref{ClEitherC1C2}, 
there exists a path $R$ from $C_1$ to $C_2$ 
with origin $u_R\in\{x_1^{1+},x_1^{1++}\}$ and terminus $v_R\in V(C_2)$ 
such that $R$ is disjoint with $P^i$ unless $v_R=x_2^i$, $i=1,2,3$. 

If $v_R=x_2^3$, 
then both $\{v_R,x_2^1\}$ and $\{v_R,x_2^2\}$ are mod-non-diagonal pairs in $C_2$. Notice that either $\{u_R,x_1^1\}$ or $\{u_R,x_1^2\}$ is a mod-non-diagonal pair in $C_1$, implying that either $R,P^1$ or $R,P^2$ are two clashing paths, a contradiction. Thus we conclude that $v_R\neq x_2^3$, and $R,P^3$ are disjoint. Note that $u_R,x_1^3$ are mod-non-diagonal in $C_1$. If $v_R\in\{x_2^1,x_2^2\}$, then $v_R,x_2^3$ are mod-non-diagonal in $C_2$, and $R,P^3$ are two clashing paths, a contradiction. Now we conclude that $v_R\neq x_2^i$, and $R,P^i$ are disjoint for all $i=1,2,3$. 

Since $u_R,x_1^3$ are mod-non-diagonal in $C_1$, 
we see that $v_R,x_2^3$ are mod-diagonal in $C_1$, i.e., $|\overrightarrow{C_2}[v_R,x_2^3]|\equiv|\overleftarrow{C_2}[v_R,x_2^3]|\equiv 1$. 
If $v_R\in V(\overrightarrow{C_2}(x_2^1,x_2^2))$, 
then both $\{v_R,x_2^1\}$ and $\{v_R,x_2^2\}$ are mod-non-diagonal pairs in $C_2$, and either $R,P^1$ or $R,P^2$ are two clashing paths, a contradiction. So we conclude that $v_R\in V(\overrightarrow{C_2}(x_2^2,x_2^1))\backslash\{x_2^3\}$.

Recall that $|\overrightarrow{C_2}[v_R,x_2^3]|\equiv|\overleftarrow{C_2}[v_R,x_2^3]|\equiv 1$. If $v_R\in V(\overrightarrow{C_2}(x_2^2,x_2^3))$, then $v_R,x_2^1$ are mod-non-diagonal in $C_2$, implying that $u_R,x_1^1$ are mod-diagonal in $C_1$. That is, $|\overrightarrow{C_1}[x_1^1,u_R]|\equiv 1$. If $v_R\in V(\overrightarrow{C_2}(x_2^3,x_2^1))$, then $v_R,x_2^2$ are mod-non-diagonal in $C_2$, implying that $u_R,x_1^2$ are mod-diagonal in $C_1$. That is, $|\overrightarrow{C_1}[u_R,x_1^2]|\equiv 1$ and $|\overrightarrow{C_1}[x_1^1,u_R]|\equiv 2$.
\end{proof}

Let $P^4$ be a path from $\overrightarrow{C_1}(x_1^1,x_1^2)$ to $\overrightarrow{C_2}(x_2^2,x_2^1)-x_2^3$ with all internal vertices in $V(G)\backslash V(M)$, as in Claim \ref{ClP4}, and let $x_1^4,x_2^4$ be the origin and terminus of $P^4$. By symmetric, we can assume that $x_2^4\in V(\overrightarrow{C_2}(x_2^2,x_2^3))$, and $|\overrightarrow{C_1}[x_1^1,x_1^4]|\equiv|\overrightarrow{C_2}[x_2^4,x_2^3]|\equiv 1$. Now we have that $|\overrightarrow{C_1}[x_1^2,x_1^3]|\equiv|\overrightarrow{C_1}[x_1^3,x_1^1]|\equiv|\overrightarrow{C_1}[x_1^1,x_1^4]|\equiv 1$ and $|\overrightarrow{C_1}[x_1^4,x_1^2]|\equiv 2$,
see Figure \ref{fig:P4}. 
We show that in fact the equations hold without the necessity of taking modulo 3. 

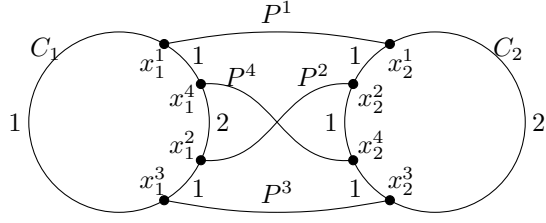
\begin{figure}[ht]
\centering
\begin{tikzpicture}[scale=0.6]
\draw(0,0) circle (2); 
\draw(7,0){coordinate (o2)} circle (2);
\coordinate (x11) at (60:2); \node at (60:1.5) {$x_1^1$};
\coordinate (x14) at (25:2); \node at (20:1.5) {$x_1^4$};
\coordinate (x12) at (-25:2); \node at (-20:1.5) {$x_1^2$};
\coordinate (x13) at (300:2); \node at (300:1.5) {$x_1^3$};
\foreach \x in {1,2,3,4} \draw[fill=black] (x1\x) circle (0.1);
\path(o2)--+(120:2) coordinate (x21); \path(o2)--+(120:1.5){node{$x_2^1$}};
\path(o2)--+(155:2) coordinate (x22); \path(o2)--+(160:1.5){node{$x_2^2$}};
\path(o2)--+(205:2) coordinate (x24); \path(o2)--+(200:1.5){node{$x_2^4$}};
\path(o2)--+(240:2) coordinate (x23); \path(o2)--+(240:1.5){node{$x_2^3$}};
\foreach \x in {1,2,3,4} \draw[fill=black] (x2\x) circle (0.1);
\node at (0:2.3) {2}; \node at (40:2.3) {1}; \node at (180:2.3) {1}; \node at (-40:2.3) {1};
\path(o2)--+(180:2.3){node{1}}; \path(o2)--+(220:2.3){node{1}};
\path(o2)--+(0:2.3){node{2}}; \path(o2)--+(140:2.3){node{1}};
\coordinate (p1) at (3.5,2); \coordinate (p2) at (3.5,0); \coordinate (p3) at (3.5,-2);
\draw(x11) to [out=10, in=180] (p1); \draw(p1) to [out=0, in=170] (x21);
\draw(x12)..controls (3.5,-1) and (3.5,1)..(x22);
\draw(x14)..controls (3.5,1) and (3.5,-1)..(x24);
\draw(x13) to [out=-10, in=180] (p3); \draw(p3) to [out=0, in=-170] (x23); 
\foreach \x in {1,3} \node[above] at (p\x) {$P^\x$};
\node[above] at (2.7,0.6) {$P^4$}; \node[above] at (4.3,0.6) {$P^2$};
\node at (135:2.3) {$C_1$}; \path(o2)--+(45:2.3){node{$C_2$}};
\end{tikzpicture}
\caption{The cycles $C_1,C_2$ and paths $P^1,P^2,P^3,P^4$.}\label{fig:P4}
\end{figure}

\begin{claim}\label{ClLengthC1C2}
$|\overrightarrow{C_1}[x_1^2,x_1^3]|=|\overrightarrow{C_1}[x_1^3,x_1^1]|=|\overrightarrow{C_1}[x_1^1,x_1^4]|=1$ and $|\overrightarrow{C_1}[x_1^4,x_1^2]|=2$; $|\overrightarrow{C_2}[x_2^1,x_2^2]|=|\overrightarrow{C_2}[x_2^2,x_2^4]|=|\overrightarrow{C_2}[x_2^4,x_2^3]|=1$ and $|\overrightarrow{C_2}[x_2^3,x_2^1]|=2$.
\end{claim}

\begin{proof}
Suppose first that $|\overrightarrow{C_1}[x_1^3,x_1^1]|\neq 1$. 
Then $|\overrightarrow{C_1}[x_1^3,x_1^1]|\geqslant 4$ by $|\overrightarrow{C_1}[x_1^3,x_1^1]|\equiv 1$. 
Now there exists a path $Q$ from $C_1$ to $C_2$ with origin $u_Q\in\{x_1^{3+},x_1^{3++}\}$ and termini $v_Q\in V(C_2)$ 
such that $Q$ is disjoint with $P^i$ unless $v_Q=x_2^i$ for $i=1,\ldots,4$. 
If $u_Q=x_1^{3++}$, 
then $u_Q$ is mod-non-diagonal with all of $x_1^i$ in $C_1$. Notice that every vertex in $C_2$ is mod-non-diagonal with at least one of $x_2^i$ in $C_2$. 
It follows that $Q,P^i$ are two clashing paths for some $i=1,\ldots,4$, a contradiction. 
Thus we conclude that $u_Q=x_1^{3+}$. 
It follows that both $u_Q,x_1^1$ and $u_Q,x_1^2$ are mod-non-diagonal in $C_1$. 
If $v_Q\notin\{x_2^1,x_2^2\}$, 
then either $v_Q,x_2^1$ or $v_Q,x_2^2$ are mod-non-diagonal in $C_2$, and either $Q,P^1$ or $Q,P^2$ are two clashing paths, a contradiction. If $v_Q=x_2^2$, then $u_Q,x_1^4$ are mod-diagonal in $C_1$ and $v_Q,x_2^4$ are mod-diagonal in $C_2$, implying that $Q,P^4$ are two parallel paths, contradicting Claim \ref{ClQRExactlyOne}. Thus we conclude that $v_Q=x_2^1$.

By a similar analysis to that presented above, there exists a path $R$ from $C_1$ to $C_2$ with origin $u_R=x_1^{1-}$ and terminus $v_R=x_2^3$, such that $R$ is disjoint with $P^1,P^2,P^4$. Note that $u_Q,u_R$ are mod-non-diagonal in $C_1$ and $v_Q,v_R$ are mod-non-diagonal in $C_2$. If $Q,R$ are disjoint, see Figure \ref{FiQRSx13x11} (i), then $Q,R$ are two clashing paths, a contradiction. So assume that $V(Q)\cap V(R)\neq\emptyset$. It follows that $Q\cup R$ contains a $(u_Q,u_R)$-path $S$ with all internal vertices in $G-(C_1\cup C_2\cup P^2\cup P^4)$, see Figure \ref{FiQRSx13x11} (ii). Now the three paths $\overrightarrow{C_1}[u_Q,u_R]$, $\overleftarrow{C_1}[u_Q,u_R]$, $\overleftarrow{C_1}[u_Q,x_1^2]x_1^2P^2x_2^2\overrightarrow{C_2}[x_2^2,x_2^4]x_2^4P^4x_1^4\overleftarrow{C_1}[x_1^4,u_R]$ are internally-disjoint with $S$, and of lengths $2,\ 0,\ 1\bmod 3$, respectively, a contradiction.

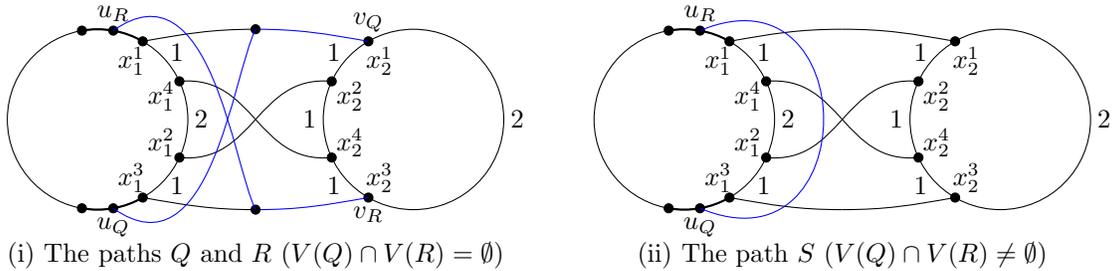
\begin{figure}[ht]
\centering
\begin{tikzpicture}[scale=0.6]

\begin{scope}
\draw(0,0) circle (2); 
\draw(7,0){coordinate (o2)} circle (2);
\coordinate (x11) at (60:2); \node at (60:1.5) {$x_1^1$};
\coordinate (x14) at (25:2); \node at (20:1.5) {$x_1^4$};
\coordinate (x12) at (-25:2); \node at (-20:1.5) {$x_1^2$};
\coordinate (x13) at (300:2); \node at (300:1.5) {$x_1^3$};
\foreach \x in {1,2,3,4} \draw[fill=black] (x1\x) circle (0.1);
\path(o2)--+(120:2) coordinate (x21); \path(o2)--+(120:1.5){node{$x_2^1$}};
\path(o2)--+(155:2) coordinate (x22); \path(o2)--+(160:1.5){node{$x_2^2$}};
\path(o2)--+(205:2) coordinate (x24); \path(o2)--+(200:1.5){node{$x_2^4$}};
\path(o2)--+(240:2) coordinate (x23); \path(o2)--+(240:1.5){node{$x_2^3$}};
\foreach \x in {1,2,3,4} \draw[fill=black] (x2\x) circle (0.1);
\node at (0:2.3) {2}; \node at (40:2.3) {1}; \node at (-40:2.3) {1};
\path(o2)--+(180:2.3){node{1}}; \path(o2)--+(220:2.3){node{1}};
\path(o2)--+(0:2.3){node{2}}; \path(o2)--+(140:2.3){node{1}};
\coordinate (p1) at (3.5,2); \coordinate (p3) at (3.5,-2);
\draw(x11) to [out=10, in=180] (p1); \draw[blue] (p1) to [out=0, in=170] (x21);
\draw(x12)..controls (3.5,-1) and (3.5,1)..(x22);
\draw(x14)..controls (3.5,1) and (3.5,-1)..(x24);
\draw(x13) to [out=-10, in=180] (p3); \draw[blue] (p3) to [out=0, in=-170] (x23); 
\draw[fill=black] (80:2) {coordinate (uR)} {node[above] {$u_R$}} circle (0.1); 
\draw[fill=black] (280:2) {coordinate (uQ)} {node[below] {$u_Q$}} circle (0.1);
\node[above] at (x21) {$v_Q$}; \node[below] at (x23) {$v_R$};
\draw[fill=black] (100:2) circle (0.1); \draw[fill=black] (260:2) circle (0.1);
\draw[thick] (60:2) arc (60:100:2); \draw[thick] (260:2) arc (260:300:2);
\draw[fill=black] (p1) circle (0.1); \draw[fill=black] (p3) circle (0.1);
\draw[blue] (uR)..controls (2.5,3.5) and (3,-1)..(p3);
\draw[blue] (uQ)..controls (2.5,-3.5) and (3,1)..(p1);
\node[below] at (3.5,-2.5) {(i) The paths $Q$ and $R$ ($V(Q)\cap V(R)=\emptyset$)};
\end{scope}

\begin{scope}[xshift=13cm]
\draw(0,0) circle (2); 
\draw(7,0){coordinate (o2)} circle (2);
\coordinate (x11) at (60:2); \node at (60:1.5) {$x_1^1$};
\coordinate (x14) at (25:2); \node at (20:1.5) {$x_1^4$};
\coordinate (x12) at (-25:2); \node at (-20:1.5) {$x_1^2$};
\coordinate (x13) at (300:2); \node at (300:1.5) {$x_1^3$};
\foreach \x in {1,2,3,4} \draw[fill=black] (x1\x) circle (0.1);
\path(o2)--+(120:2) coordinate (x21); \path(o2)--+(120:1.5){node{$x_2^1$}};
\path(o2)--+(155:2) coordinate (x22); \path(o2)--+(160:1.5){node{$x_2^2$}};
\path(o2)--+(205:2) coordinate (x24); \path(o2)--+(200:1.5){node{$x_2^4$}};
\path(o2)--+(240:2) coordinate (x23); \path(o2)--+(240:1.5){node{$x_2^3$}};
\foreach \x in {1,2,3,4} \draw[fill=black] (x2\x) circle (0.1);
\node at (0:2.3) {2}; \node at (40:2.3) {1}; \node at (-40:2.3) {1};
\path(o2)--+(180:2.3){node{1}}; \path(o2)--+(220:2.3){node{1}};
\path(o2)--+(0:2.3){node{2}}; \path(o2)--+(140:2.3){node{1}};
\coordinate (p1) at (3.5,2); \coordinate (p3) at (3.5,-2);
\draw(x11) to [out=10, in=180] (p1); \draw(p1) to [out=0, in=170] (x21);
\draw(x12)..controls (3.5,-1) and (3.5,1)..(x22);
\draw(x14)..controls (3.5,1) and (3.5,-1)..(x24);
\draw(x13) to [out=-10, in=180] (p3); \draw(p3) to [out=0, in=-170] (x23); 
\draw[fill=black] (80:2) {coordinate (uR)} {node[above] {$u_R$}} circle (0.1); 
\draw[fill=black] (280:2) {coordinate (uQ)} {node[below] {$u_Q$}} circle (0.1);
\draw[fill=black] (100:2) circle (0.1); \draw[fill=black] (260:2) circle (0.1);
\draw[thick] (60:2) arc (60:100:2); \draw[thick] (260:2) arc (260:300:2);
\draw[blue] (uR)..controls (4,3.5) and (4,-3.5)..(uQ);
\node[below] at (3.5,-2.5) {(ii) The path $S$ ($V(Q)\cap V(R)\neq\emptyset$)};
\end{scope}

\end{tikzpicture}
\caption{The paths $Q,R$ and $S$ in Claim \ref{ClLengthC1C2} with $|\overrightarrow{C_1}[x_1^3,x_1^1]|\neq 1$.}\label{FiQRSx13x11}
\end{figure}

Suppose second that $|\overrightarrow{C_1}[x_1^1,x_1^4]|\neq 1$. Then $|\overrightarrow{C_1}[x_1^1,x_1^4]|\geqslant 4$. there exists a path $Q$ from $C_1$ to $C_2$ with origin $u_Q\in\{x_1^{4-},x_1^{4--}\}$ and terminus $v_Q\in V(C_2)$ such that $Q$ is disjoint with $P^i$ unless $v_Q=x_2^i$ for $i=1,\ldots,4$. If $u_Q=x_1^{4--}$, then $u_Q$ is mod-non-diagonal with all vertices in $\{x_1^1,x_1^3,x_1^4\}$ in $C_1$. Notice that every vertex in $C_2$ is mod-non-diagonal with at least one vertex in $\{x_2^1,x_2^3,x_2^4\}$ in $C_2$. It follows that $Q,P^i$ are two clashing paths for some $i=1,3,4$, a contradiction. Now we conclude that $u_Q=x_1^{4-}$, and thus both $u_Q,x_1^1$ and $u_Q,x_1^2$ are mod-non-diagonal in $C_1$. If $v_Q\notin\{x_2^1,x_2^2\}$, then either $v_Q,x_2^1$ or $v_Q,x_2^2$ are mod-non-diagonal in $C_2$, and either $Q,P^1$ or $Q,P^2$ are two clashing paths, a contradiction. If $v_Q=x_2^2$, then $Q,P^4$ are two disjoint paths from $C_1$ to $C_2$ such that their origins are mod-diagonal in $C_1$ and their termini are mod-diagonal in $C_2$, implying that $Q,P^4$ are two parallel paths, a contradiction. So we conclude that $v_Q=x_2^1$, see Figure \ref{FiQx11x14x12} (i). Now we have that $|\overrightarrow{C_1}[x_1^3,u_Q]|\equiv|\overrightarrow{C_1}[u_Q,x_1^4]|\equiv 1$. By the same analysis as above (taking $Q$ instead of $P^1$), we can conclude that $|\overrightarrow{C_1}[x_1^3,u_Q]|=1$, a contradiction.

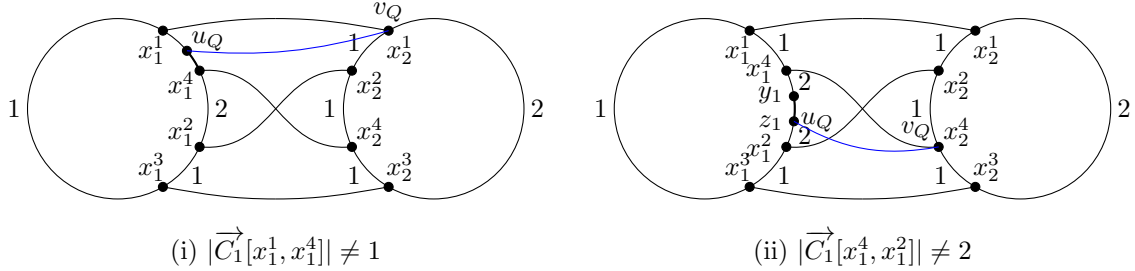
\begin{figure}[ht]
\centering
\begin{tikzpicture}[scale=0.6]

\begin{scope}
\draw(0,0) circle (2); 
\draw(7,0){coordinate (o2)} circle (2);
\coordinate (x11) at (60:2); \node at (60:1.5) {$x_1^1$};
\coordinate (x14) at (25:2); \node at (20:1.5) {$x_1^4$};
\coordinate (x12) at (-25:2); \node at (-20:1.5) {$x_1^2$};
\coordinate (x13) at (300:2); \node at (300:1.5) {$x_1^3$};
\foreach \x in {1,2,3,4} \draw[fill=black] (x1\x) circle (0.1);
\path(o2)--+(120:2) coordinate (x21); \path(o2)--+(120:1.5){node{$x_2^1$}};
\path(o2)--+(155:2) coordinate (x22); \path(o2)--+(160:1.5){node{$x_2^2$}};
\path(o2)--+(205:2) coordinate (x24); \path(o2)--+(200:1.5){node{$x_2^4$}};
\path(o2)--+(240:2) coordinate (x23); \path(o2)--+(240:1.5){node{$x_2^3$}};
\foreach \x in {1,2,3,4} \draw[fill=black] (x2\x) circle (0.1);
\node at (0:2.3) {2}; \node at (180:2.3) {1}; \node at (-40:2.3) {1};
\path(o2)--+(180:2.3){node{1}}; \path(o2)--+(220:2.3){node{1}};
\path(o2)--+(0:2.3){node{2}}; \path(o2)--+(140:2.3){node{1}};
\coordinate (p1) at (3.5,2); \coordinate (p3) at (3.5,-2);
\draw(x11) to [out=10, in=180] (p1); \draw(p1) to [out=0, in=170] (x21);
\draw(x12)..controls (3.5,-1) and (3.5,1)..(x22);
\draw(x14)..controls (3.5,1) and (3.5,-1)..(x24);
\draw(x13) to [out=-10, in=180] (p3); \draw(p3) to [out=0, in=-170] (x23); 
\draw[fill=black] (40:2) {coordinate (uQ)} circle (0.1);
\draw[thick] (x14) arc (25:40:2);
\draw[blue] (uQ) to [bend right=10] (x21);
\node at (2,1.5) {$u_Q$};
\node[above] at (x21) {$v_Q$}; 
\node[below] at (3.5,-2.5) {(i) $|\overrightarrow{C_1}[x_1^1,x_1^4]|\neq 1$};
\end{scope}

\begin{scope}[xshift=13cm]
\draw(0,0) circle (2); 
\draw(7,0){coordinate (o2)} circle (2);
\coordinate (x11) at (60:2); \node at (60:1.5) {$x_1^1$};
\coordinate (x14) at (25:2); \node at (36:1.5) {$x_1^4$};
\coordinate (x12) at (-25:2); \node at (-36:1.5) {$x_1^2$};
\coordinate (x13) at (300:2); \node at (300:1.5) {$x_1^3$};
\foreach \x in {1,2,3,4} \draw[fill=black] (x1\x) circle (0.1);
\path(o2)--+(120:2) coordinate (x21); \path(o2)--+(120:1.5){node{$x_2^1$}};
\path(o2)--+(155:2) coordinate (x22); \path(o2)--+(160:1.5){node{$x_2^2$}};
\path(o2)--+(205:2) coordinate (x24); \path(o2)--+(200:1.5){node{$x_2^4$}};
\path(o2)--+(240:2) coordinate (x23); \path(o2)--+(240:1.5){node{$x_2^3$}};
\foreach \x in {1,2,3,4} \draw[fill=black] (x2\x) circle (0.1);
\node at (40:2.3) {1}; \node at (180:2.3) {1}; \node at (-40:2.3) {1};
\path(o2)--+(180:2.3){node{1}}; \path(o2)--+(220:2.3){node{1}};
\path(o2)--+(0:2.3){node{2}}; \path(o2)--+(140:2.3){node{1}};
\coordinate (p1) at (3.5,2); \coordinate (p3) at (3.5,-2);
\draw(x11) to [out=10, in=180] (p1); \draw(p1) to [out=0, in=170] (x21);
\draw(x12)..controls (3.5,-1) and (3.5,1)..(x22);
\draw(x14)..controls (3.5,1) and (3.5,-1)..(x24);
\draw(x13) to [out=-10, in=180] (p3); \draw(p3) to [out=0, in=-170] (x23); 
\draw[fill=black] (8:2) {coordinate (y1)} circle (0.1); 
\draw[fill=black] (-8:2) {coordinate (z1)} circle (0.1);
\node at (2.5,-0.3) {$u_Q$}; \node at (4.7,-0.5) {$v_Q$};
\node at (15:2.3) {2}; \node at (-15:2.3) {2}; 
\draw[thick] (y1) arc (8:-8:2);
\node at (12:1.5) {$y_1$}; \node at (-12:1.5) {$z_1$}; 
\draw[blue] (z1) to [bend right=20] (x24);
\node[below] at (3.5,-2.5) {(ii) $|\overrightarrow{C_1}[x_1^4,x_1^2]|\neq 2$};
\end{scope}

\end{tikzpicture}
\caption{The path $Q$ in Claim \ref{ClLengthC1C2} with $|\overrightarrow{C_1}[x_1^3,x_1^1]|=1$.}\label{FiQx11x14x12}
\end{figure}

By symmetric we have that $|\overrightarrow{C_1}[x_1^2,x_1^3]|=1$.

Suppose third that $|\overrightarrow{C_1}[x_1^4,x_1^2]|\neq 2$. Then $|\overrightarrow{C_1}[x_1^4,x_1^2]|\geqslant 5$, implying that there are at least 4 vertices in $\overrightarrow{C_1}(x_1^4,x_1^2)$. Let $y_1,z_1$ be two vertices in $V(\overrightarrow{C_1}(x_1^4,x_1^2))$ such that $|\overrightarrow{C_1}[x_1^4,y_1]|\equiv 2$ and $z_1=y_1^+$ (so $|\overrightarrow{C_1}[z_1,x_1^2]|\equiv 2$). Then there exists a path $Q$ from $C_1$ to $C_2$ with origin $u_Q\in\{y_1,z_1\}$ and terminus $v_Q\in V(C_2)$ such that $Q,P^i$ are disjoint unless $v_Q=x_2^i$, $i=1,\ldots,4$. By symmetric we can assume that $u_Q=z_1$. Thus $u_Q$ is mod-non-diagonal with all vertices in $\{x_1^2,x_1^3,x_1^4\}$. If $v_Q\notin\{x_2^3,x_2^4\}$, then $v_Q$ is mod-non-diagonal with either $x_2^3$ or $x_2^4$ in $C_2$; if $v_Q=x_2^3$, then $v_Q$ is mod-non-diagonal with $x_2^2$. For each of the above case $Q,P^i$ are two clashing paths for some $i=2,3,4$. So we assume that $v_Q=x_2^4$, see Figure \ref{FiQx11x14x12} (ii). By the same analysis as above (taking $Q$ instead of $P^4$), we can conclude that $|\overrightarrow{C_1}[x_1^1,u_Q]|=1$, a contradiction.

The second assertion can be proved similarly.
\end{proof}

By Cliam \ref{ClLengthC1C2} we see that $|C_1|=|C_2|=5$. Set $x_1^5=x_1^{4+}$ and $x_2^5=x_2^{3+}$. Then $C_1=x_1^1x_1^4x_1^5x_1^2x_1^3x_1^1$, $C_2=x_2^1x_2^2x_2^4x_2^3x_2^5x_2^1$. By Claim \ref{ClQRExactlyOne}, it holds that $|P^1|\equiv|P^2|\equiv|P^3|\equiv|P^4|\equiv 1$. We now show that all the four paths in fact have lengths exactly 1.

\begin{claim}\label{ClLengthP1P2P3P4}
$|P^1|=|P^2|=|P^3|=|P^4|=1$.
\end{claim}

\begin{proof}
Suppose that $|P^1|\neq 1$. Then $|P^1|\geqslant 4$ since $|P^1|\equiv 1$. Let $H$ be the component of $G-((M\cup P^4)-P^1(x_1^1,x_2^1))$ containing $P^1(x_1^1,x_2^1)$. Notice that $H$ is nontrivial. Since $G$ is essentially 3-connected, $N(H)\backslash\{x_1^1,x_2^1\}\neq\emptyset$. It follows that there exists a path $Q_1$ from $P^1(x_1^1,x_2^1)$ to $(M\cup P^4)-P^1$. Let $u_1$ and $v_1$ be the origin and terminus of $Q_1$. 

Suppose first that $v_1\in V(C_1)\backslash\{x_1^1\}$. Set $Q=P^1[x_2^1,v_1]v_1Q_1$. Now $Q$ is a path from $C_2$ to $C_1$ with origin $x_2^1$ and terminus $v_1$, such that $Q,P^i$ are disjoint unless $v_1=x_1^i$ for $i=2,3,4$. Recall that $|C_i|=5$, $i=1,2$, implying that each two vertices are mod-diagonal in $C_i$ if and only if they are adjacent in $C_i$. It follows that $x_2^1$ is mod-non-diagonal to both $x_2^3$ and $x_2^4$, and every vertex in $V(C_1)\backslash\{x_1^1\}$ is mod-non-diagonal to either $x_1^3$ or $x_1^4$. So we have that either $Q,P^3$ or $Q,P^4$ are two clashing paths, a contradiction. 

Suppose now that $v_1\in(V(C_2)\backslash\{x_1^1\})\cup\bigcup_{i=2}^4(V(P^i)\backslash\{x_1^i\})$. We set
\[
  Q=\left\{\begin{array}{l}
  P^1[x_1^1,u_1]u_1Q_1,\\
  P^1[x_1^1,u_1]u_1Q_1v_1P^2[v_1,x_2^2],\\
  P^1[x_1^1,u_1]u_1Q_1v_1P^3[v_1,x_2^3],\\
  P^1[x_1^1,u_1]u_1Q_1v_1P^4[v_1,x_2^4],
\end{array}\right. u_Q=x_1^1,\mbox{ and }v_Q=\left\{\begin{array}{ll}
  v_1,   & \mbox{if }v_1\in V(C_2)\backslash\{x_2^1\},\\ 
  x_2^2, & \mbox{if }v_1\in V(P^2)\backslash\{x_1^2,x_2^2\},\\
  x_2^3, & \mbox{if }v_1\in V(P^3)\backslash\{x_1^3,x_2^3\},\\
  x_2^4, & \mbox{if }v_1\in V(P^4)\backslash\{x_1^4,x_2^4\}.
\end{array}\right.
\]
Now $Q$ is a path from $C_1$ to $C_2$, such that $Q,P^i$ are disjoint unless $v_Q=x_2^i$ for $i=2,3,4$.

Note that $u_Q,x_1^2$ are mod-non-diagonal in $C_2$. If $v_Q\in\{x_2^3,x_2^5\}$, then $v_Q,x_2^2$ are mod-non-diagonal in $C_1$, and $Q,P^2$ are two clashing paths, a contradiction. If $v_Q=x_2^2$, then $u_Q,x_1^4$ are mod-diagonal in $C_1$ and $v_Q,x_2^4$ are mod-diagonal in $C_2$, implying that $Q,P^4$ are two parallel paths, a contradiction. If $v_Q=x_2^4$, then $u_Q,x_1^3$ are mod-diagonal in $C_1$ and $v_Q,x_2^3$ are mod-diagonal in $C_2$, implying that $Q,P^3$ are two parallel paths, also a contradiction. 

By symmetry, we can show that $|P^2|=|P^3|=|P^4|=1$.
\end{proof}

From Claims \ref{ClLengthC1C2} and \ref{ClLengthP1P2P3P4}, $M\cup P^4$ is the graph obtained from a Petersen graph by removing an edge, see Figure \ref{fig:Cilength5Pilength1}.

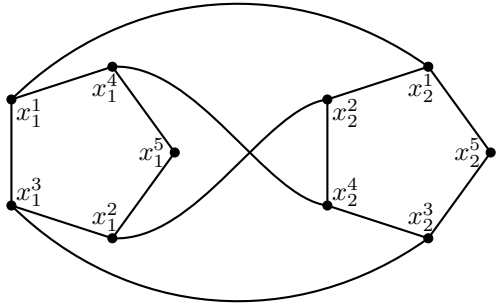
\begin{figure}[ht]
\centering
\begin{tikzpicture}[scale=0.6]
\foreach \x in {1,2,...,5} 
{\draw[fill=black] (\x*72:2) {coordinate (a\x)} circle (0.1);
\path(7,0)--+(\x*72:2) {coordinate (b\x)};
\draw[fill=black] (b\x) circle (0.1);}
\node at (72:1.5) {$x_1^4$}; \node at (144:1.5) {$x_1^1$};
\node at (-72:1.5) {$x_1^2$}; \node at (-144:1.5) {$x_1^3$};
\node at (0:1.5) {$x_1^5$};
\path(7,0)--+(72:1.5) {node {$x_2^1$}}; \path(7,0)--+(144:1.5) {node {$x_2^2$}};
\path(7,0)--+(-72:1.5) {node {$x_2^3$}}; \path(7,0)--+(-144:1.5) {node {$x_2^4$}};
\path(7,0)--+(0:1.5) {node {$x_2^5$}};
\draw[thick] (a1)--(a2)--(a3)--(a4)--(a5)--(a1);
\draw[thick] (b1)--(b2)--(b3)--(b4)--(b5)--(b1);
\draw[thick] (a2) to [bend left=40] (b1); \draw[thick] (a3) to [bend right=40] (b4);
\draw[thick] (a1)..controls (2.5,2) and (4,-1)..(b3);
\draw[thick] (a4)..controls (2.5,-2) and (4,1)..(b2);
\end{tikzpicture}
\caption{$|C_1|=|C_2|=5$ and $|P^1|=|P^2|=|P^3|=|P^4|=1$.}\label{fig:Cilength5Pilength1}
\end{figure}

\begin{claim}\label{ClPathFromC1toC2}
If $Q$ is a path from $C_1$ to $C_2$, then $Q$ is either one of $P^i$, $i=1,\ldots,4$, or $Q=x_1^5x_2^5$.
\end{claim}

\begin{proof}
Let $u_Q,v_Q$ be the origin and terminus of $Q$. By Claim \ref{ClLengthP1P2P3P4}, $Q$ is disjoint with $P^i$ unless $u_Q=x_1^i$ or $v_Q=x_2^i$. 

We first show that $u_Q=x_1^i$ if and only if $v_Q=x_2^i$, $i=1,\ldots,5$. Assume $u_Q=x_1^1$. If $v_Q\in\{x_2^3,x_2^5\}$, then $Q,P^2$ are two clashing paths, a contradiction. If $v_Q=x_2^2$, then $Q,P^4$ are two parallel paths; if $v_Q=x_2^4$, then $Q,P^3$ are two parallel paths; both yield to a contradiction. This shows that if $u_Q=x_1^1$, then $v_Q=x_2^1$. On the other hand, we assume $v_Q=x_2^1$. If $u_Q\neq x_1^1$, then $u_Q$ is mod-non-diagonal with either $x_1^3$ or $x_1^4$, implying that either $Q,P^3$ or $Q,P^4$ are two clashing paths, a contradiction. Thus we conclude that $u_Q=x_1^1$ if and only if $v_Q=x_2^1$. By a similar analysis we can show that $u_Q=x_1^i$ if and only if $v_Q=x_2^i$ for $i=2,3,4$. Furthermore, this implies that $u_Q=x_1^5$ if and only if $v_Q=x_2^5$.

Now let $Q$ be a path from $u_Q=x_1^i$ to $v_Q=x_2^i$, $i=1,\ldots,5$. Through similar analysis to that in Claims \ref{Cl011vs122}-\ref{ClLengthP1P2P3P4}, we have $|Q|=1$, i.e., $Q=x_1^ix_2^i$. It follows that $Q=P^i$ for $i=1,\ldots,4$ and $Q=x_1^5x_2^5$ for $i=5$.
\end{proof}

\begin{claim}\label{ClVG}
$V(G)=V(C_1)\cup V(C_2)$.
\end{claim}

\begin{proof}
Suppose that $H$ is a component of $G-C_1\cup C_2$. By the choice of $C_1,C_2$, $H$ has neighbors in both $C_1$ and $C_2$. It follows that there exists a path $Q$ from $C_1$ to $C_2$ with all internal vertices in $H$ and $|Q|\geqslant 2$, contradicting Claim \ref{ClPathFromC1toC2}.
\end{proof}

By Claims \ref{ClLengthC1C2}-\ref{ClVG}, we get that $G$ is either a Petersen graph or a graph obtained from a Petersen graph by removing an edge. This implies that $e(G)\leqslant 15q+\frac{3}{2}r$, as desired.

\section{Extremal graphs and concluding remarks}\label{section: extremal graphs}

Let $G_1=K_1$. For $n\geqslant 2$, we deﬁne the graph $G_n$ as follows: set
\begin{align*}
  n-1   & =9q+r,\ 0\leqslant r\leqslant 8, \mbox{ and}\\
  r     & =2q'+r',\ 0\leqslant r'\leqslant 1.
\end{align*}
Let $G_n$ be a connected graph consisting of $q$ blocks isomorphic to Petersen graphs, $q'$ blocks isomorphic to triangles and $r'$ blocks isomorphic to $K_2$'s. One can see that $G_n$ contains no $(1\bmod 3)$-cycles and $e(G_n)=15q+\lfloor\frac{3r}{2}\rfloor$. We have
$$ex(n,\mathcal{C}_{1\bmod 3})=15q+\left\lfloor\frac{3r}{2}\right\rfloor, \mbox{ where }n-1=9q+r,\ 0\leqslant r\leqslant 8.$$

Notice that the graphs $G_n$ constructed above are extremal graphs for $ex(n,\mathcal{C}_{1\bmod 3})$. Here we give the proof of Theorem \ref{thm: petersen-blocks}, which states that if $9|(n-1)$, then these are the only extremal graphs for $ex(n,\mathcal{C}_{1\bmod 3})$. We remark that in general, there exist extremal graphs other than $G_n$.

\begin{proof}[\bf Proof of Theorem \ref{thm: petersen-blocks}]
The assertion $e(G)\leqslant \frac{5}{3}(n-1)$ can be deduced immediately by Theorem \ref{thm: main}. Suppose now that $e(G)=\frac{5}{3}(n-1)$. By Theorem \ref{thm: main}, $n-1=9q$ and $e(G)=15q$. We shall show that each block of $G$ is a Petersen graph. Let $G$ be a smallest counterexample.

We claim that $G$ is 2-connected. Suppose otherwise that $G$ is the union of two nontrivial graphs $G_1$, $G_2$ with a common vertex $x$. Set $n_i=|V(G_i)|$, and $n_i-1=9q_i+r_i$, $i=1,2$. Thus $n+1=n_1+n_2$ and either $q=q_1+q_2$, $r_1=r_2=0$ or $q=q_1+q_2+1$, $r_1+r_2=9$. By Theorem \ref{thm: main}, $e(G_i)\leqslant 15q_i+\lfloor\frac{3r_i}{2}\rfloor$. Thus
$$e(G)=15q=e(G_1)+e(G_2)\leqslant 15q_1+\left\lfloor\frac{3r_1}{2}\right\rfloor+15q_2+\left\lfloor\frac{3r_2}{2}\right\rfloor\leqslant 15(q_1+q_2)+\left\lfloor\frac{3}{2}(r_1+r_2)\right\rfloor.$$
This only holds when $r_1=r_2=0$ and $q=q_1+q_2$. It follows that each block of $G_i$, $i=1,2$, is a Petersen graph, and thus each block of $G$ is a Petersen graph, a contradiction. Thus as we claimed, $G$ is 2-connected.

Now we claim that $\delta(G)\geqslant 3$. Suppose there is a vertex $x$ with $d_G(x)\leqslant 2$. Let $G'=G-x$. Then $e(G')=e(G)-d_G(x)\geqslant 15q-2=15(q-1)+13$. By Theorem \ref{thm: main}, $G'$ contains a $(1\bmod 3)$-cycle, a contradiction. Thus as we claimed, $\delta(G)\geqslant 3$. Now by Theorem \ref{thm: dean et al.}, $G$ is a Petersen graph, a contradiction.
\end{proof}

\end{spacing}
\end{document}